\documentclass{amsart}
\usepackage{hyperref}
\usepackage{amsrefs}
\usepackage{amssymb}

\usepackage{tikz}
\usetikzlibrary{matrix,arrows,decorations.pathmorphing, cd}

\newcommand{\Aut}       {\operatorname{Aut}}
\newcommand{\Coset}	{\operatorname{Coset}}
\newcommand{\Dec}       {\operatorname{Dec}}
\newcommand{\Div}       {\operatorname{Div}}
\newcommand{\End}       {\operatorname{End}}

\newcommand{\Fix}       {\operatorname{Fix}}
\newcommand{\Gal}       {\operatorname{Gal}}
\newcommand{\Hom}       {\operatorname{Hom}}
\newcommand{\Ind}       {\operatorname{Ind}}
\newcommand{\Iso}       {\operatorname{Iso}}
\newcommand{\Irr}       {\operatorname{Irr}}

\newcommand{\Map}       {\operatorname{Map}}
\newcommand{\Prim}	{\operatorname{Prim}}
\newcommand{\Rep}	{\operatorname{Rep}}
\newcommand{\Sub}	{\operatorname{Sub}}
\newcommand{\Tor}	{\operatorname{Tor}}
\newcommand{\Trans}	{\operatorname{Trans}}

\newcommand{\ann}       {\operatorname{ann}}

\newcommand{\euler}     {\operatorname{euler}}
\newcommand{\fix}       {\operatorname{fix}}
\newcommand{\img}       {\operatorname{image}}
\newcommand{\soc}       {\operatorname{soc}}

\newcommand{\spf}       {\operatorname{spf}}

\newcommand{\tr}        {\operatorname{tr}}
\newcommand{\res}       {\operatorname{res}}

\newcommand{\F}         {{\mathbb{F}}}
\newcommand{\Fp}        {{\mathbb{F}_p}}
\newcommand{\N}         {{\mathbb{N}}}
\newcommand{\Z}         {{\mathbb{Z}}}
\newcommand{\Zp}        {{\mathbb{Z}_p}}
\newcommand{\Q}         {{\mathbb{Q}}}
\newcommand{\R}         {{\mathbb{R}}}
\newcommand{\C}         {{\mathbb{C}}}

\newcommand{\CA}        {{\mathcal{A}}}
\newcommand{\CB}        {{\mathcal{B}}}
\newcommand{\CC}        {{\mathcal{C}}}
\newcommand{\CD}        {{\mathcal{D}}}

\newcommand{\CF}        {{\mathcal{F}}}
\newcommand{\CG}        {{\mathcal{G}}}
\newcommand{\CH}        {{\mathcal{H}}}
\newcommand{\CK}        {{\mathcal{K}}}
\newcommand{\CL}        {{\mathcal{L}}}

\newcommand{\CN}        {{\mathcal{N}}}
\newcommand{\CO}        {{\mathcal{O}}}
\newcommand{\CP}        {{\mathcal{P}}} 
\newcommand{\CS}        {{\mathcal{S}}}

\newcommand{\CV}        {{\mathcal{V}}}
\newcommand{\CW}        {{\mathcal{W}}}
\newcommand{\CX}        {{\mathcal{X}}}
\newcommand{\CXL}	{{\mathcal{XL}}}

\newcommand{\bG}        {\bar{G}}
\newcommand{\bI}	{\bar{I}}
\newcommand{\bJ}	{\bar{J}}
\newcommand{\bN}        {\bar{N}}
\newcommand{\bQ}	{\bar{Q}}
\newcommand{\bR}        {\bar{R}}
\newcommand{\bT}        {\bar{T}}
\newcommand{\bGW}	{\bar{G}W}
\newcommand{\bNW}	{\bar{N}W}
\newcommand{\bTW}	{\bar{T}W}
\newcommand{\bWF}	{W\bar{F}}

\newcommand{\bh}        {\bar{h}}

\newcommand{\bCL}	{{\overline{\mathcal{L}}}}

\newcommand{\bCV}	{{\overline{\mathcal{V}}}}
\newcommand{\hCV}	{{\widehat{\mathcal{V}}}}
\newcommand{\bCXL}	{{\mathcal{X}\overline{\mathcal{L}}}}

\newcommand{\tE}        {\widetilde{E}}
\newcommand{\tK}	{\widetilde{K}}
\newcommand{\tQ}	{\widetilde{Q}}

\newcommand{\GG}        {{\mathbb{G}}}
\newcommand{\HH}        {{\mathbb{H}}}

\newcommand{\ha}        {\widehat{a}}

\newcommand{\al}        {\alpha}
\newcommand{\bt}        {\beta} 
\newcommand{\gm}        {\gamma}
\newcommand{\dl}        {\delta}
\newcommand{\ep}        {\epsilon}
\newcommand{\zt}        {\zeta}
\newcommand{\tht}       {\theta}
\newcommand{\kp}        {\kappa}
\newcommand{\lm}        {\lambda}
\newcommand{\sg}        {\sigma}

\newcommand{\hxi}       {\widehat{\xi}}
\newcommand{\tphi}      {\widetilde{\phi}}
\newcommand{\tsg}       {\widetilde{\sg}}

\newcommand{\Gm}        {\Gamma}
\newcommand{\Dl}        {\Delta}
\newcommand{\Tht}       {\Theta}
\newcommand{\Sg}        {\Sigma}

\newcommand{\CPi}       {{\mathbb{C}P^\infty}}
\newcommand{\mxi}       {\mathfrak{m}}
\newcommand{\convto}    {\Longrightarrow}
\newcommand{\hot}       {\widehat{\otimes}}
\newcommand{\xra}       {\xrightarrow}
\newcommand{\xla}       {\xleftarrow}
\newcommand{\Sgip}	{\Sigma^\infty_+}
\newcommand{\bF}	{\bar{F}}

\newcommand{\ot}        {\otimes}
\newcommand{\Smash}     {\wedge}
\newcommand{\bigWedge}  {\bigvee}
\newcommand{\psb}[1]    {[\![#1]\!]}
\newcommand{\tm}        {\times}
\newcommand{\st}        {\;|\;}
\newcommand{\sm}        {\setminus}
\newcommand{\sse}       {\subseteq}
\newcommand{\ip}[1]     {\langle #1\rangle}
\newcommand{\bl}        {\bullet}
\newcommand{\ov}[1]     {\overline{#1}}
\newcommand{\aff}       {\mathbb{A}}
\newcommand{\TT}        {{\mathbb{T}}}

\newcommand{\bcf}[2]    {\binom{#1}{#2}}

\newcommand{\pwr}[2]    {#1 {}^\wedge #2}
\newcommand{\pwrb}[2]   {#1 {}^\wedge(#2)}
\newcommand{\pwrbb}[2]  {(#1 {}^\wedge(#2))}

\renewcommand{\:}{\colon}

\newtheorem{theorem}{Theorem}[section]

\newtheorem{lemma}[theorem]{Lemma}
\newtheorem{proposition}[theorem]{Proposition}
\newtheorem{corollary}[theorem]{Corollary}
\theoremstyle{definition}
\newtheorem{convention}[theorem]{Convention}
\newtheorem{remark}[theorem]{Remark}

\newtheorem{definition}[theorem]{Definition}
\newtheorem{example}[theorem]{Example}

\newtheorem{notation}[theorem]{Notation}

\newif\ifshowkeys

\ifshowkeys
\newcommand{\lbl}[1]{\label{#1}\textup{[\texttt{#1}]}\par}
\else
\newcommand{\lbl}{\label}
\fi

\begin{document}
\title{Chromatic (co)homology of finite general linear groups}
\author{S.~M.~A.~Hutchinson}
\author{S.~J.~Marsh}
\author{N.~P.~Strickland}

\begin{abstract}
 We study the Morava $E$-theory (at a prime $p$) of $BGL_d(F)$, where
 $F$ is a finite field with $|F|=1\pmod{p}$.  Taking all $d$ together,
 we obtain a structure with two products $\tm$ and $\bl$.  We
 prove that it is a polynomial ring under $\tm$, and that the module
 of $\tm$-indecomposables inherits a $\bl$-product, and we
 describe the structure of the resulting ring.  In the process, we
 prove many auxiliary structural results.
\end{abstract}


\maketitle 

\section{Introduction}
\lbl{sec-intro}

Let $K$ be the Morava $K$-theory of height $n$ at a prime $p>2$, and
let $E$ be the corresponding Morava $E$-theory.  (Some details of
these theories will be recalled in Section~\ref{sec-morava}.)    

Let $F$ be a finite field with $|F|=q=1\pmod{p}$.  The aim of this
paper is to understand $E^0(BGL_d(F))$ and various related groups, for
all $d$.  In some sense this is already known, by a theorem of
Tanabe~\cite{ta:mkc}.  To explain this, let $\bF$ be an algebraic
closure of $F$, and let $\Gm$ be the associated Galois group, which is
topologically generated by the map $\phi\:a\mapsto a^q$.  One can
then show that $E^0(BGL_d(\bF))$ is a formal power series ring 
$E^0\psb{c_1,\dotsc,c_d}$.  Now put $r_i=\phi^*(c_i)-c_i$.
\begin{theorem}[Tanabe]\lbl{thm-tanabe}
 The elements $r_i$ form a regular sequence in $E^0(BGL_d(\bF))$, and
 \[ E^0(BGL_d(F)) = E^0(BGL_d(\bF))_\Gm = 
     E^0\psb{c_1,\dotsc,c_d}/(r_1,\dotsc,r_d).
 \]
 Moreover, this is a finitely generated free module over $E^0$, and
 $E^1(BGL_d(F))=0$. 
\end{theorem}
Although this is in many ways very satisfactory, it is not easy to
analyse the action of $\phi^*$, or to find a basis for $E^0(BGL_d(F))$
over $E^0$.  Here we will develop some other approaches that will shed
light in these questions.

We will start by explaining the most concrete and computational
consequences of our results.  First, let $r$ be the largest integer
such that $q=1\pmod{p^r}$.  Put $N_0=p^{nr}$, and
\[ N_k = p^{(n-1)k+n(r-1)}(p^n-1) \]
for $k>0$, and $\bN_k=\sum_{i\leq k}N_i$.  Next, we will introduce two
different products and one coproduct on the graded object
$E^0(BGL_*(F))=\{E^0(BGL_d(F))\}_{d\in\N}$.  The first product is just
the ordinary one induced by the diagonal map, and written
$a\ot b\mapsto ab$.  Next, we have evident inclusions
$GL_i(F)\tm GL_j(F)\to GL_{i+j}(F)$.  The associated transfer maps
give a second product, written $a\ot b\mapsto a\tm b$.  The associated
restriction maps also give a coproduct.  Both of these respect the
grading where we put $E^0(BGL_d(F))$ in degree $d$.  We write
$\Ind_*(E^0(BGL_*(F)))$ for the $\tm$-indecomposables, and
$\Prim_*(E^0(BGL_*(F)))$ for the coalgebra primitives.  We also put
\[ X_k = \{c_{p^k}^j\st 0\leq j<N_k\} \sse E^0(BGL_{p^k}(F)). \]

\begin{notation}
 Where necessary to improve readability of exponents, we will write
 $\pwr{a}{k}$ for $a^k$.
\end{notation}

\begin{theorem}\lbl{thm-poly}\leavevmode
 \begin{itemize}
  \item[(a)] $E^0(BGL_*(F))$ is a polynomial ring under the
   $\tm$-product, freely generated by $\coprod_kX_k$.
  \item[(b)] The natural map
   $K^0\ot_{E^0}E^0(BGL_*(F))\to K^0(BGL_*(F))$ is an isomorphism, so
   $K^0(BGL_*(F))$ is also polynomial, with the same generators.
  \item[(c)] The $\tm$-decomposable elements form an ideal under the
   ordinary product, so $\Ind_*(E^0(BGL_*(F)))$ has a natural ring
   structure.  In fact $\Ind_{p^k}(E^0(BGL_*(F)))$ has the form
   $E^0\psb{c_{p^k}}/g_k(c_{p^k})$ for some series $g_k(t)$ of
   Weierstrass degree $N_k$, and this is a complete regular local
   noetherian ring.  On the other hand, $\Ind_d(E^0(BGL_*(F)))=0$ if
   $d$ is not a power of $p$.
  \item[(d)] $\Prim_d(E^0(BGL_d(F)))$ has a natural structure as a
   module over $\Ind_d(E^0(BGL_d(F)))$, and in fact it is free of rank
   one. 
  \item[(e)] $\Prim_{p^k}(K^0(BGL_*(F)))$ is generated by
   $\pwr{c_{p^k}}{\bN_{k-1}}$, and the socle of $K^0(BGL_{p^k}(F))$ is
   generated by $\pwrb{c_{p^k}}{\bN_k\,-1}$.  
 \end{itemize}
\end{theorem}
\begin{proof}
 Most of Claim~(a) is in Proposition~\ref{prop-more-poly-rings},
 except for the identification of the generators, which follows from
 Claim~(c).  Claim~(b) is covered by Corollary~\ref{cor-EV-free}.
 Claim~(c) combines Lemma~\ref{lem-ind-degrees} with
 Proposition~\ref{prop-Q-regular} and Corollary~\ref{cor-al-epi}.  (In
 Section~\ref{sec-indec}, the notation $I$ is used for the primitives,
 and $Q=R/J$ for the indecomposables.)  Claim~(d) is covered by
 Proposition~\ref{prop-I-gen}.  Finally, Claim~(e) is
 Proposition~\ref{prop-socle}.
\end{proof}

Some other features of our work are as follows.
\begin{itemize}
 \item We will find it convenient to consider the groupoid $\CV$ of
  finite-dimensional vector spaces over $F$, and isomorphisms between
  them.  There is an equivalence
  \[ B\CV \simeq \coprod_{d=0}^\infty BGL_d(F), \]
  so by using $\CV$, we consider all possible values of $d$
  simultaneously.  We can use the direct sum and tensor product to
  make $\CV$ into a symmetric bimonoidal category, which gives a rich
  algebraic structure on $E^0(B\CV)$ and related objects.  We will
  also be able to compare $\CV$ in a useful way with various other
  symmetric bimonoidal groupoids.
 \item In particular, we will compare $\CV$ with the groupoid $\CV(k)$
  of finite-dimensional vector spaces over the field
  $F(k)$, which is the unique field extension of $F$ of degree $p^k$
  contained in $\bF$.
 \item For various groupoids $\CG$ we will study the interplay between
  $H^*(B\CG)$, $K^0(B\CG)$, $E^0(B\CG)$, the generalised character
  ring $D^0(\CG)$ of Hopkins-Kuhn-Ravenel, and the formal schemes
  $\spf(K^0(B\CG))$ and $\spf(E^0(B\CG))$.  As is usual in this theory,
  the generalised character rings have an elegant description in terms
  of the discrete group $\Tht=(\Z/p^\infty)^n$.  The challenge is
  to formulate and prove analogous statements about $\spf(E^0(B\CG))$
  in which $\Tht$ is replaced by the formal group scheme $\GG$
  associated to $E$.
 \item We will also use the dual objects $H_*(B\CG)$, $K_0(B\CG)$,
  $E^\vee_0(B\CG)$ and $D_0(B\CG)$, while remembering that
  $K(n)$-local duality theory gives a natural isomorphism
  $K_0(B\CG)\simeq K^0(B\CG)$, and similarly for $E$ and $D$.
\end{itemize}

\begin{remark}\lbl{rem-q-props}
 We have assumed that $|F|=1\pmod{p}$, so in particular the
 characteristic of $F$ is not $p$.  This restriction on the
 characteristic is essential; the problem would be very different and
 much harder if $F$ had characteristic $p$.  On the other hand, the
 restriction $|F|=1\pmod{p}$ is not so essential.  If $F_0$ is a
 finite field with $|F_0|\neq 0\pmod{p}$ then we can let $m$ denote
 the multiplicative order of $|F_0|$ in $\Z/p$ (so $m$ divides $p-1$).
 We can then construct a Galois extension $F/F_0$ with Galois group
 $C_m$.  The results in this paper will determine the Morava
 $K$-theory of the groupoid $\CV$ associated to $F$.  Because $C_m$
 has order coprime to $p$, it is essentially a matter of bookkeeping
 to recover the Morava $K$-theory of the corresponding groupoid
 $\CV_0$ associated to $F_0$.  Details will be given elsewhere.
\end{remark}

\begin{remark}\lbl{rem-p-odd}
 We have also assumed that $p>2$.  We expect that only minor (but
 pervasive) adjustments are needed for $p=2$, possibly including the
 assumption that $|F|=1\pmod{4}$, but we have not checked this.
\end{remark}

This paper contains results from the Ph.D. theses of the
first~\cite{hu:mcf} and second~\cite{ma:met} authors, written under
the supervision of the third author.  The second author's thesis
covered $E^0(BGL_d(F))$ for $d\leq p$; building on this, the first
author obtained results for all $d$.  Neither thesis has previously
been published.

We thank the referee for their careful reading of our work.

\subsection{Outline of the paper}

\begin{itemize}
 \item In Section~\ref{sec-fields} we introduce various extension
  fields of $F$ and their Witt rings.
 \item In Section~\ref{sec-GL} we introduce various matrix groups over
  the rings in Section~\ref{sec-fields}, and study their
  group-theoretic properties.  We also introduce the corresponding
  groupoids. 
 \item In Section~\ref{sec-ordinary} we describe results about the
  ordinary mod $p$ (co)homology of the classifying spaces of the
  groups and groupoids in Section~\ref{sec-GL}.  Some results are
  standard, some are quoted from work of Quillen~\cite{qu:ckt}, but
  some results about restrictions and transfers are new.
 \item In Section~\ref{sec-morava} we recall the spectra $K$ and $E$
  representing Morava $K$-theory and $E$-theory.  We review and extend
  the duality theory which gives a canonical inner product on
  $E^0(BG)$, and we prove various parts of our main theorems that do
  not rely on the Atiyah-Hirzebruch spectral sequence.
 \item In Section~\ref{sec-ann}, we further extend our duality theory
  by proving various results about annihilators and socles.
 \item In Section~\ref{sec-tanabe}, we recall the outline of Tanabe's
  proof of Theorem~\ref{thm-tanabe}, and give an alternative approach
  to part of the argument, which may be of independent interest.  Here
  we start to make serious use of the language of formal schemes.
 \item In Section~\ref{sec-ahss} we analyse the Atiyah-Hirzebruch
  spectral sequence for $K_*(B\CV)$.  We define a trigraded spectral
  sequence by purely algebraic means, and use transfer formulae from
  Section~\ref{sec-ordinary} to produce a morphism from this spectral
  sequence to the topological one.  It is then easy to check that this
  is an isomorphism, and to deduce the precise structure of
  $K_*(B\CV)$. 
 \item In Section~\ref{sec-hkr} we review the generalised character
  theory of Hopkins, Kuhn and Ravenel~\cite{hokura:ggc}, and spell out
  how it works for the groups and groupoids in Section~\ref{sec-GL}.
  This gives further insight into the structure of
  $\Q\otimes E^0(B\CV)$.  
 \item In Section~\ref{sec-indec}, we study the indecomposables for
  the $\tm$-product, and the primitives for the corresponding
  coproduct.  By combining information from Morava $K$-theory and
  rationalised Morava $E$-theory, we are able to understand the
  structure quite precisely.
 \item In Section~\ref{sec-relations}, we prove some additional
  relations in $E^0(B\CV)$.  In particular, we relate certain socle
  generators to a kind of cannibalistic class which has a very natural
  interpretation in generalised character theory.
 \item In Section~\ref{sec-twist}, we set up a theory of twisted
  convolution products, which will be used in the following section.
 \item In Section~\ref{sec-HC} we define the Harish-Chandra product
  on $E^0(B\CV)$, and show that it can be viewed as a twisting of the
  ordinary product.  The twisted and untwisted products give us two
  different rings, but we will prove that they are isomorphic.  We
  also consider whether our product and coproduct on $E^0(B\CV)$
  interact in the right way to give a Hopf algebra structure.  We will
  show that the relevant diagram does not commute, but that it can be
  made commutative by inserting a kind of twisting similar to that
  considered in Section~\ref{sec-twist}.  Unfortunately, the details
  are such that no combination of twistings gives a Hopf algebra.
 \item In Section~\ref{sec-line-bundles} we consider the Morava
  $E$-theory of the groupoid $\CXL$ of finite sets equipped with an
  $F$-linear line bundle.  There is an attractive story relating this
  to power operations in Morava $E$-theory, subgroups of formal
  groups, and our main theorems about the structure of $E^0(B\CV)$.
 \item In Appendix~\ref{apx-ahss}, we revisit the Atiyah-Hirzebruch
  spectral sequence.  Our analysis in Section~\ref{sec-ahss} leaves a
  certain coefficient $t\in\F_p^\times$ undetermined.  The value is
  not really important, but in this appendix we tidy things up by
  proving that $t=1$.
 \item Appendix~\ref{apx-index} is an index of notation.
\end{itemize}

\subsection{Related work}

To place this paper in context, we briefly survey some other
calculations of the Morava $E$-theory (and similar invariants) of
classifying spaces of finite groups.  For finite abelian groups $A$,
the rings $E^*(BA)$ were described in the language of formal group
theory by Hopkins, Kuhn and Ravenel~\cite{hokura:ggc}, using a method
of calculation due to Landweber.  For various purposes it is important
to understand transfer maps between rings of the form $E^*(BA)$; this
has been addressed by Barthel and Stapleton~\cite{bast:ttm}.  Given a
finite nonabelian group $G$, one can try to relate $E^0(BG)$ to the
inverse limit of the rings $E^0(BA)$ as $A$ runs over abelian
subgroups, by analogy with the Artin and Brauer induction theorems in
group representation theory.  The main theorem of~\cite{hokura:ggc} is
a sharp theorem of this type that applies after tensoring with $\Q$;
there are also integral results that are less precise, appearing in
work of Mathew, Naumann and Noel~\cite{manano:dir} and the third
author together with Greenlees~\cite{grst:vlc}.  Complete calculations
of $E^*(BG)$ have been given for a number of specific groups $G$ or
classes of groups of small order or small nilpotence class.  For example,
Schuster has calculated many examples of $2$-groups of order $32$ or
less~\cites{sc:ktg,sc:kca}, and some additional cases were later
calculated by Bakuradze and Jibladze~\cite{baji:mkr}.
Bakuradze~\cite{ba:ccc} has recently calculated much of the structure
of $K^*(BG)$ in cases where $G$ has a subgroup of index $2$ isomorphic
to $C_{2^n}^2$.  The bibliography of that paper refers to quite a
number of other papers by Bakuradze and coauthors containing
calculations for similar classes of $p$-groups (sometimes for $p=2$
and sometimes for $p>2$).  For the symmetric groups, the rings
$E^0(B\Sg_d)$ were calculated by the third author~\cite{st:msg}, and
another proof was given more recently by Schlank and
Stapleton~\cite{scst:tps}.  The latter proof is part of a larger
project to relate Morava $E$-theory rings of different chromatic
heights.  One can regard~\cite{hokura:ggc} as the origin of this
transchromatic theory, and the next steps were taken by
Torii~\cite{to:hpg}, but most of the work has been done by
Stapleton~\cites{st:tcm,st:stg}.  All groups mentioned so far are
``good'', in the sense that $K^*(BG)$ is concentrated in even
degrees.  As well as specific calculations, in~\cite{hokura:ggc} there
are a number of general theorems that can be used to prove that groups
are good.  A much more recent addition to the arsenal is a theorem of
Barthel and Stapleton~\cite{bast:cgg}: if $G$ is good and $A\leq G$ is
a finite abelian $p$-group, then the centraliser $Z_G(A)$ is also
good.  On the other hand, Kriz and Lee~\cite{krle:ode} have shown that
the Sylow $p$-subgroup of $GL_4(\F_p)$ is not good, and their work
suggests that most groups of reasonable complexity are probably not
good. 

\section{Some fields and rings}
\lbl{sec-fields}

\begin{definition}\lbl{defn-F}
 As in the introduction, we will assume that $F$ is a finite field of
 order $q$, with $q=1\pmod{p}$.  We let $q_0$ denote the
 characteristic of $F$, so $q$ is a power of $q_0$.  We put
 $r=v_p(q-1)$, where $v_p$ denotes the $p$-adic valuation, so $r$ is
 the largest natural number such that $q=1\pmod{p^r}$.  By assumption
 we have $r\geq 1$.
\end{definition}

We will repeatedly use the following result.

\begin{lemma}\lbl{lem-vp}
 For $j>0$ we have $v_p(q^j-1)=r+v_p(j)$.
\end{lemma}
\begin{proof}
 In general, suppose that $u\in\Z$ with $v_p(u-1)>0$, say $u=1+p^tw$
 with $t>0$ and $w\neq 0\pmod{p}$.   We then have
 $u^m-1=\sum_{k=1}^m\bcf{m}{k}p^{kt}w^k$  If $m$ is not divisible by
 $p$ then the $k=1$ term has valuation $t$ and the other terms have
 strictly higher valuation so $v_p(u^m-1)=t=v_p(u-1)$.  Suppose
 instead that $m=p$.  The coefficients $\bcf{p}{k}$
 are divisible by $p$ for $0<k<p$, and it follows easily that the
 terms for $k\geq 2$ are divisible by $p^{t+2}$, whereas the term for
 $k=1$ is only divisible by $p^{t+1}$, so $v_p(u^p-1)=t+1=v_p(u-1)+1$.
 For general $j$ we can write $j=p^vm$ with $m\neq 0\pmod{p}$ and use
 the two special cases above to see that $v_p(u^j-1)=t+v$, as
 required.
\end{proof}

\begin{definition}\lbl{defn-F-bar}
 We let $\bF$ denote an algebraic closure of $F$.  We define
 $\phi\:\bF\to\bF$ by $\phi(a)=a^q$, so $F=\{a\st\phi(a)=a\}$.  It is
 standard that the Galois group $\Gm=\Aut_F(\bF)$ is topologically
 generated by $\phi$.  More precisely, there is a homomorphism
 $\Z\to\Gm$ given by $n\mapsto\phi^n$, and this has a canonical
 extension $\widehat{\Z}\to\Gm$ (where $\widehat{Z}$ is the profinite
 completion of $\Z$), and this extension is an isomorphism.  We next
 need to introduce some finite subfields of $\bF$.  We put
 \begin{align*}
  F[m] &= \{a\in\bF\st\phi^m(a)=a\} 
        = \text{ the unique subfield of $\bF$ of degree $m$ over $F$ } \\
  F(k) &= F[p^k] \\
  F(\infty) &= \bigcup_k F(k).
 \end{align*}
\end{definition}

\begin{remark}\lbl{rem-nested-subfields}
 Note that $F[m]\leq F[n]$ if and only if $m$ divides $n$, and $\bF$ is
 the union of all the subfields $F[m]$, or equivalently the union of
 the increasing sequence of subfields $F[k!]$. 

 Note also that $F(\infty)$ is not the same as $\bF$, but this is just
 an annoying technicality.  Most invariants that we study will behave
 the same way for $F(\infty)$ and $\bF$.  We can define
 \[ m_k = \prod\{r^k\st r \text{ is prime and } r \leq p+k\}, \]
 then $F[m_k]\leq F[m_{k+1}]$ and $\bF=\bigcup_kF[m_k]$ and 
 $F[m_k]\cap F(\infty)=F(k)$. 
\end{remark}

\begin{proposition}\lbl{prop-all-units}
 The group $GL_1(F[m])=F[m]^\tm$ is cyclic of order $q^m-1$.  One can
 choose an isomorphism $i\:GL_1(\bF)\to\mu_{q'}(\C)$, where 
 \[ \mu_{q'}(\C) = \{z\in\C\st z^m=1
     \text{ for some } m \text{ with } (m,q)=1\}.
 \]
 Moreover, for any such $i$ we have
 $\mu_{p^\infty}(\C)\sse i(F(\infty)^\tm)$, where
 \[ \mu_{p^\infty}(\C) = \{z\in\C\st z^{p^k}=1
     \text{ for some } k\geq 0\}.
 \]
\end{proposition}
\begin{proof}
 For the first statement, it is standard that the multiplicative group
 of any finite field is cyclic, of order one less than the order of
 the field itself.  

 Next, for any prime $l$ we put 
 \[ \mu_{l^\infty}(\bF) =
      \{z\in\bF\st z^{l^m}=1 \text{ for some } m\}.
 \]
 As $\bF^\tm$ is an abelian torsion group, we see that it is the
 direct sum of its $l$-torsion parts, or in other words the groups
 $\mu_{l^\infty}(\bF)$.  If $l=q_0$, then the map $z\mapsto z^l$ is an
 automorphism so $\mu_{l^\infty}(\bF)=1$.  Let $l$ be a different
 prime.  As $\bF$ is algebraically closed, we can choose an element
 $u_1$ that is a primitive $l$'th root of unity.  We can then
 recursively choose $u_{k+1}$ with $u_{k+1}^l=u_k$.  The powers of
 $u_k$ then give $l^k$ distinct roots of $x^{l^k}-1$, so these are all
 the roots.  It follows that $\mu_{l^\infty}(\bF)$ is generated by all
 the elements $u_k$, and that there is an isomorphism
 $\mu_{l^\infty}(\bF)\to\mu_{l^\infty}(\C)$ sending $u_k$ to
 $\exp(2\pi i/l^k)$.  The claim follows easily from this.

 Finally, Lemma~\ref{lem-vp} shows that the order
 $|F(k)^\tm|=(\pwr{q}{p^k})-1$ is divisible by $p^{r+k}$, so $F(k)^\tm$
 contains a cyclic subgroup of order $p^{r+k}$, and this must be all
 all the $p^{r+k}$'th roots of unity.  It is clear from this that
 $\mu_{p^\infty}(\bF)\leq F(\infty)^\tm$ and
 $\mu_{p^\infty}(\C)\sse i(F(\infty)^\tm)$.
\end{proof}

\begin{definition}\lbl{defn-duals}
 We write $\Z/p^\infty=\Z[1/p]/\Z$, and we silently identify this with
 $\mu_{p^\infty}(\C)$ by $x\mapsto e^{2\pi i x}$ where convenient.
 For any $p$-local abelian group $A$ we put
 $A^*=\Hom(A,\Z/p^\infty)=\Hom(A,\mu_{p^\infty}(\C))$ 
 and $A^\#=\Hom(A,\mu_{p^\infty}(\bF))$.
\end{definition}

\begin{remark}\lbl{rem-duals}
 It is standard that the evident ring map $\Z\to\End(\Z/p^\infty)$
 extends canonically to a ring map $\Z_p\to\End(\Z/p^\infty)$, and
 that this extension is an isomorphism.  We have seen that
 $\mu_{p^\infty}(\bF)$ is isomorphic to $\Z/p^\infty$.  It follows
 that the group $T=\Hom(\Z/p^\infty,\mu_{p^\infty}(\bF))$ is an
 invertible $\Zp$-module, whose inverse is
 $\Hom(\mu_{p^\infty}(\bF),\Z/p^\infty)$.  There are natural
 isomorphisms $T\ot_{\Zp}A^*\to A^\#$ and
 $T^{-1}\ot_{\Zp}A^\#\to A^*$.  Thus, a choice of basis for $T$ gives
 an isomorphism $A^\#\simeq A^*$ that is natural in $A$.
\end{remark}

\begin{proposition}\lbl{prop-F-span}
 If $t\leq r$, then the $F$-linear span of $\mu_{p^t}(\bF)$ is just
 $F$.  If $t>r$, then the $F$-linear span of $\mu_{p^t}(\bF)$ is
 $F(t-r)$.
\end{proposition}
\begin{proof}
 We will write $E(t)$ for the $F$-linear span of $\mu_{p^t}(\bF)$.
 This is the image of a ring map $F[\mu_{p^t}(\bF)]\to\bF$, so it is a
 finite subring of $\bF$.  Every element of $\bF^\tm$ has finite
 multiplicative order, so every subring is a subfield.  If $t\leq r$
 then $p^t$ divides $q-1$ so $\mu_{p^t}(\bF)\leq F^\tm$ and the claim
 is clear.  Suppose instead that $t\geq r$.  Lemma~\ref{lem-vp} tells
 us that $p^t$ divides $|F(t-r)^\tm|=q^{p^{t-r}}-1$, so
 $\mu_{p^t}(\bF)\sse F(t-r)$, so $E(t)\leq F(t-r)$.  Galois theory
 tells us that the fields between $F$ and $F(t-r)$ are precisely
 $\{F(j)\st 0\leq j\leq t-r\}$ and using Lemma~\ref{lem-vp} again we
 see that $\mu_{p^t}(\bF)\not\leq F(t-r-1)$ so $E(t)=F(t-r)$.
\end{proof}

\begin{definition}\lbl{defn-witt}
 We write $W$ for the Witt ring functor, so $WF$ is a complete discrete
 valuation ring with $WF/q_0=F$, and similarly for $WF(m)$ and $W\bF$.
\end{definition}

It is a standard fact that for each $a\in F$ there is a unique element
$\ha\in WF$ with $\ha^q=\ha$ and $\ha=a\pmod{q_0}$.  This is called
the \emph{Teichm\"uller lift} of $a$.  The next two results are
also standard but we include proofs for convenience.

\begin{proposition}\lbl{prop-witt-equiv}
 For any $d\geq 0$ the projections $BGL_d(WF)\to BGL_d(F)$ and
 $BGL_d(W\bF)\to BGL_d(\bF)$ are mod $p$ equivalences.
\end{proposition}
\begin{proof}
 Let $\Gm_m$ be the kernel of the map $GL_d(WF)\to GL_d(WF/q_0^m)$.
 For $m>0$ it is easy to see that $\Gm_m/\Gm_{m+1}$ is an elementary
 abelian $q_0$-group, so its classifying space is $p$-adically
 contractible.  An induction based on this shows that
 $B(\Gm_m/\Gm_{m+k})$ is again $p$-adically contractible, and by
 passing to the limit we see that $B\Gm_1$ is $p$-adically
 contractible.  The fibration $B\Gm_1\to BGL_d(WF)\to BGL_d(F)$ now
 shows that the second map is a $p$-adic equivalence.  The same
 argument works for $\bF$.
\end{proof}

\begin{proposition}\lbl{prop-witt-embedding}
 There exists an injective ring homomorphism $i\:W\bF\to\C$.
 Moreover, any such homomorphism restricts to give an isomorphism
 $\bF^\tm\to\mu_{q'}(\C)$ as in Proposition~\ref{prop-all-units}.
\end{proposition}
\begin{proof}
 It is easy to see that the ring $K=\Q\ot W\bF$ is a field.  We can
 use Zorn's lemma to find a maximal transcendental subset
 $X\subset K$, so $K$ is an algebraic over the rational function field
 $\Q(X)$.  It follows by some cardinal arithmetic that
 $2^{\aleph_0}=|K|=|\Q(X)|=|X|$.  Similarly, we can choose a maximal
 transcendental subset $Y\sse\C$, and we find that $\C$ is an
 algebraic closure of $\Q(Y)$, and $|Y|=2^{\aleph_0}=|X|$.  We can
 thus choose a bijection $X\to Y$, giving an isomorphism
 $\Q(X)\to\Q(Y)$, and extend by Galois theory to get an embedding
 $K\to\C$.  We can restrict to $W\bF\leq K$ to get the required map
 $i$.  We can then compose with the Teichm\"uller map
 $\bF^\tm\to W\bF^\tm$ to get an embedding of $\bF^\tm$ in
 $\mu_{q'}(\C)<\C^\tm$.  For every number $m$ that is coprime to $q$,
 we find that $\bF^\tm$ and $\mu_{q'}(\C)$ both contain precisely $m$
 elements of order dividing $m$.  Using this, we see that the
 embedding $\bF^\tm\to\mu_{q'}(\C)$ is actually an isomorphism.
\end{proof}

We will fix a map $i$ as above for the rest of this document.
However, we will arrange our results and arguments in such a way as to
minimise the dependence on this choice.

\section{General linear groups and groupoids}
\label{sec-GL}

\begin{definition}\lbl{defn-groups}\leavevmode
 \begin{itemize}
  \item[(a)] We put $G_d=GL_d(F)$, and $\bG_d=GL_d(\bF)$.
  \item[(b)] These have diagonal subgroups $T_d=(F^\tm)^d$ and
   $\bT_d=(\bF^\tm)^d$.
  \item[(c)] We identify the symmetric group $\Sg_d$ with the group of
   permutation matrices, so $\Sg_d\leq G_d\leq\bG_d$.
  \item[(d)] Recall that a \emph{monomial matrix} is a matrix with
   precisely one nonzero entry in each row, and precisely one nonzero
   entry in each column.  We write $N_d$ for the group of monomial
   matrices in $G_d$, which is isomorphic to the wreath product
   $GL_1(F)\wr\Sg_d$, and is the normaliser of $T_d$.  Similarly, we
   write $\bN_d$ for the group of monomial matrices in $\bG_d$, which
   is isomorphic to the wreath product $GL_1(\bF)\wr\Sg_d$, and is
   the normaliser of $\bT_d$.
  \item[(e)] We also define groups $GW_d$, $\bGW_d$, $TW_d$, $\bTW_d$,
   $NW_d$ and $\bNW_d$, by replacing $F$ and $\bF$ by the
   corresponding Witt rings $WF$ and $\bWF$.
 \end{itemize}
\end{definition}

We now prove some group-theoretic results whose cohomological
significance will emerge later.

\begin{proposition}\lbl{prop-Gd-order}
 For all $d\geq 0$ we have $|G_d|=\prod_{k=0}^{d-1}(q^d-q^k)$, and
 thus
 \[ v_p|G_d|=dr+v_p(d!)=dr+(d-\al(d))/(p-1), \]
 where $\al(d)$ is the sum of the digits of $d$ in base $p$.
\end{proposition}
\begin{proof}
 Note that for a matrix $g\in G_d$, the $k$'th column can be any
 vector in $F^d$ not lying in the $(k-1)$-dimensional subspace spanned
 by the previous columns.  This gives the claimed (and well-known)
 formula for $|G_d|$.  Recall that $v_p(q)=0$ and
 $v_p(q^j-1)=r+v_p(j)$.  This gives
 \[ v_p|G_d| = \sum_{0\leq i<d}v_p(q^{d-i}-1) 
             = \sum_{j=1}^d(r+v_p(j)) 
             = dr + v_p(d!). 
 \]
 The formula $v_p(d!)=(d-\al(d))/(p-1)$ is also well-known, and is
 easily checked by counting $\{j\leq d\st v_p(j)\geq m\}$ for all $m$.             
\end{proof}

\begin{proposition}\lbl{prop-index}
 The index of $N_d$ in $G_d$ is coprime to $p$.
\end{proposition}
\begin{proof}
 It is clear that $|N_d|=d!(q-1)^d$, so $v_p|N_d|=dr+v_p(d!)=v_p|G_d|$
 as required.
\end{proof}

\begin{proposition}\lbl{prop-non-p-power}
 Suppose that $d$ can be written in base $p$ as $\sum_id_ip^i$, with
 $0\leq d_i<p$.  Then the subgroup $H=\prod_iG_{p^i}^{d_i}\leq G_d$
 has index coprime to $p$.
\end{proposition}
\begin{proof}
 As $\al(p^i)=1$, we have $v_p|G_{p^i}|=p^ir+(p^i-1)/(p-1)$.  This
 gives
 \[ v_p|H|=\sum_id_iv_p|G_{p^i}| =
     \sum_id_ip^ir + \frac{\sum_id_ip^i-\sum_id_i}{p-1} = 
      dr + (d-\al(d))/(p-1) = v_p|G|.
 \]
\end{proof}

\begin{lemma}\lbl{lem-rep-th}
 Let $A$ be a finite abelian $p$-group, let $K$ be a field of
 characteristic $q_0\neq p$, and let $V$ be a finite-dimensional
 $K$-linear representation of $A$.
 \begin{itemize}
  \item[(a)] $V$ can be decomposed as a direct sum of finitely many
   irreducible representations. 
  \item[(b)] If $V$ is irreducible, then the ring $R=\End_A(V)$ is a
   field, and $V$ has dimension one over $R$, and the natural map
   $K[A]\to R$ is surjective.
  \item[(c)] If $V$ is irreducible and $K$ is algebraically closed
   then $V$ has dimension one over $K$.
 \end{itemize}
\end{lemma}
\begin{proof}\leavevmode
 \begin{itemize}
  \item[(a)] This is just Maschke's Theorem, which is valid in this
   context because $|A|$ is invertible in $K$.  In more detail, given
   a subrepresentation $0<U<V$, we can choose a $K$-linear
   retraction $\al_0\:V\to U$, then put
   $\al(v)=|A|^{-1}\sum_{a\in A}a^{-1}\al_0(av)$.  This gives a
   $K[A]$-linear retraction, and thus a splitting $V=U\oplus U'$.  The
   claim follows by iterating this.
  \item[(b)] Schur's Lemma shows that $R$ is a division ring.  Let
   $R_0$ be the image of $K[A]$ in $R$.  This is an integral domain of
   finite dimension over a field, so every multiplication operator is
   an injective endomorphism of a finite-dimensional vector space and
   so is an automorphism.  It follows that $R_0$ is a field.  The
   $R_0$-linear subspaces of $V$ are the same as the
   subrepresentations, so irreducibility means that
   $\dim_{R_0}(V)=1$.  It follows in turn that $R=\End_{R_0}(V)=R_0$.
  \item[(c)] The field $R$ is a finite extension of $K$, and so must
   be equal to $K$.  The claim is clear from this.
 \end{itemize}
\end{proof}

\begin{proposition}\lbl{prop-torus}
 Any finite abelian $p$-subgroup $A\leq \bG_d$ is conjugate to a
 subgroup of $\bT_d$.
\end{proposition}
\begin{proof}
 The inclusion $A\to\bG_d$ allows us to regard $\bF^d$ as an
 $\bF$-linear representation of $A$.  By Lemma~\ref{lem-rep-th}, this
 is isomorphic to a direct sum of one-dimensional representations.
 Any such isomorphism is given by an element of $\bG_d$ that
 conjugates $A$ into $\bT_d$.
\end{proof}

\begin{proposition}\lbl{prop-exponent}
 Let $A$ be any abelian $p$-subgroup of $G_d$, and let $k$ be the
 largest integer such that $p^k\leq d$.  Then $A$ has exponent
 dividing $p^{k+r}$.
\end{proposition}
\begin{proof}
 The inclusion $A\to G_d$ allows us to regard $F^d$ as an
 $F$-representation of $A$, which we call $V$.  We can again split $V$
 as a direct sum of irreducible representations $V_i$, although we can no
 longer guarantee that these are one-dimensional, because $F$ is not
 algebraically closed.  Put
 $R_i=\End_A(V_i)$ as in Lemma~\ref{lem-rep-th}, so $R_i$ is a finite field
 extension of $F$, and $\dim_{R_i}(V_i)=1$.  Put
 $e_i=\dim_F(R_i)=\dim_F(V_i)\leq d$, so $v_p(e_i)\leq k$.  This
 means that $v_p|R_i^\tm|=v_p(q^{e_i}-1)=r+v_p(e_i)\leq r+k$ (using
 Lemma~\ref{lem-vp}).  As $A$ acts faithfully on $V$ via a
 homomorphism to $\prod_iR_i^\tm$, we conclude that the exponent of
 $A$ divides the $p$-part of the exponent of $\prod_iR_i^\tm$, which
 in turn divides $p^{r+k}$. 
\end{proof}

\begin{definition}\lbl{defn-brauer-character}
 We define a class function $\xi_d\:\bG_d\to\C$ as follows.  Given
 $g\in\bG_d$, we let $\lm_1,\dotsc,\lm_d\in\bF^\tm$ denote the
 eigenvalues, repeated with the appropriate multiplicity.  We then use
 our chosen isomorphism $i\:\bF^\tm\to\mu_{q'}(\C)$ (from
 Proposition~\ref{prop-all-units} or~\ref{prop-witt-embedding}) and
 put $\xi_d(g)=\sum_ki(\lm_k)$.  We call this the \emph{universal
  Brauer character}. 
\end{definition}

\begin{theorem}\lbl{thm-brauer}
 For any $m\geq 0$, the restriction of $\xi_d$ to $GL_d(F(m))$ is the
 character of a virtual complex representation.
\end{theorem}
\begin{proof}
 This was proved by Green~\cite{gr:cfg}.  Work of
 Lusztig~\cite{lu:drg} gives a more explicit and constructive proof.
\end{proof}

\begin{remark}\lbl{rem-chern}
 By combining Theorem~\ref{thm-brauer} with the theorem of Tanabe, we
 see that $E^0(BGL_d(F))$ is generated by Chern classes of ordinary
 complex representations.  The Galois automorphism $\phi^*$ is
 essentially the same as the Adams operation $\psi^q$, which is also
 determined by complex representation theory.  In~\cite{st:cag} we
 defined a ``Chern approximation'' map $C(G,E)\to E^0(BG)$ for any
 finite group $G$.  The above observations imply that this map is an
 isomorphism in the case $G=GL_d(F)$.
\end{remark}

Rather than working directly with the above groups, we will find it
convenient to package them as groupoids.  

\begin{definition}\lbl{defn-gpd-misc}
 Given a groupoid $\CG$ and an object $a\in\CG$, we write $\CG(a)$ for
 the automorphism group $\CG(a,a)$.  We also write $\pi_0(\CG)$ for
 the set of isomorphism classes in $\CG$.  We say that a functor
 $\phi\:\CG\to\CH$ is a \emph{$\pi_0$-isomorphism} if the induced map
 $\pi_0(\CG)\to\pi_0(\CH)$ is a bijection.
\end{definition}

\begin{definition}\lbl{defn-gpd-finite}
 Let $\CG$ be a groupoid.  We say that $\CG$ is \emph{hom-finite} if
 $\CG(a,a')$ is finite for all $a,a'\in\CG$ (or equivalently, $\CG(a)$
 is finite for all $a$).  We say that $\CG$ is \emph{finite} if, in
 addition, there are only finitely many isomorphism classes.

 Now let $\phi\:\CG\to\CH$ be a functor between hom-finite groupoids.
 We say that $\phi$ is \emph{finite} if for each $b\in\CH$, the full
 subcategory $\{a\in\CG\st\phi(a)\simeq b\}\sse\CG$ is finite.
\end{definition}

\begin{definition}\lbl{defn-graded-gpd}
 We regard $\N$ as a groupoid with only identity morphisms.  We use
 addition and multiplication to make this into a symmetric bimonoidal
 groupoid.  A \emph{graded groupoid} means a groupoid $\CG$ equipped
 with a functor $\deg\:\CG\to\N$, so $\CG$ splits as a disjoint union
 of groupoids $\CG_d=\deg^{-1}\{d\}$.  We say that $\CG$ is
 \emph{of finite type} if each groupoid $\CG_d$ is finite, or
 equivalently, $\deg$ is a finite morphism.
\end{definition}

\begin{definition}\lbl{defn-cats}\leavevmode
 \begin{itemize}
  \item[(a)] We write $\CL$ for the category of one-dimensional vector
   spaces over $F$, and $\bCL$ for the category of one-dimensional vector
   spaces over $\bF$.  The construction $L\mapsto\bF\ot_FL$ gives a
   faithful functor $\CL\to\bCL$.  Note that $\CL$ is equivalent to
   $GL_1(F)$, considered as a one-object groupoid, and similarly for
   $\bCL$; the functor $\CL\to\bCL$ corresponds to the inclusion
   $GL_1(F)\to GL_1(\bF)$.  We give $\CL$ and $\bCL$ the grading given
   by the constant functor $L\mapsto 1$ (so $\CL$ is of finite type
   but $\bCL$ is not).
  \item[(b)] As mentioned previously, we define $\CV$ to be the
   groupoid of finite-dimensional vector spaces over $F$, and we
   define $\bCV$ to be the corresponding groupoid for $\bF$.  The map
   $V\mapsto\dim_F(V)$ gives a grading on $\CV$, and similarly for
   $\bCV$.  Both of these groupoids are symmetric bimonoidal under
   $\oplus$ and $\otimes$, and the grading is compatible with this.
   We also write $\CV(k)$ for the groupoid of finite-dimensional
   vector spaces over the extension field $F(k)$ of degree $p^k$ over
   $F$.
  \item[(c)] We write $\CX$ for the groupoid of finite sets and
   bijections.  This is symmetric bimonoidal under disjoint union and
   cartesian product.  It is equivalent to the disjoint union of the
   groups $\Sg_d$, considered as one-object groupoids.  There is an
   evident symmetric bimonoidal functor $X\mapsto F[X]$ from $\CX$ to
   $\CV$.  The map $X\mapsto |X|$ gives a grading, compatible with the
   symmetric bimonoidal structure.
  \item[(d)] We write $\CXL$ for the groupoid of pairs $(X,L)$,
   where $X$ is a finite set and $L$ is an $F$-linear line bundle over
   $X$.  This is equivalent to the disjoint union of the groups
   $N_d=GL_1(F)\wr\Sg_d$, considered as one-object groupoids.
   Given $(X_i,L_i)\in\CXL$ for $i=0,1$, we can patch together
   $L_0$ and $L_1$ to get a line bundle $L$ over $X_0\amalg X_1$, and
   we define $(X_0,L_0)\amalg (X_1,L_1)=(X_0\amalg X_1,L)$.
   Similarly, we have a bundle $L'$ over $X_0\tm X_1$ with fibre
   $(L_0)_{x_0}\ot_F(L_1)_{x_1}$ at $(x_0,x_1)$, and we define
   $(X_0,L_0)\tm(X_1,L_1)=(X_0\tm X_1,L')$.  This makes $\CXL$ into
   a symmetric bimonoidal category.  There is a symmetric bimonoidal
   functor $\gm\:\CXL\to\CV$ given by $\gm(X,L)=\bigoplus_{x\in X}L_x$.
   The map $(X,L)\mapsto |X|$ gives a grading, compatible with the
   symmetric bimonoidal structure.
  \item[(e)] We define $\bCXL$ and $\gm\:\bCXL\to\bCV$ in the evident
   analogous way.
 \end{itemize}
\end{definition}

We next want to sharpen Proposition~\ref{prop-exponent} by identifying
the cyclic subgroups of maximal exponent in $G_{p^k}$ or in the
subgroup $N_{p^k}$.  It is convenient to formulate the proof using
groupoids.

\begin{definition}\lbl{defn-field-cat}
 For $k\geq 0$, we define $\CF_k$ to be the category of field
 extensions of $F$ of degree $p^k$.  (It is standard that such
 extensions exist and are all isomorphic, so $\CF_k$ is a connected
 groupoid.)  We also define $\CC_k$ to be the category of triples
 $(C,X,L)$ where $(X,L)\in\CXL$ and $|X|=p^k$ and $C$ is a cyclic
 subgroup of order $p^{k+r}$ in $\Aut(X,L)$.
\end{definition}

\begin{proposition}\lbl{prop-field-cat}
 There are functors $\phi\:\CF_k\to\CC_k$ and $\psi\:\CC_k\to\CF_k$
 which can be described as follows.  We have $\phi(P)=(C,X,L)$ where 
 \begin{align*}
  X &= \{u\in P\st u^{p^k}=1\} \\
  C &= \{v\in P\st v^{p^{k+r}}=1\} \\
  L_u &= \text{span}_F\{v\in C\st v^{p^r}=u\}.
 \end{align*}
 In the other direction, $\psi(C,X,L)$ is the image of $F[C]$ in the
 endomorphism ring of the space $V=\gm(X,L)=\bigoplus_xL_x$.
 Moreover, $\psi\phi(P)$ is naturally isomorphic to $P$ for all $P$,
 and $\phi\psi(C,X,L)$ is unnaturally isomorphic to $(C,X,L)$.
\end{proposition}
\begin{proof}
 Suppose that $P\in\CF_k$, so $|P|=\pwr{q}{p^k}$, so $P^\tm$ is a cyclic
 group of order $(\pwr{q}{p^k})-1$, so $v_p|P^\tm|=r+v_p(p^k)=k+r$ by
 Lemma~\ref{lem-vp}.  Thus, in $\phi(P)=(C,X,L)$, we see that $C$ is
 cyclic of order $p^{k+r}$, and is the unique such subgroup of
 $P^\tm$.  It also follows that $|X|=p^k$.  If $u\in X$ then the
 cyclic property of $C$ means that we can choose $v\in C$ with
 $v^{p^r}=u$.  If $w$ is another element of $C$ with $w^{p^r}=u$ then
 $(w/v)^{p^r}=1$ so $(w/v)^{q-1}=1$ so $w/v\in F$.  This shows that
 $L_u=F.v$, which has dimension one over $F$ as expected, so
 $(X,L)\in\CXL_{p^k}$.  It is also clear that $C$ acts faithfully on
 $(X,L)$ by the rule $c.u=c^{p^r}u$ and $c.v=cv$, so
 $(C,X,L)\in\CC_k$.  This validates the definition of $\phi$.

 In the opposite direction, suppose we start with $(C,X,L)\in\CC_k$.
 The group $C$ then acts on the space $V=\gm(X,L)=\bigoplus_xL_x$, and
 we let $P$ denote the image of $F[C]$ in $\End(V)$.  We claim that
 this lies in $\CF_k$.  To see this, we first note that $C$ is a
 subgroup of $\Aut(X,L)$, so the map $C\to P^\tm$ is injective, so the
 $p$-exponent of $P^\tm$ is at least $k+r$.  On the other hand, $P$ is
 a finite-dimensional commutative $F$-algebra, so it splits as a
 product of factors $P_i$ such that $P_i/\sqrt{0}$ is a field
 extension of $F$, of degree $d_i$ say.  Corresponding to this we get
 a splitting $V=\bigoplus_iV_i$, where $V_i$ is a nontrivial module
 for $P_i$, and so has dimension at least $d_i$ over $F$.  Next,
 because $|F|$ is coprime to $p$, we see that the $p$-exponent of the
 unit group of $P_i$ is the same as for $P_i/\sqrt{0}$, which is
 $r+v_p(d_i)$.  It follows that the $p$-exponent of $P^\tm$ is
 $r+\max_iv_p(d_i)$, so for some $i$ we must have $v_p(d_i)\geq k$ and
 so $d_i\geq p^k$.  For this $i$ we have
 $\dim_F(V_i)\geq d_i\geq p^k=\dim_F(V)$.  It follows that all other
 $V_j$ and $P_j$ are zero, and also that $\sqrt{0}$ must be trivial in
 $P_i$.  This means that $P=P_i\in\CF_k$, as claimed, validating the
 definition of $\psi$.

 If $(C,X,L)=\phi(P)$ then each $L_u$ is a subspace of $P$ and it is
 easy to check that $P=\bigoplus_uL_u$, with $C$ acting by
 multiplication.  From this it is clear that $\psi\phi(P)$ is
 naturally isomorphic to $P$.

 In the opposite direction, suppose we start with $(C,X,L)\in\CC_k$.
 Pick a basepoint $x_0\in X$, and a nonzero element
 $m_0\in L_{x_0}$.   Let $D\leq C$ be the stabiliser of
 $x_0$.  As $C$ injects in $\Aut(X,L)$ we see that $D$ injects in
 $\Aut(L_{x_0})=F$, so $|D|\leq p^r$.  On the other hand, the set
 $C/D$ must inject in $X$ so $|C/D|\leq p^k$.  As $|C|=p^{k+r}$ we
 must actually have $|D|=p^r$ and $|C/D|=p^k$, so $C$ acts
 transitively on $X$.  Now put $X'=\{u\in C\st u^{p^k}=1\}$ and
 $L'_u=\text{span}_F\{v\in C\st v^{p^r}=u\}$.  It is then not hard to
 check that there is a unique $C$-equivariant isomorphism
 $(X',L')\to(X,L)$ sending $1\in X'$ to $x_0\in X$ and $1\in L'_1$ to
 $m_0\in L_{x_0}$.
\end{proof}

\begin{corollary}\lbl{cor-cyclic-subgroup}
 The groupoid $\CC_k$ is connected.  Moreover, if $H$ is a subgroup of
 $G_{p^k}$ of index coprime to $p$, then there is at least one
 $H$-conjugacy class of cyclic subgroups of order $p^{k+r}$ in $H$.
 If $H$ contains the monomial subgroup $N=F^\tm\wr\Sg_{p^k}$, then
 there is precisely one conjugacy class.
\end{corollary}
\begin{proof}
 The first claim follows easily from the proposition.  

 Next, for any subgroup $H\leq G_{p^k}$ of index coprime to $p$, we
 let $\CC'_k(H)$ be the set of cyclic subgroups of order $p^{k+r}$ in
 $H$, and we note that $H$ acts on $\CC'_k(H)$ by conjugation.  We can
 use this to construct a translation groupoid $\Trans(H,\CC'_k(H))$
 with object set $\CC'_k(H)$ and morphism set $H\times\CC'_k(H)$.  The
 claim is now that $\Trans(H,\CC'_k(H))$ is always nonempty, and that
 it is connected when $H\geq N$.

 Now put $X_0=\{0,1,\dotsc,p^k-1\}$ and $L_0=F\times X_0$ so
 $(X_0,L_0)\in\CXL_{p^k}$ and $\Aut(X_0,L_0)=N$ and every object of
 $\CXL_{p^k}$ is (unnaturally) isomorphic to $(X_0,L_0)$.  The
 construction $C\mapsto(C,X_0,L_0)$ gives a functor
 $\Trans(N,\CC'_k(N))\to\CC_k$, which is easily seen to be an
 equivalence.  We have seen that $\CC_k$ is connected, so
 $\Trans(N,\CC'_k(N))$ is also connected, as required.

 We can now choose a group $C_0\in\CC'_k(N)$.  Sylow theory
 tells us that we can also choose a Sylow $p$-subgroup $P_0$ of $N$
 with $C_0\leq P_0$.  Because $N$ has index coprime to $p$, this is
 also a Sylow $p$-subgroup of $G_{p^k}$.

 Now consider an arbitrary subgroup $H\leq G_{p^k}$ of index coprime
 to $p$.  Let $P$ be a Sylow $p$-subgroup of $H$.  Then $P$ is also a
 Sylow $p$-subgroup of $G_{p^k}$, so $P=gP_0g^{-1}$ for some
 $g\in G_{p^k}$.  It follows that $gC_0g^{-1}\in\CC'_k(H)$, so
 $\CC'_k(H)$ is nonempty.

 Now consider a group $H$ with $N\leq H\leq G_{p^k}$, and an element
 $C\in\CC'_k(H)$.  We can again choose a Sylow subgroup $P$ of $H$
 with $C\leq P$.  Because $|G_{p^k}/N|$ is coprime to $p$, we see that
 $P_0$ is also a Sylow subgroup of $H$, so $P_0=hPh^{-1}$ for some
 $h\in H$.  Now put $C'=hCh^{-1}$.  This is an element of the set
 $\CC'_k(N)$, and we have seen that $N$ acts transitively on
 $\CC'_k(N)$, so we can choose $n\in N$ with $C_0=nC'n^{-1}$.  We now
 see that the element $nh$ lies in $H$ and conjugates $C$ to $C_0$.
 This shows that $H$ acts transitively on $\CC'_k(H)$, as required.
\end{proof}

\begin{definition}\lbl{defn-H-coeffs}
 The notation $H^*(X)$ will refer to the cohomology of $X$ with
 coefficients $\Fp$, unless otherwise specified.  Similarly, $H_*(X)$
 will refer to the homology with coefficients $\Fp$.
\end{definition}

\begin{definition}\lbl{defn-determined}
 Let $G$ be a group (which may be infinite).  Given a finite abelian
 $p$-subgroup $A\leq G$, we have a natural map
 $\al_A\:H_*(BA)\to H_*(BG)$.  By taking the sum over all
 possible $A$, we obtain a map
 $\al\:\bigoplus_AH_*(BA)\to H_*(BG)$.  We say that $G$ is
 \emph{$\al$-determined} if this map is surjective.  (Equivalently, the
 dual map $H^*(BG)\to\prod_AH^*(BA)$ should be injective.)
 We also say that a groupoid is $\al$-determined iff every automorphism
 group $G(a)$ is $\al$-determined.
\end{definition}

\begin{proposition}\lbl{prop-determined}\leavevmode
 \begin{itemize}
  \item[(a)] If $G$ is a product of two $\al$-determined subgroups,
   then it is $\al$-determined.
  \item[(b)] If $G$ is a directed union of $\al$-determined subgroups,
   then it is $\al$-determined.
  \item[(c)] If $G$ is an abelian torsion group, then it is
   $\al$-determined.
  \item[(d)] If $G$ has an $\al$-determined subgroup $H$ such that
   $|G/H|$ is finite and coprime to $p$, then $G$ is $\al$-determined.
  \item[(e)] If $G=H\wr\Sg_n$ and $H$ is $\al$-determined, then so is
   $G$. 
 \end{itemize}
\end{proposition}
\begin{proof}
 \begin{itemize}
  \item[(a)] Suppose that $G=G_1\tm G_2$, where each factor $G_i$ is
   $\al$-determined.  Put $M_i=\bigoplus_{A_i}H_*(BA_i)$, where $A_i$
   runs over the finite abelian $p$-subgroups of $G_i$.  By hypothesis
   the maps $\al_i\:M_i\to H_*(BG_i)$ are surjective, so the same is
   true of the map $\al_1\ot\al_2\:M_1\ot M_2\to H_*(BG)$.  Here
   $M_1\ot M_2$ is the direct sum of the groups $H_*(BA)$ as $A$ runs
   over those finite abelian $p$-subgroups of $G$ that have the form
   $A_1\tm A_2$ for some $A_i\leq G_i$.  From this it is clear that $G$
   is $\al$-determined.
  \item[(b)] If $G$ is a directed union of subgroups $G_i$ then it
   is standard that $H_*(BG)$ is the directed union of the subgroups
   $H_*(BG_i)$.  If each $G_i$ is $\al$-determined, it follows
   easily that the same is true of $G$ itself.
  \item[(c)] Using~(b) we can reduce to the case where $G$ itself is a
   finite abelian group.  We can then split $G$ as $G_1\tm G_2$, where
   $G_1$ is a $p$-group and $G_2$ has order coprime to $p$.  We then
   find that $H_*(BG_2)$ is just a single copy of $\Z/p$, so
   $H_*(BG)=H_*(BG_1)$, which proves the claim.
  \item[(d)] If $H$ has index $d\neq 0\pmod{p}$ in $G$, then the
   induced map $H_*(BH)\to H_*(BG)$ is surjective by a transfer
   argument.  It follows that if $H$ is $\al$-determined, then so is $G$.
  \item[(e)] This is the Corollary to~\cite{th:ccc}*{Lemma 9.4}. In
   outline, one can use~(d) to reduce to the case $n=p$, and that case
   can be analysed using a standard spectral sequence
   \[ H^*(\Sg_p;H^*(BH)^{\ot p}) \convto H^*(B(H\wr\Sg_p)). \]
 \end{itemize}
\end{proof}

\begin{proposition}\lbl{prop-determined-examples}
 The groups $G_d$, $\bG_d$, $T_d$, $\bT_d$, $\Sg_d$, $N_d$ and $\bN_d$
 are all $\al$-determined.  Equivalently, the groupoids $\CV$, $\bCV$,
 $\CL^d$, $\bCL^d$, $\CX$, $\CXL$ and $\bCXL$ are all $\al$-determined 
 (where $\CL^d$ and $\bCL^d$ denote cartesian powers of $\CL$ and $\bCL$).
\end{proposition}
\begin{proof}
 The groups $T_d$ and $\bT_d$ are abelian torsion groups, so they are
 $\al$-determined.  The groups $\Sg_d$, $N_d$ and $\bN_d$ are
 therefore $\al$-determined by Proposition~\ref{prop-determined}(e).
 It therefore follows by Propositions~\ref{prop-index}
 and~\ref{prop-determined}(d) that the group $G_d=GL_d(F)$ is
 $\al$-determined.  Here $F$ is an arbitrary finite field of
 characteristic congruent to $1$ mod $p$, so the same logic shows that
 $GL_d(F[k!])$ is $\al$-determined for all $k$.  It therefore follows
 by Proposition~\ref{prop-determined}(b) that $\bG_d$ is
 $\al$-determined.
\end{proof}

We next start to formulate a result that will be useful for comparing
$\CV$ with $\CV(k)$.
\begin{definition}\lbl{defn-galois-twist}
 Consider an object $V\in\bCV$ and an element $\gm\in\Gm=\Gal(\bF/F)$.
 We define $\gm(V)$ to consist of symbols $\gm(v)$ for all $v\in V$,
 subject to the rules $\gm(v+v')=\gm(v)+\gm(v')$ and
 $a.\gm(v)=\gm(\gm^{-1}(a)v)$ (for $v,v'\in V$ and $a\in\bF$).  This
 gives a functor $\gm\:\bCV\to\bCV$.
\end{definition}

\begin{remark}\lbl{rem-untwisted}
 Suppose that $V=\bF\ot_{F(k)}V_0$ for some $V_0\in\CV(k)$, and that
 $\gm\in\ip{\phi^{p^k}}=\Gal(\bF/F(k))$.  One can then check that
 there is an isomorphism $\gm(V)\to V$ given by
 $\gm(a\ot_{F(k)}v_0)\mapsto\gm(a)\ot_{F(k)}v_0$. 
\end{remark}

\begin{lemma}\lbl{lem-twist-split}
 For $V\in\CV(k)$, there is a natural isomorphism
 \[ \xi\:\bF\ot_F V \to \prod_{i=0}^{p^k-1} \phi^i(\bF\ot_{F(k)}V) \]
 given by 
 \[ \xi(a\ot_Fv)_i = a.\phi^i(1\ot_{F(k)}v)
                   = \phi^i(\phi^{-i}(a)\ot_{F(k)}v).
 \]
\end{lemma}
\begin{proof}
 It is straightforward to check that the above formula gives a
 well-defined natural map; we just need to check that it is an
 isomorphism.  As all constructions respect direct sums, it will be
 enough to take $V=F(k)<\bF$.  In that case we can define an
 isomorphism $\phi^i(\bF\ot_{F(k)}V)\to\bF$ by
 $\phi^i(a\ot v)\mapsto\phi^i(av)$.  Using this identification, we get
 a map 
 \[ \xi\:\bF\ot_FF(k) \to \prod_{i=0}^{p^k-1}\bF \]
 given by $\xi(a\ot v)_i=a\,\phi^i(v)$.  It is a well-known fact of
 Galois theory that this is an isomorphism.  (It is an $\bF$-linear
 map between vector spaces of the same finite dimension, and it is
 injective by Dedekind's lemma on independence of automorphisms.)
\end{proof}

\section{Ordinary (co)homology}
\lbl{sec-ordinary}

\begin{notation}
 We remind the reader that all (co)homology groups have coefficients
 $\Fp$ unless otherwise specified.  We write $P[\dotsb]$ and
 $E[\dotsb]$ for polynomials algebras and exterior algebras over
 $\F_p$.
\end{notation}

\begin{definition}\lbl{defn-xH}
 We have chosen an injection $i\:\bG_1=GL_1(\bF)\to S^1$, which gives a map
 $(Bi)^*\:H^*(\CPi)\to H^*(B\bG_1)$.  There is a standard
 generator $x\in H^2(\CPi)$, and we also write $x$ for the
 image of this class in $H^2(B\bG_1)=H^2(B\bCL)$ or in $H^2(BGL_1(F[m]))$.
 After recalling that $\bG_1$ is the union of finite cyclic groups
 $GL_1(F[k!])$, standard calculations give $H^*(B\bG_1)=P[x]$.
 We write $b_k$ for the class in $H_{2k}(B\bG_1)$ that is dual
 to $x^k$.  
\end{definition}

\begin{remark}\lbl{rem-HEK-subscript}
 Later we will define similar classes in Morava $K$-theory and Morava
 $E$-theory.  When it is necessary to distinguish between these, we
 will use notation such as $x_H$, $x_E$, $x_K$, $b_{Hk}$, $b_{Ek}$ and
 $b_{Kk}$.  Similar remarks will apply to the other classes in
 ordinary (co)homology defined later in this section.
\end{remark}

\begin{definition}\lbl{defn-aH}
 Now consider $H^1(BG_1)=H^1(B\CL)$.  Let $m=q-1$ be the order of
 $G_1$, so $m$ is divisible by $p$.  There is a unique generator
 $c\in G_1$ with $i(c)=e^{2\pi i/m}$, and a unique homomorphism
 $G_1\to\Fp$ with $c\mapsto 1$.  There is a standard isomorphism
 $H^1(BG_1)=\Hom(G_1,\Fp)$, so we obtain a class $a\in H^1(BG_1)$.  As
 $G_1$ is cyclic, it is a standard fact that $H^*(BG_1)=P[x]\ot E[a]$
 (where $x$ is the class in $H^2(BG_1)$ restricted from $H^2(\CPi)$ as
 in Definition~\ref{defn-xH}).  We write $e_k$ for the class in
 $H_{2k+1}(BG_1)$ that is dual to $x^ka$.
\end{definition}

\begin{definition}\lbl{defn-H-gens-T}
 Now note that the K\"unneth theorem gives 
 \begin{align*}
  H^*(B\bT^d) &= P[x_{1},\dotsc,x_{d}] \\
  H^*(BT^d)   &= P[x_{1},\dotsc,x_{d}] \ot
                   E[a_{1},\dotsc,a_{d}].
 \end{align*}
 These rings have an evident action of the symmetric group $\Sg_d$.  
 We let $c_{k}$ be the sum of all monomials in the
 $\Sg_d$-orbit of $(-1)^kx_{1}\dotsb x_{k}$, so $|c_{k}|=2k$ and
 \[ \sum_i c_it^{d-i} = \prod_j(t-x_j). \]
 We also let $v_{k}$ be the sum of all monomials in the
 $\Sg_d$-orbit of $(-1)^ka_{1}x_{2}\dotsb x_{k}$.  
\end{definition}

\begin{remark}\lbl{rem-cv-sign}
 Note that in the case $d=1$ we have
 \[ H^*(BT^1)=P[x_1]\ot E[a_1] = P[c_1]\ot E[v_1] \]
 with $c_1=-x_1$ and $v_1=-a_1$.
\end{remark}

We now recall Quillen's calculation of the cohomology of
$BG_d$ and $B\bG_d$.  For ease of comparison with our calculation in
Morava $E$-theory and $K$-theory, we will also recall part of the proof.

\begin{proposition}[Quillen~\cite{qu:ckt}]\lbl{prop-HG}
 The inclusions $T_d\to G_d$ and $\bT_d\to\bG_d$ give isomorphisms
 \begin{align*}
  H^*(B\bG_d) &= H^*(B\bT_d)^{\Sg_d} 
               = P[c_{1},\dotsc,c_{d}] \\
  H^*(BG_d)   &= H^*(BT_d)^{\Sg_d} 
               = P[c_{1},\dotsc,c_{d}] \ot
                 E[v_{1},\dotsc,v_{k}].
 \end{align*}
\end{proposition}
\begin{proof}
 Proposition~\ref{prop-determined-examples} shows that $\bG_d$ is
 $\al$-determined, so the cohomology of $B\bG_d$ is detected on
 abelian $p$-subgroups of $\bG_d$, but Proposition~\ref{prop-torus}
 tells us that every such subgroup is subconjugate to $\bT_d$, so we
 see that the restriction $H^*(B\bG_d)\to H^*(B\bT_d)$ is injective.
 The permutation action on $\bT_d$ comes from inner automorphisms of
 $\bG_d$, and inner automorphisms act as the identity in cohomology,
 so the restriction map lands in $H^*(B\bT_d)^{\Sg_d}$, which is just
 $P[c_{1},\dotsc,c_{d}]$ by the classical theorem of Newton.  The
 same argument shows that $H^*(BG_d)$ maps to $H^*(BT_d)^{\Sg_d}$.
 The structure of this ring of invariants is less well-known, but it
 can be proved in the same way as Newton's theorem, using a
 lexicographic ordering of monomials.

 Next, the maps $\bF\xla{}W\bF\xra{i}\C$ give a diagram
 \begin{center}
  \begin{tikzcd}
   \bT_d = GL_1(\bF)^d
    \arrow[d,rightarrowtail] &
   \bTW_d = GL_1(W\bF)^d
    \arrow[l,twoheadrightarrow]
    \arrow[r,rightarrowtail]
    \arrow[d,rightarrowtail] &
   GL_1(\C)^d
    \arrow[d,rightarrowtail] \\
   \bG_d = GL_d(\bF) &
   \bGW_d = GL_d(W\bF)
    \arrow[l,twoheadrightarrow]
    \arrow[r,rightarrowtail] &
   GL_d(\C).
  \end{tikzcd}
 \end{center}
 We now pass to cohomology rings, taking account of
 Proposition~\ref{prop-witt-equiv} and the discussion associated with 
 Definition~\ref{defn-H-gens-T}.  This gives a diagram as follows:
 \begin{center}
 \begin{tikzcd}
   P[x_{1},\dotsc,x_{d}]^{\Sg_d} \arrow[r,equal] &
   P[x_{1},\dotsc,x_{d}]^{\Sg_d} &
   P[x_{1},\dotsc,x_{d}]^{\Sg_d} \arrow[l,equal] \\
   H^*(B\bG_d) \arrow[u,rightarrowtail] \arrow[r,"\simeq"'] &
   H^*(B\bGW_d) \arrow[u] &
   P[c_{1},\dotsc,c_{d}] \arrow[u,"\simeq"'] \arrow[l]
   \end{tikzcd}
 \end{center}
 It is well-known that the last vertical map is an isomorphism.  By
 chasing the diagram, we see that the same is true of the other two
 vertical maps.  This proves our claim for $\bG_d$.  We use the
 notation $c_{k}$ for the unique element of $H^{2k}(B\bG_d)$ that
 maps to $c_{k}$ in $H^{2k}(B\bT_d)$, and we also use the same
 notation for the restriction of this class to $BG_d$.  For the rest
 of the proof for $BG_d$, we refer to~\cite{qu:ckt}*{Theorem 3}.
\end{proof}

\begin{definition}\lbl{defn-chern-poly}
 Let $P$ be a group, and let $W$ be an $\bF[P]$-module of dimension
 $d<\infty$ over $\bF$.  There is then a corresponding functor
 $P\to\bCV_d$, and thus a map $H^*(B\bCV_d)\to H^*(BP)$.  We define
 $c_k(W)\in H^{2k}(BP)$ to be the image of $c_k\in H^{2k}(B\bCV_d)$
 under this map.  We also put $f_W(t)=\sum_ic_i(W)t^{d-i}$.  We call
 the elements $c_k(W)$ the \emph{Chern classes} of $W$, and we call
 $f_W(t)$ the \emph{Chern polynomial}.
\end{definition}

\begin{remark}\lbl{rem-chern-plus}
 From the definitions one can check that $c_0(V)=1$ and
 $c_k(V\oplus W)=\sum_{k=i+j}c_i(V)c_j(W)$.  Equivalently, $f_W(t)$ is
 always a monic polynomial, and $f_{V\oplus W}(t)=f_V(t)f_W(t)$.  
\end{remark}

Now note that we can use the direct sum operation on $\CV$, or
equivalently the standard inclusions $G_i\tm G_j\to G_{i+j}$, to give
a ring structure on the object $H_*(B\CV)=H_*(BG_*)$.  This is
naturally bigraded, with $H_i(BG_d)$ in bidegree $(i,d)$.  Because
$H^*(BG_d)=((H^*BG_1)^{\ot d})^{\Sg_d}$, we have
$H_*(BG_d)=((H_*BG_1)^{\ot d})_{\Sg_d}$.  We can also do the same
thing with $\bG_d$.  This gives the following:
\begin{proposition}\lbl{prop-BV-homology}
 $H_*(B\bCV)$ is the symmetric algebra generated by $H_*(B\bCL)$, or
 equivalently
 \[ H_*(B\bCV) = P[b_{k}\st k\geq 0]
    \hspace{4em}
    (|b_{k}| = (2k,1)).
 \]
 Similarly, $H_*(B\CV)$ is the free graded-commutative algebra
 generated by $H_*(B\CL)$, or equivalently
 \[ H_*(B\CV) = P[b_{k}\st k\geq 0] \ot E[e_{k}\st k\geq 0]
    \hspace{4em}
    (|b_{k}| = (2k,1),\; |e_{k}| = (2k+1,1)). \qed
 \]
\end{proposition}

\begin{remark}\lbl{rem-intermediate-gens}
 Above we have given various results about $BGL_d(F)$, but $F$ is an
 arbitrary finite field with $|F|=1\pmod{p}$, so the results apply
 equally well to $BGL_d(F(k))$.  This gives generators for
 $H^*(BGL_d(F(k)))$, which we could call $c_{j}^{(k)}$ and
 $v_{j}^{(k)}$.

 There are two different ways to relate the groups $GL_d(F(k))$ as $k$
 varies.  Usually we will just consider the evident inclusion
 $GL_d(F(k-1))\to GL_d(F(k))$, corresponding to the functor
 $\CV(k-1)\to\CV(k)$ given by $V\mapsto F(k)\ot_{F(k-1)}V$.  However,
 we will sometimes also consider the forgetful functor
 $\CV(k)_d\to\CV(k-1)_{pd}$, and the corresponding inclusion
 $GL_d(F(k))\to GL_{pd}(F(k-1))$.  (The latter relies on a choice of
 basis of $F(k)$ over $F(k-1)$, but only up to an inner automorphism,
 which acts as the identity up to homotopy on $BGL_{pd}(F(k-1))$.)

 The restriction maps
 \[ H^*(BGL_d(\bF)) \to H^*(BGL_d(F(k))) \to H^*(BGL_d(F(k-1))) \]
 have $c_{j}\mapsto c_{j}^{(k)}\mapsto c_{j}^{(k-1)}$, so we will
 usually just write $c_{j}$ for these elements.  However, there are
 no elements $v_{j}$ in $H^*(BGL_d(\bF))$, and one can check that the
 restriction map $H^*(BGL_d(F(k))) \to H^*(BGL_d(F(k-1)))$ sends
 $v^{(k)}_{j}$ to zero, so we need to be more careful about the
 notation for these classes.

 Dually, we just write $b_j$ for the polynomial generators of
 $H_*(BGL_*(F(d)))$, but we write $e_{j}^{(k)}$ for the exterior
 generators. 
\end{remark}

It will turn out that for our analysis of the Atiyah-Hirzebruch
spectral sequence, we need to understand the effect in (co)homology of
the inclusions $GL_d(F(k))\to GL_{p^kd}(F)$.  We can reduce by
induction to the case where $d=k=1$.  For this case, we will
temporarily use the following notation:
\[ G = GL_p(F) \hspace{4em} 
   T = GL_1(F)^p \hspace{4em}
   C = GL_1(F(1)).
\]
This means that $C$ is cyclic of order $q^p-1$, and the $p$-torsion
subgroup is cyclic of order $p^{r+1}$).  We let $\rho\:C\to G$ be the
inclusion.  We have seen that
\begin{align*}
  H^*(BG) = H^*(BGL_p(F)) &=
   P[c_1,\dotsc,c_p]\ot E[v_1,\dotsc,v_p] \\
  H^*(BT) = H^*(BGL_1(F)^p) &=
   P[x_1,\dotsc,x_p]\ot E[a_1,\dotsc,a_p] \\
  H^*(BC) = H^*(BGL_1(F(1))) &=
   P[c^{(1)}_1] \ot E[v^{(1)}_1].                               
\end{align*}
For typographical convenience, we put $\ov{c}=c^{(1)}_1$ and
$\ov{v}=v^{(1)}_1$.  Note that there are also natural generators
$\ov{x}=-\ov{c}$ and $\ov{a}=-\ov{v}$.

It is not hard to understand the effect of $\rho$ on the polynomial
generators: 
\begin{lemma}\lbl{lem-res-c}
 We have $\rho^*(c_p)=\ov{c}^p=(-\ov{x})^p$, and $\rho^*(c_i)=0$ for
 $0<i<p$. 
\end{lemma}
\begin{proof}
 Let $V$ denote $F(1)$, regarded as a $p$-dimensional $F$-linear
 representation of $C$ by multiplication.  The inclusions
 $C\to GL_p(F)\to GL_p(\bF)$ give a $p$-dimensional $\bF$-linear
 representation of $C$, which we call $\ov{V}$.  Essentially by
 construction, we have $\ov{V}=\bF\ot_FF(1)$.  This means that the
 polynomial $f(t)=\sum_{i=0}^p\rho^*(c_i)t^{p-i}$ is the Chern
 polynomial for $\ov{V}$.  Also, the inclusion $F(1)\to\bF$
 gives a one-dimensional $\bF$-linear representation of $C$,
 which we call $L$.  The Euler class of $L$ is $\ov{x}$.  Our claim is
 that $f(t)=t^p-\ov{x}^p$.   

 Let $\phi$ be the Frobenius automorphism $x\mapsto x^q$ on $F(1)$, so
 $\text{Gal}(F(1)/F)=\{\phi^i\st 0\leq i<p\}$.  Using
 Lemma~\ref{lem-twist-split}, we see that
 $\ov{V}\simeq\bigoplus_j(\phi^j)^*(L)$, and thus that
 $f(t)=\prod_{j=0}^{p-1}(t-q^j\ov{x})$.  As $q=1\pmod{p}$ this gives
 $f(t)=(t-\ov{x})^p=t^p-\ov{x}^p$ as claimed.
\end{proof}

To compute $\rho^*(v_i)$, we will need to describe $v_i$ as a
transfer, and we will need some general results about transfers.

\begin{lemma}\lbl{lem-abelian-tr}
 Let $A$ be a finite abelian group, and let $B$ be a subgroup of index
 $d$.  We can then define $t\:A\to B$ by $t(a)=da$, and this induces
 $t^*\:\Hom(B,\Fp)\to\Hom(A,\Fp)$, and $t^*$ is the same as
 $\tr_B^A\:H^1(BB)\to H^1(BA)$.
\end{lemma}
\begin{proof}
 Classical concrete formulae for the transfer in homology are given
 in~\cite{th:ccc}*{Chapter 2}, for example.  The claim follows by
 specialising to the abelian case, and then dualising.
\end{proof}

\begin{corollary}\lbl{cor-tr-zero}
 The transfer is zero if $dA\leq pB$.  In particular, this holds if
 \begin{itemize}
  \item[(a)] $B$ is a summand in $A$, and $p$ divides $d$; or
  \item[(b)] The $p$-torsion part of $A/B$ is not cyclic. 
 \end{itemize}
\end{corollary}
\begin{proof}
 The first claim is clear from the lemma.  In case~(a), we have
 $A=B\oplus C$ for some $C$ with $|C|=d$, so
 $dA=dB\oplus dC=dB\oplus 0\leq pB$.  In case~(b), we can write
 $d=p^k\,d_1$ with $d_1\neq 0\pmod{p}$.  We then have
 $A/B=C\oplus D$ say, where $|C|=p^k$ and $|D|=d_1$ and $C$ is not
 cyclic, so $k>1$ and $C$ is annihilated by $p^{k-1}$.  It follows
 that $A/B$ is annihilated by $d/p=p^{k-1}d_1$, so $(d/p).A\leq B$, so
 $dA\leq pB$. 
\end{proof}

\begin{lemma}\lbl{lem-vc-as-tr}
 In $H^*(BG)=P[c_1,\dotsc,c_p]\ot E[v_1,\dotsc,v_p]$ we have
 \begin{align*}
   v_i &= \tr_T^G(a_1x_2\dotsb x_i)
             \text{ for } 1\leq i\leq p \\
   c_i &= i\tr_T^G(x_1x_2\dotsb x_i)
             \text{ for } 1\leq i<p.
 \end{align*}
\end{lemma}
\begin{proof}
 We will just discuss $v_i$; the argument for $c_i$ is essentially the
 same.  Put $v'_i=a_1x_2\dotsb x_i$, so the claim is that
 $v_i=\tr_T^G(v'_i)$.  Let $H_i$ be the stabiliser of $v'_i$ in the
 symmetric group $\Sg_p$, so $H_i\simeq\Sg_{i-1}\tm\Sg_{p-i}$.  By
 definition, $v_i$ is the element that satisfies
 \[ \res^G_T(v_i)=(-1)^i\sum_{\sg\in\Sg_p/H_i}\sg^*(v'_i). \]
 As $|H_i|$ is not divisible by $p$, we can rewrite this as 
 \[ \res^G_T(v_i)=(-1)^i|H_i|^{-1}\sum_{\sg\in\Sg_p}\sg^*(v'_i). \]
 Moreover, $|H_i|=(i-1)!(p-i)!$.  Wilson's theorem tells us that
 $|H_1|=-1\pmod{p}$, and it is not hard to deduce by induction that
 $|H_i|=(-1)^i\pmod{p}$.  We thus have
 \[ \res^G_T(v_i) = \sum_{\sg\in\Sg_p}\sg^*(v'_i). \]
 The theorem of Quillen tells us that the restriction map is
 injective, so it will suffice to show that we also have
 $\res^G_T(\tr_T^G(v'_i))=\sum_\sg\sg^*(v'_i)$. 
 
 Next, there is a Mackey formula expressing $\res^G_T(\tr_T^G(v'_i))$
 as a sum of terms indexed by double cosets $TgT$.  The normaliser of
 $T$ in $G$ is $\Sg_p\ltimes T$, so we have a double coset for each
 $\sg\in\Sg_p$, and the corresponding term is just $\sg^*(v'_i)$.
 Thus, it will suffice to show that when $g$ is not in the normaliser
 of $T$, the term for $TgT$ in the double coset formula is zero.  This
 term involves the group $U=T\cap gTg^{-1}$, which we can describe as
 follows.  As $g$ is not in the normaliser, we see that it is not a
 monomial matrix, so we see that there is at least one triple
 $(i,j,k)$ where $g_{ik}\neq 0\neq g_{jk}$ and $i\neq j$.  Let $E$ be
 the smallest equivalence relation on $\{1,\dotsc,p\}$ such that $iEj$
 for all such triples $(i,j,k)$.  For $u\in(F^\tm)^p$, let $\dl(u)$
 denote the corresponding diagonal matrix in $T$.  It is not hard to
 check that $\dl(u)\in U$ iff $u_i=u_j$ whenever $iEj$.  This means
 that $U\simeq(F^\tm)^r$ for some $r<p$, and that $U$ is a retract of
 $T$, with index divisible by $p$.  This means that the transfer map
 $\tr_U^T\:H^*(BU)\to H^*(BT)$ is zero (by
 Corollary~\ref{cor-tr-zero}), and thus that $TgT$ does not contribute
 to the double coset formula.
\end{proof}

\begin{lemma}\lbl{lem-res-v}
 $\rho^*(v_i)=-\ov{a}\,\ov{x}^{i-1}=(-1)^{i-1}\ov{c}^{i-1}\ov{v}$ for
 all $i$. 
\end{lemma}
\begin{proof}
 We again put $v'_i=a_1x_2\dotsb x_i$, so $v_i=\tr_T^G(v'_i)$ We will
 use the double coset formula to evaluate
 $\rho^*(v_i)=\res^G_C(\tr_T^G(v'_i))$.  This will involve the group
 $Z$ of multiples of the identity in $GL_p(F)$, so $Z\simeq F^\tm$ and
 $Z$ is the centre of $G$ and $Z=C\cap T$.  The map $\rho\:C\to G$ was
 defined using a basis $e_1,\dotsc,e_p$ for $F(1)$ over $F$, which can
 be chosen to have $e_1=1$.  Let $X\subset G$ be the set of matrices
 $g$ such that $ge_1=e_1$, and the first nonzero entry in $ge_i$ is
 one for all $i$.  We claim that $X$ contains precisely one element
 from each double coset $ChT$.  Indeed, given $h\in G$ we note that
 $he_1$ is a nonzero element of $F^p\simeq F(1)$, so we can regard it
 as an element $c\in F(1)^\tm=C$.  Now let $u_i$ be the first nonzero
 entry in the vector $c^{-1}he_i$, and let $t$ be the diagonal matrix
 with entries $u_1,\dotsc,u_p$, so $t\in T$.  We then find that the
 element $g=c^{-1}ht^{-1}$ lies in $X$, and that $X\cap(ChT)=\{g\}$.
 Next, suppose that $g\in X$, and consider the group
 $H_g=C\cap gTg^{-1}$.  If $h\in H_g$ then $h.(ge_1)=c.(ge_1)$ for
 some $c\in C=F(1)^\tm$, but also $h.(ge_1)=u_1.(ge_1)$ for some
 $u_1\in F^\tm$.  It follows that $c=u_1\in F^\tm$, and thus that
 $h\in Z$.  From this we see that $H_g$ is just equal to $Z$, and thus
 that every conjugation map acts as the identity on $H_g$.  Because of
 this, the double coset formula just simplifies to
 $\res^G_C(\tr_T^G(w))=|X|\tr_Z^C(\res^T_Z(w))$ for all
 $w\in H^*(BT)$.  Standard methods show that
 \[ |X| = \prod_{j=1}^{p-1}\frac{q^p-q^j}{q-1} =
      q^{p(p-1)/2}\prod_{j=1}^{p-1}\sum_{m=0}^{p-j-1}q^m.
 \]
 As $q=1\pmod{p}$, this gives $|X|=(p-1)!=-1\pmod{p}$.  Next, we
 have $H^*(BZ)=P[x]\ot E[a]$, where $x$ is the image of any of the
 elements $x_i\in H^2(BT)$, and $a$ is the image of any of the
 elements $a_i\in H^1(BT)$.  This gives $\res^T_Z(v'_i)=ax^{i-1}$
 and so $\res^G_C(v_i)=-\tr_Z^C(ax^{i-1})$.  Now $x$ is the Euler class of
 the representation given by the inclusion $F^\tm\to\bF^\tm$, whereas
 $\ov{x}$ is the Euler class of the inclusion $F(1)^\tm\to\bF^\tm$, so
 $\res^C_Z(\ov{x})=x$.  By the reciprocity formula for transfers, we
 therefore have $\tr_Z^C(ax^{i-1})=\tr_Z^C(a)\ov{x}^{i-1}$.  Next,
 recall that we chose $i\:\bF^\tm\to\C^\tm$, and there is a unique
 $\al\in F^\tm=Z$ with $i(\al)=\exp(2\pi i/(q-1))$, and that $a$ can
 be identified with the map $Z\to\Fp$ sending $\al$ to $1$.
 Similarly, there is a unique $\ov{\al}\in C$ with
 $i(\ov{\al})=\exp(2\pi i/(q^p-1))$, and $\ov{a}$ can be identified
 with the map $C\to\Fp$ sending $\ov{\al}$ to $1$.  We find that
 $\ov{\al}^{|C/Z|}=\al$, so Lemma~\ref{lem-abelian-tr} gives
 $\tr_Z^C(a)=\ov{a}$.  Putting this together, we get
 \[ \res^G_C(v_i) = -\ov{a}\,\ov{x}^{i-1}
     = (-1)^{i-1}\ov{c}^{i-1}\ov{v}
 \]
 as claimed.
\end{proof}

We can now dualise the above results to describe the map
$H_*(B\CV(1))\to H_*(B\CV)$, and extend inductively to describe the
map $H_*(B\CV(n+m))\to H_*(B\CV(n))$ for any $n,m\geq 0$.  We will
call this map $\rho_*$.  Recall here that
 \[ H_*(B\CV(n)) = P[b^{(n)}_i\st i\geq 0] \ot 
                   E[e^{(n)}_i\st i\geq 0].
\]
To describe $\rho_*$, it is convenient to introduce 
formal power series as follows:
\begin{definition}\lbl{defn-be-series} 
 We put
 \begin{align*}
  b^{(n)}(s) &= \sum_i b^{(n)}_is^i \in H_*(B\CV(n))\psb{s} \\
  e^{(n)}(s) &= \sum_i e^{(n)}_is^i \in H_*(B\CV(n))\psb{s}.
 \end{align*} 
\end{definition}
The map $\rho_*$ induces a map 
\[ H_*(B\CV(n+m))\psb{s} \to H_*(B\CV(n))\psb{s}, \]
which we again call $\rho_*$.

\begin{proposition}\lbl{prop-rho-lower-star}
 We have 
 \[ \rho_*(b^{(n+m)}(s)) = 
     (b^{(n)}(s))^{p^m} = 
      \sum_j (b_j^{(n)})^{p^m}s^{p^mj},
 \]
 so 
 \[ \rho_*(b^{(n+m)}_i) = 
     \begin{cases}
      (b^{(n)}_{i/p^m})^{p^m} & \text{ if } p^m\mid i \\
      0 & \text{ otherwise. }
     \end{cases}
 \]
 We also have 
 \[ \rho_*(e^{(n+m)}(s)) = 
     (b^{(n)}(s))^{p^m-1} e^{(n)}(s).
 \]
\end{proposition}
\begin{proof}
 We can reduce inductively to the case where $n=0$ and $m=1$.  In this
 case we will write $b(s)=b^{(0)}(s)$ and $\ov{b}(s)=b^{(1)}(s)$, and
 similarly for $e(s)$ and $\ov{e}(s)$.  The element
 $\rho_*(\ov{b}(s))$ is characterised by the fact that
 \[ \ip{\rho_*(\ov{b}(s)),c^\al}=\ip{\ov{b}(s),\rho^*(c^\al)} \]
 for all monomials $c^\al=\prod_ic_i^{\al_i}$ in $c_1,\dotsc,c_p$.
 By Lemma~\ref{lem-res-c}, the right hand side is zero unless
 $c^{\al}=c_p^m$ for some $m$.  In that special case, we have
 $\rho^*(c^\al)=\ov{c}^{pm}=(-\ov{x})^{pm}$.  We also have
 $\ip{\ov{b}_i,\ov{x}^j}=\dl_{ij}$, so we get
 $\ip{\ov{b}(s),\rho^*(c_p^m)}=(-s)^{pm}$.  Now consider
 $\ip{b(s)^p,c^\al}$.  Note that the inclusion
 $T=GL_1(F)^p\to GL_p(F)=G$ induces a map
 $\sg_*\:H_*(BGL_1(F))^{\ot p}\to H_*(BGL_p(F))$, and it is this map
 that is used in defining the power $b(s)^p$.  
 We therefore find that $\ip{b(s)^p,u}=\ip{b(s)^{\ot p},\sg^*(u)}$ for
 all $u\in H^*(BG)$.  On the other hand, for any
 $w\in P[x_1,\dotsc,x_p]$, the inner product $\ip{b(s)^{\ot p},w}$ is
 just the result of replacing each $x_i$ by $s$.  It follows that the
 map $w\mapsto\ip{b(s)^{\ot p},w}$ is a ring homomorphism.  It
 sends $c_i$ to $(-1)^i$ times the $i$'th elementary symmetric
 function in the list $s,\dotsc,s$, which is zero mod $p$ if $0<i<p$,
 and is $(-s)^p$ if $i=p$.   This gives
 $\ip{b(s)^p,c^\al}=\ip{\ov{b}(s),\rho^*(c^\al)}$ for all $\al$, so
 $\rho_*(\ov{b}(s))=b(s)^p$ as claimed.

 We now consider $\rho_*(\ov{e}(s))$.  Put
 \[ e^*(s) = b(s)^{\ot(p-1)}\ot e(s) \in H_*(BT), \]
 so the claim is that $\rho_*(\ov{e}(s))=\sg_*(e^*(s))$, or
 equivalently that 
 \[ \ip{e^*(s),\sg^*(c^\al v^\ep)} =
     \ip{\ov{e}(s),\rho^*(c^\al v^\ep)}
 \]
 for all monomials $c^\al v^\ep\in H^*(BG)$.  Here $v^\ep$ is in
 principle any product of distinct terms of the form $v_k$, but it is
 easy to see that both sides are zero unless there is precisely one
 term.  We must therefore check that
 \[ \ip{e^*(s),\sg^*(c^\al v_k)} =
     \ip{\ov{e}(s),\rho^*(c^\al v_k)}
 \]
 for all $\al$ and $k$.  Next, one can check that the map
 \[ \ip{e^*(s),-} \: H^*(BT) =
      P[x_1,\dotsc,x_p]\ot E[a_1,\dotsc,a_p] \to P[s]
 \]
 can be described as follows.  Given $w\in H^*(BT)$ we first replace
 $x_1,\dotsc,x_p$ by $s$ and $a_1,\dotsc,a_{p-1}$ by $0$, to get an
 element $\xi(w)\in P[s]\ot E[a_p]$; then $\ip{e^*(s),w}$ is the
 coefficient of $a_p$ in $\xi(w)$.  Here $\xi$ is a ring map with
 $\xi(\sg^*(c_i))=0$ for $0<i<p$, and $\xi(\sg^*(c_p))=(-s)^p$, as we
 saw while considering $\rho_*(\ov{b}(s))$.  We also find that
 \[ \xi(\sg^*(v_k)) = (-1)^k \bcf{p-1}{k-1} s^{k-1}a_p. \]
 (The factor of $(-1)^k$ is incorporated in the definition of $v_k$,
 and the binomial coefficient comes from the combinatorics of the
 $\Sg_p$-action.) It is an exercise to check that
 $(-1)^k\bcf{p-1}{k-1}=-1\pmod{p}$ for all $k$, so
 $\xi(v_k)=-s^{k-1}a_p$.   We therefore have
 \[ \ip{e^*(s),\sg^*(c_p^mv_k)} = (-1)^{pm+1}s^{mp+k-1}, \]
 and $\ip{e^*(s),\sg^*(c^\al v^\ep)}=0$ for all other monomials
 $c^\al v^\ep$.  It follows from Lemma~\ref{lem-res-v} that this is
 the same as $\ip{\ov{e}(s),\rho^*(c^\al v^\ep)}$, as expected.
\end{proof}

\section{\texorpdfstring{Morava $E$-theory and $K$-theory}
                                          {Morava E-theory and K-theory}}
\lbl{sec-morava}

\begin{definition}\lbl{defn-E}
 We write $E$ for the Morava $E$-theory spectrum of height $n$ at the
 prime $p$.  To be definite, we use the version with 
 \[ E_* = \Zp\psb{u_1,\dotsc,u_n}[u^{\pm 1}] \] where $|u|=2$ and
 $|u_i|=0$.  We make the convention $u_0=p$ and $u_n=1$.  This
 spectrum has a standard complex coordinate $x\in\tE^0(BU(1))$ such
 that $E^0(BU(1))=E^0\psb{x}$ (and $E^1BU(1)=0$).  The associated
 formal group law satisfies
 \[ \log_F(x) = x + p^{-1}\sum_{i=1}^n \log_F(u_ix^{p^i}). \]
 We also let $K$ denote the corresponding $2$-periodic Morava
 $K$-theory spectrum, with 
 \[ K_* = E_*/(u_0,\dotsc,u_{n-1}) = \Fp[u^{\pm 1}]. \]
\end{definition}

\begin{remark}
 We choose to work with a specific version of Morava $E$-theory for
 expository convenience, but it is not hard to transfer the results to
 other versions.  Indeed, all formal group laws of height $n$ over
 finite fields of characteristic $p$ become isomorphic over the
 algebraic closure of $\F_p$.  It follows that all versions of Morava
 $K$-theory become isomorphic after tensoring with that algebraic
 closure, and thus that our Morava $K$-theory results are valid for
 all versions.  A similar argument works for different versions of
 Morava $E$-theory, although the details are a little more
 complicated. 
\end{remark}

\begin{definition}\lbl{defn-HH}
 We will often discuss rings of the form $E^0(X)$ in terms of the
 geometry of formal schemes, as in~\cites{st:fsfg,st:cag,st:fsf}.  In
 particular, we put $S=\spf(E^0)$ and $\GG=\spf(E^0(BU(1)))$ and
 $\HH=\spf(E^0(BGL_1(\bF)))$.  
\end{definition}

\begin{remark}\lbl{rem-E-BUd}
 It is again standard that $E^0(BU(1)^d)=E^0\psb{x_1,\dotsc,x_d}$.
 Moreover, if we let $c_k$ denote $(-1)^k$ times the $k$'th elementary
 symmetric polynomial in the variables $x_i$, then the inclusion
 $U(1)^d\to U(d)$ gives an isomorphism 
 \[ E^0(BU(d)) \to E^0\psb{c_1,\dotsc,c_d} = 
     E^0\psb{x_1,\dotsc,x_d}^{\Sg_d} \leq E^0(BU(1)^d).
 \]
 (Moreover, $E^1$ is zero for all spaces mentioned above.)  In the
 language of formal schemes, we can say that
 $\spf(E^0(BU(1)^d))=\GG^d$ and $\spf(E^0(BU(d)))=\GG^d/\Sg_d$.  This
 can also be identified with the moduli scheme $\Div_d^+(\GG)$ for
 effective divisors of degree $d$ on $\GG$, as discussed
 in~\cite{st:fsfg}*{Section 5.1}.
\end{remark}

\begin{remark}\lbl{rem-period}
 From the functional equation for $\log_F(x)$ it is easy to see that
 $\log_F(x)$ involves only powers $x^i$ with $i=1\pmod{p-1}$.  Thus,
 over the ring $E_0[w]/(w^p-w)$ we have $w^{k(p-1)+1}=w$ for all $k$,
 so $\log_F(wx)=w\,\log_F(x)$, so $\exp_F(wx)=w\,\exp_F(x)$, so
 $(wx)+_F(wy)=w\,(x+_Fy)$, so $[k]_F(wx)=w\,[k]_F(x)$ for all
 $k\in\Z$.  From this it follows that $\exp_F(x)$ and $[k]_F(x)$
 involve only powers $x^i$ with $i=1\pmod{p-1}$, and $x+_Fy$ involves
 only monomials $x^iy^j$ with $i+j=1\pmod{p}$.
 
 Recall also that the polynomial $t^p-t$ factors completely in $\Zp$,
 and the reduction map $\Zp\to\Fp$ gives a bijection from the set of
 roots to $\Fp$.  If $m$ is one of these roots then we can substitute
 it for $w$ in the above discussion, giving $(mx)+_F(my)=m\,(x+_Fy)$
 and $[k]_F(mx)=m\,[k]_F(x)$.

 Note also that the standard definition of $[m]_F(x)$ for $m\in\Z$ can
 be extended to $m\in\Zp$ by the rule $[m]_F(x)=\exp_F(m\,\log_F(x))$.
 If $m^p=m$ as above, then we just get $[m]_F(x)=mx$.
\end{remark}

\begin{proposition}\lbl{prop-p-series}
 For each $k\geq 0$ there is a unique monic polynomial $h_k(x)$ of
 degree $p^{nk}$ over $E_0$ such that 
 \[ h_k(x)=x^{p^{nk}} \pmod{u_0,\dotsc,u_{n-1}} \]
 and $[p^k]_F(x)$ is a unit multiple of $h_k(x)$ in $E^0\psb{x}$.
 Moreover: 
 \begin{itemize}
  \item[(a)] If $m$ is a unit multiple of $p^k$ in $\Zp$, then
   $[m]_F(x)$ is also a unit multiple of $h_k(x)$ in $E^0\psb{x}$.
  \item[(b)] There is a unique polynomial $\bh_k$ such that
   $h_k(x)=x\,\bh_k(-x^{p-1})$.
 \end{itemize}
\end{proposition}
\begin{proof}
 The polynomial $h_k(x)$ exists by the formal Weierstrass preparation
 theorem (see~\cite{st:fsfg}*{Section 5.2}, for example).  For
 claim~(a), we can write $m=p^km'$ with $m'\neq 0\pmod{p}$, and this
 gives $[m]_F(x)=[m']_F([p^k]_F(x))$.  Here $[m']_F(x)=m'x+O(x^2)$ and
 $m'$ is invertible in $\Zp$; it follows that $[m']_F(x)$ is a unit
 multiple of $x$, and thus that $[m']_F([p^k]_F(x))$ is a unit
 multiple of $[p^k]_F(x)$, and thus also of $h_k(x)$.  For claim~(b),
 suppose that $m\in\Zp$ with $m^{p-1}=1$.  We then find that
 $m^{-1}h_k(mx)$ has the defining properties of $h_k(x)$, and so is
 equal to $h_k(x)$.  This proves that $h_k(x)$ contains only terms
 $x^i$ with $i=1\pmod{p-1}$, and this implies the existence of
 $\bh_k$.
\end{proof}

There is an extensive theory of the structure of $E^0(B\CG)$ for
finite groupoids $\CG$.  We next recall some of this theory.

\begin{definition}\lbl{defn-K-loc-cat}
 We write $\CS$ for the category of spectra that are local (in the
 sense of Bousfield) with respect to our Morava $K$-theory spectrum
 $K$.  Given $X,Y\in\CS$ we write $X\Smash Y$ for the $K$-localised
 smash product.  This makes $\CS$ into a symmetric monoidal category,
 whose unit is the $K$-local sphere spectrum, which we call $S$.  We
 also write $DX$ for the function spectrum $F(X,S)$, and
 $\bigWedge_iX_i$ for the $K$-localised wedge of a family of objects
 $X_i$.  Given any groupoid $\CG$, we write $L\CG$ for the
 $K$-localisation of $\Sgip B\CG$.
\end{definition}

We next recall some duality theory for $E^0(B\CG)$.  We will be
primarily interested in the case where $\CG$ is even, as defined
below. 
\begin{definition}\lbl{defn-even}
 Let $\CG$ be a hom-finite groupoid.  We say that $\CG$ is \emph{even}
 if $K_1(B\CG)=0$.
\end{definition}

\begin{remark}\lbl{rem-even}
 If $\CG$ is even, we see from~\cite{host:mkl}*{Section 8} that
 $K^1(B\CG)=E^\vee_1(B\CG)=E^1(B\CG)=0$, and that $E^\vee_0(B\CG)$ and
 $E^0(B\CG)$ are both pro-free, and that
 $E^0(B\CG)=\Hom_{E_0}(E^\vee_0(B\CG),E_0)$.  Moreover, if $\CG$ is
 actually finite then $E^\vee_0(B\CG)$ and $E^0(B\CG)$ are free
 modules of the same finite rank over $E_0$.
\end{remark}

In~\cite{st:kld} we constructed, for every functor $\phi\:\CG\to\CH$
of finite groupoids, a map $R\phi\:L\CH\to L\CG$ which is adjoint, in
a certain sense, to $L\phi\:L\CG\to L\CH$.  (In many places, we will
use the notation $\phi$ for $L\phi$ and $\phi^!$ for $R\phi$.)  If
$\phi$ is a finite morphism between hom-finite groupoids then we can
decompose it as the coproduct of a family of functors
$\phi_i\:\CG_i\to\CH_i$ of finite groupoids, and then define
\[ R\phi = \bigWedge_i R\phi_i \:
   L\CH  = \bigWedge_i L\CH_i \to 
   \bigWedge_i L\CG_i = L\CG.
\]
It is easy to check that this is independent of the choice of
decomposition.  We leave to the reader the task of adapting results
from~\cite{st:kld} to this slightly more general context.  In
particular, subject to suitable finiteness conditions, we have  Mackey
property: for any homotopy pullback square of groupoids
\begin{center}
 \begin{tikzcd}
    \CA \arrow[r,"\al"] \arrow[d,"\bt"'] & \CB \arrow[d,"\gm"] \\
    \CC \arrow[r,"\dl"'] & \CD, 
  \end{tikzcd}
\end{center}
we have $(R\gm)(L\dl)=(L\al)(R\bt)\: L\CC\to L\CB$.  As a special case
of this, one can show that $R\phi=L\phi^{-1}$ whenever $\phi$ is an
equivalence.  (However, we will see cases where $\phi$ is not an
equivalence but nonetheless $L\phi$ is an equivalence; in these cases,
$R\phi$ is usually different from $(L\phi)^{-1}$.)

The map $L\phi\:L\CG\to L\CH$ induces a ring map 
$E^0(B\CH)\to E^0(B\CG)$, which we call $\phi^*$.  Similarly, $R\phi$
induces a map $E^0(B\CG)\to E^0(B\CH)$, which we denote by $\phi_!$.  For a
homotopy pullback square as above, we then have
\[ \dl^*\gm_!=\bt_!\al^*\:E^0(B\CB)\to E^0(B\CC) \]
By applying this to a suitably chosen square, we obtain a Frobenius
reciprocity formula: for any $\phi\:\CG\to\CH$ and any 
$a\in E^0(B\CH)$ and $b\in E^0(B\CG)$ we have 
\[ a.\phi_!(b) = \phi_!(\phi^*(a)\,b) \in E^0(B\CH). \]

Recall also that when $\phi$ is faithful, the map 
$B\phi\:B\CG\to B\CH$ is (up to homotopy equivalence) a covering map,
and $\phi^!$ is the $K$-localisation of the associated transfer map
$\Sgip B\CH\to\Sgip B\CG$.

If $\CG$ has only finitely many isomorphism classes, then the
projection $\ep\:\CG\to 1$ is a finite functor and so induces a map
$\ep_!\:E^0(B\CG)\to E^0$.  Provided that $\CG$ is even, the map
$\ep_!$ is a Frobenius form, which means that the rule
\[ \ip{f,g} = \ep_!(fg) \] 
gives a perfect $E^0$-linear pairing on $E^0(B\CG)$, and so identifies
$E^0(B\CG)$ with the dual module $E^\vee_0(B\CG)$.  More generally, we
can try to define $\ep_!$ as a sum over isomorphism classes, even if
there are infinitely many of them.  However, the domain of $\ep_!$
will not be the whole of $E^0(B\CG)$, as we will need to impose
auxiliary conditions to make the sum converge in the $I_n$-adic
topology. 

If $X=\spf(E^0(B\CG))$, then any element $f\in E^0(B\CG)$ can be thought
of as a function on $X$.  We use the notation $\int_Xf(x)\,dx$ for
$\ep_!(f)$.  In this notation, the perfect pairing is 
\[ \ip{f,g} = \int_X f(x)g(x)\,dx. \]

\begin{remark}\label{rem-int}
 $K(n)$-local self-duality for classifying spaces is the simplest
 manifestation of the phenomenon of ambidexterity, which has been
 developed much further by Hopkins and Lurie~\cite{holu:akl} and by
 Carmeli, Schlank and Yanovski~\cite{cascya:ach}.  The latter paper
 also uses integral notation for a construction related to
 ambidexterity.  However, although there may be some shared intuitive
 motivation, we do not believe that there is any direct connection
 between these constructions.
\end{remark}

More generally, suppose we have $\phi\:\CG\to\CH$, and we
put $Y=\spf(E^0(B\CH))$, so $\phi$ gives a map $X\to Y$ and also a map 
\[ \phi_! \: \CO_X = E^0(B\CG) \to E^0(B\CH) = \CO_Y. \]
we use the notation $\int_{x\in\phi^{-1}\{y\}}f(x)\,dx$ for
$(\phi_!f)(y)$.  In this notation, the Frobenius reciprocity formula
is 
\[ \int_{x\in\phi^{-1}\{y\}} g(\phi(x)) f(x)\,dx = 
    g(y)\,\int_{x\in\phi^{-1}\{y\}} f(x)\,dx.
\]
If $\CG\xra{\phi}\CH\xra{\psi}\CK$ then $(\psi\phi)_!=\psi_!\phi_!$,
which can be written in integral notation as
\[ \int_{x\in(\psi\phi)^{-1}\{z\}} f(x)\,dx = 
    \int_{y\in\psi^{-1}\{z\}} 
     \int_{x\in\phi^{-1}\{y\}} f(x)\,dx\,dy.
\]
Given a homotopy cartesian square
\begin{center}
 \begin{tikzcd}
    \CA \arrow[r,"\al"] \arrow[d,"\bt"'] & \CB \arrow[d,"\gm"] \\
    \CC \arrow[r,"\dl"'] & \CD, 
  \end{tikzcd}
 \end{center}
the associated Mackey property $\dl^*\gm_!=\bt_!\al^*$ can be written
as 
\[ \int_{x\in\gm^{-1}\{\dl(y)\}}f(x)\,dx =
    \int_{w\in\bt^{-1}\{y\}} f(\al(w))\,dw.
\]

\begin{proposition}\lbl{prop-quot-equiv}
 Let $\phi\:\CG\to\CH$ be a full functor of hom-finite groupoids that
 is also a $\pi_0$-isomorphism.  Suppose that for all $a\in\CG$, the
 kernel of the surjection $\phi\:\CG(a,a)\to\CH(\phi(a),\phi(a))$ has
 order coprime to $p$.  Then $L\phi\:L\CG\to L\CH$ is an equivalence.
\end{proposition}
\begin{proof}
 We can reduce easily to the following claim: if $G$ is a finite group
 and $N$ is a normal subgroup of order coprime to $p$, then the map
 $K_*(BG)\to K_*(B(G/N))$ is an isomorphism.  Because of the
 Atiyah-Hirzebruch spectral sequence, it will suffice to prove that
 the map $H_*(G;K_*)\to H_*(G/N;K_*)$ is an isomorphism.  There is
 another standard spectral sequence
 $H_*(G/N;H_*(N;K_*))\convto H_*(G;K_*)$.  As $N$ has order coprime to
 $p$, we have $H_0(N;K_*)=K_*$ and $H_i(N;K_*)=0$ for $i>0$.  The
 spectral sequence therefore collapses to give the required
 isomorphism $H_*(G;K_*)\to H_*(G/N;K_*)$.
\end{proof}

\begin{proposition}\lbl{prop-inc-split}
 Let $\phi\:\CG\to\CH$ be a faithful functor of hom-finite groupoids
 that is also a $\pi_0$-isomorphism.  Suppose that for all $a\in\CG$,
 the index of $\phi(\CG(a))$ in $\CH(\phi(a))$ is coprime to
 $p$.  Then the composite $\phi\phi^!\:L\CH\to L\CH$ is an
 equivalence, so $\phi\:L\CG\to L\CH$ is a split epimorphism, and
 $\phi^!\:L\CH\to L\CG$ is a split monomorphism.  In fact, the map
 \[ \Sgip\phi \: \Sgip B\CG \to \Sgip B\CH \]
 is already a split epimorphism after $p$-completion.
\end{proposition}
\begin{proof}
 We can reduce easily to the following claim: if $H$ is a finite
 group, and $G$ is a subgroup of index $m$, and $m$ is coprime to $p$,
 then the composite 
 \[ \Sgip BH \xra{\text{tr}} \Sgip BG \xra{\text{inc}} \Sgip BH \]
 induces an isomorphism in Morava $K$-theory.  Because of the
 Atiyah-Hirzebruch spectral sequence, it will suffice to prove that
 the above map gives an isomorphism in mod $p$ homology.  However, it
 is a standard fact in group homology theory that the effect in
 homology is just multiplication by $m$, which is an isomorphism as
 $m$ is coprime to $p$.
\end{proof}

\begin{proposition}\lbl{prop-twisting-element}
 Suppose we have a functor $\phi\:\CG\to\CH$ satisfying the hypotheses
 of Proposition~\ref{prop-quot-equiv} (so it is full and a
 $\pi_0$-isomorphism, and the kernels of the induced maps on
 automorphism groups have orders coprime to $p$).  Suppose also that
 we have a functor $\psi\:\CH\to\CG$ with $\phi\psi\simeq 1$.  Put
 $u=\phi_!(1)\in E^0(B\CH)$ and $v=\psi_!(1)\in E^0(B\CG)$.  Then the
 maps $\phi^*\:E^0(B\CH)\to E^0(B\CG)$ and
 $\psi^*\:E^0(B\CG)\to E^0(B\CH)$ are isomorphisms and are inverse to
 each other.  Moreover, we have $\phi_!(g)=u.\psi^*(g)$ for all
 $g\in E^0(B\CG)$ and $\psi_!(h)=v.\phi^*(h)$ for all $h\in E^0(B\CH)$.
 In particular, we have $u.\psi^*(v)=1$ and $\phi^*(u).v=1$.
\end{proposition}
\begin{proof}
 Proposition~\ref{prop-quot-equiv} tells us that $L\phi$ is an
 equivalence, so $\phi^*\:E^0(B\CH)\to E^0(B\CG)$ is an isomorphism.
 As $\phi\psi\simeq 1$ we have $\psi^*\phi^*=1$, so $\psi^*$ is also
 an isomorphism and is inverse to $\phi^*$.  We can now use the
 Frobenius reciprocity formula to get
 \[ \phi_!(g) = \phi_!(\phi^*\psi^*(g).1) = \psi^*(g)\,\phi_!(1)
      = \psi^*(g)u
 \]
 as claimed.  Now take $g=v=\psi_!(1)$ to get
 \[ \psi^*(v)u = \phi_!(v)=(\phi\psi)_!(1) = 1_!(1) = 1. \]
 Applying $\phi^*$ to this gives $v\,\phi^*(u)=1$.  We can also use
 Frobenius reciprocity again to get
 \[ \psi_!(h) = \psi_!(\psi^*\phi^*(h).1) = \phi^*(h)\,\psi_!(1)
      = \phi^*(h)v.
 \]
\end{proof}

\begin{definition}\lbl{defn-xE}
 We have chosen an injection $i\:\bG_1=GL_1(\bF)\to S^1$, which gives
 a map $(Bi)^*\:E^*(\CPi)\to E^*(B\bG_1)$.  There is a standard
 generator $x_E\in E^0(\CPi)$, and we also write $x_E$ for the image
 of this class in $E^0(B\bG_1)=E^0(B\bCL)$ or in $E^0(BGL_1(F[m]))$.
 As $\bG_1$ is abelian, and is the colimit of a sequence of cyclic
 groups, standard methods give $E^0(B\bG_1)=E^0\psb{x_E}$.  Similarly,
 we have $K^0(B\bG_1)=K^0\psb{x_K}$, where $x_K$ is the reduction of
 $x_E$ modulo $I_n$.
\end{definition}

\begin{remark}\label{rem-HH}
 In the language of formal schemes, we can say that the object
 $\HH=\spf(E^0(B\bG_1))$ is naturally isomorphic to
 $\Hom(\Hom(\bF^\tm,S^1),\GG)$ and thus unnaturally isomorphic to
 $\GG$.  (If we took the time to set up the relevant definitions, we
 could also describe $\HH$ as $\Tor(\bF^\tm,\GG)$ or as
 $T\ot_{\Zp}\GG$, where
 $T=\Hom(\Z/p^\infty,\mu_{p^\infty}(\bF))\simeq\Zp$ as in
 Remark~\ref{rem-duals}.)
\end{remark}

\begin{definition}\lbl{defn-E-gens-T}
 Now note that the K\"unneth theorem gives
 \[ E^0(B\bT_d) = E^0\psb{x_{E1},\dotsc,x_{Ed}}. \]
 This has an evident action of the symmetric group $\Sg_d$.  We define
 $c_{Ek}\in E^0(B\bT_d)$ to be the sum of all the monomials in the
 orbit of $(-1)^kx_{E1}\dotsb x_{Ek}$.  We also let $c_{Kk}$ be the
 image of $c_{Ek}$ in Morava $K$-theory.
\end{definition}

\begin{proposition}\lbl{prop-EbG}
 The inclusion $\bT_d\to\bG_d$ gives isomorphisms
 \begin{align*}
   E^0(B\bG_d) &\simeq E^0(B\bT_d)^{\Sg_d} =
                 E^0\psb{c_{E1},\dotsc,c_{Ed}} \\
   K^0(B\bG_d) &\simeq K^0(B\bT_d)^{\Sg_d} =
                 K^0\psb{c_{K1},\dotsc,c_{Kd}}.
 \end{align*}
\end{proposition}
\begin{proof}
 The permutations come from inner automorphisms of $\bG_d$, and it
 follows that the image of the restriction map is contained in the
 invariants.  Recall from Section~\ref{sec-ordinary} that the
 cohomology rings $H^*(B\bT_d)$ and $H^*(B\bG_d)$ are both polynomial
 on generators in even degrees.  It follows that the Atiyah-Hirzebruch
 spectral sequences $H^*(B\bT_d;K^*)\convto K^*(B\bT_d)$ and
 $H^*(B\bG_d;K^*)\convto K^*(B\bG_d)$ both collapse.  It follows
 easily from this that $K^*(B\bG_d)$ is as claimed.  In particular,
 $K^*(B\bG_d)$ is concentrated in even degrees, and the restriction
 map is a split monomorphism of $K^*$-modules.  Because $K^*(B\bG_d)$
 is in even degrees, we see from~\cite{host:mkl}*{Section 8} that so
 $E^*(B\bG_d)$ is pro-free and concentrated in even degrees, with
 $E^*(B\bG_d)/I_n=K^*(B\bG_d)$.  The same applies to $B\bT_d$.  By
 choosing bases in Morava $K$-theory that are compatible with the
 splitting, and lifting them to Morava $E$-theory, we see that
 $E^0(B\bG_d)$ also maps isomorphically to $E^0(B\bT_d)^{\Sg_d}$.
\end{proof}

\begin{remark}\lbl{rem-Div-scheme}
 In the language of formal schemes, this says that
 $\spf(E^0(B\bG_d))=\HH^d/\Sg_d$, which is the same as $\Div_d^+(\HH)$
 (the moduli scheme for effective divisors of degree $d$ on $\HH$).
 Thus, we have $\spf(E^0(B\bCV_*))=\Div_*^+(\HH)$.  As $\bCV$ is a
 symmetric bimonoidal category, we see that $\spf(E^0(B\bCV_*))$ has a
 natural structure as a graded semiring scheme.  This just corresponds
 to the usual graded semiring structure of $\Div_*^+(\HH)$ under
 addition and convolution of divisors, which is familiar from the
 parallel case of $\spf(E^0(BGL_*(\C)))=\Div_*^+(\GG)$.  All this is
 discussed in more detail in~\cite{st:fsfg}*{Sections 5.1 and 7.3}.

 Tanabe's main theorem (which we stated as Theorem~\ref{thm-tanabe})
 tells us that
 \[ \spf(E^0(B\CV_*)) = \spf(E^0(B\bCV_*)_\Gm) =
     \spf(E^0(B\bCV_*))^\Gm = \Div_*^+(\HH)^\Gm.
 \]
\end{remark}

\begin{definition}\lbl{defn-bE}
 We let $b_{Ei}\in E^\vee_0(B\bG_1)=E^\vee_0(B\bCV_1)$ be dual to
 $x_E^i$, and similarly for $b_{Ki}\in K_0(B\bG_1)=K_0(B\bCV_1)$.  We
 use the direct sum operation to make $K_*(B\bCV_*)$ into a 
 ring, and similarly for $E^\vee_*(B\bCV_*)$.
\end{definition}

\begin{proposition}\lbl{prop-bE}
 $K_0(B\bCV_*)$ is the polynomial ring $K_0[b_{Ki}\st i\geq 0]$.
 Similarly, $E^\vee_0(B\bCV_*)$ is the $I_n$-adic completion of the
 polynomial ring $E_0[b_{Ei}\st i\geq 0]$.
\end{proposition}
\begin{proof}
 The Morava $K$-theory statement says that $K_0(B\bCV_d)$ is the
 $d$'th symmetric tensor power of $K_0(B\bCV_1)$, which is true by a
 straightforward dualisation of Proposition~\ref{prop-EbG}.  Next, the
 polynomial ring $E_0[b_{Ei}\st i\geq 0]$ maps to $E^\vee_0(B\CV)$,
 and this map must extend over the $I_n$-adic completion because
 $E^\vee_0(B\CV)$ is pro-free.  The extended map becomes an
 isomorphism if we reduce it modulo $I_n$, so the unreduced map is an
 isomorphism by the general theory of pro-free modules.
\end{proof}

\begin{definition}\lbl{defn-bialgebra}
 We will use the term \emph{bialgebra} for a module equipped with
 both an algebra structure and a coalgebra structure.  We make no
 assumption about the interaction between these structures.
\end{definition}

\begin{definition}\lbl{defn-various-products}
 Let $\sg\:\CV_*\tm\CV_*\to\CV_*$ be the direct sum functor.  We use
 the induced map $\sg_*\:K_0(B\CV_*)\ot K_0(B\CV_*)\to K_0(B\CV_*)$ to
 make $K_0(B\CV_*)$ into a graded ring.  (Note that this is
 commutative without any need for $\pm$ signs.  The unit is the
 obvious generator of $K_0(B\CV_0)=K_0(\text{point})$.)  We just write
 $ab$ for $\sg_*(a\ot b)$.  We also use the transfer map $\sg^!$ to make
 $K_0(B\CV_*)$ into a cocommutative graded coalgebra, with counit
 given by the obvious projection of $K_0(B\CV_*)\to K_0(B\CV_0)$.  
 We make $E^\vee_0(B\CV_*)$ into a graded bialgebra over $E_0$ in the
 same way.

 Next, we use the diagonal map $\dl\:\CV_*\to\CV_*\tm\CV_*$ to make
 $K^0(B\CV_*)$ and $E^0(B\CV_*)$ into rings.  We just write $ab$ for
 $\dl^*(a\ot b)$.  The grading behaviour is that $ab=0$ if
 $|a|\neq|b|$, and $|ab|=d$ if $|a|=|b|=d$.

 The map $\sg^!$ gives a second product on
 $K^0(B\CV_*)$, and we write $a\tm b$ for $(\sg^!)^*(a\ot b)$.  This
 should be thought of as a kind of convolution.  Here we have the more
 usual kind of grading behaviour: $|a\tm b|=|a|+|b|$.  The map $\sg^*$
 also gives a coproduct, which again respects the grading in the usual
 way.  We define two products and a coproduct on $E^0(B\CV_*)$ in the
 same way.  If we need to distinguish the ordinary product (induced by
 $\dl$) from the convolution product, then we will call it the
 \emph{diagonal product}.
\end{definition}

\begin{remark}\lbl{rem-bialgebra-duality}
 The theorem of Tanabe tells us that $E^0(B\CV_d)$ is a finitely
 generated free module over $E^0$, and that $E^1(B\CV_d)=0$.  (This
 could also be deduced from the results of Section~\ref{sec-ahss}
 below.)  This means that duality theory~\cite{st:kld} and the
 K\"unneth theorem apply in their simplest forms.
 
 Duality theory identifies $K_0(B\CV_*)$ with $K^0(B\CV_*)$, and this
 identification converts $\sg_*$ to $(\sg^!)^*$ and $\sg_!$ to
 $\sg^*$.  Thus, if we ignore the diagonal product on $K^0(B\CV_*)$,
 then $K_0(B\CV_*)$ is isomorphic to $K^0(B\CV_*)$ as graded
 bialgebras.  Similarly, $E^\vee_0(B\CV_*)$ is isomorphic to
 $E^0(B\CV_*)$ as graded bialgebras.  For our analysis of the
 Atiyah-Hirzebruch spectral sequence, it will be convenient to focus
 on $K_*(B\CV_*)$.  For other parts of this paper, it is more natural
 to focus on $E^0(B\CV_*)$, which is the natural home of the diagonal
 product. 
\end{remark}

\begin{remark}\lbl{rem-not-hopf}
 It turns out that the product and coproduct do not interact in the
 right way to make the ring $R=E^0(B\CV_*)$ into a Hopf algebra.
 In other words, if we make $R\hot R$ into an $E^0$-algebra in
 the obvious way, then the coproduct $\sg^*\:R\to R\hot R$ is not a
 ring map.  However, $\sg^*$ becomes a ring map if we use a slightly
 different product rule on $R\hot R$.  This is essentially the same
 phenomenon as that described by Green in~\cite{gr:hah}, and will be
 discussed in more detail in Section~\ref{sec-HC}.

 On the other hand, if we use the diagonal product, then $E^0(B\CV_*)$
 becomes a Hopf algebra, but with nonstandard grading behaviour.
\end{remark}

\begin{definition}\lbl{defn-Ind-Prim}
 We write $\Dec_d(E^0(B\CV_*))$ for the module of decomposables of
 degree $d$ in $E^0(B\CV_*)$ (with respect to the convolution
 product).  By definition, this is the sum of the images of the
 transfer maps 
 \[ \tau_j \: E^0(BG_j\tm BG_{d-j}) \to E^0(BG_d), \]
 (for $0<j<d$) and so is an ideal in $E^0(BG_d)$.  We also put
 $\Ind_d(E^0(B\CV_*))=E^0(BG_d)/\Dec_d(E^0(B\CV_*))$, and observe that
 this inherits a ring structure.

 The scheme $\spf(\Ind_d(E^0(B\CV_*))$ is a closed subscheme of
 $\spf(E^0(BG_d))=\Div_d^+(\HH)^\Gm$; it should be thought of as the
 subscheme of $\Gm$-invariant divisors that are indecomposable under
 addition.
 
 We also define $\Prim_d(E^0(B\CV_*))$ to be the intersection of the
 kernels of the restriction maps
 \[ \rho_j \: E^0(BG_d)\to E^0(BG_j\tm BG_{d-j}), \]
 which is again an ideal in $E^0(BG_d)$.
\end{definition}

In order to understand $\Ind_*(E^0(B\CV_*))$ in detail, we need to
know that it is a free module over $E^0$.  The only way that we have
succeeded in proving this is by using the Atiyah-Hirzebruch spectral
sequence, as we will describe in Section~\ref{sec-ahss}.  Thus, the
best results about $\Ind_*(E^0(B\CV_*))$ will be deferred to 
Section~\ref{sec-indec}.  Here we will just prove some preliminary
facts. 

\begin{lemma}\lbl{lem-ind-degrees}
 If $d$ is not a power of $p$, then
 \[ \Prim_d(E^0(B\CV_*))=\Ind_d(E^0(B\CV_*))=0. \]
 If $d=1$ then
 \begin{gather*}
  \Ind_d(E^0(B\CV_*))=\Prim_d(E^0(B\CV_*))= \\
  E^0(BG_1)= E^0\psb{x}/[p^r](x)=
  E^0\{c_1^i\st 0\leq i<N_0=p^{nr}\}.
 \end{gather*}
 If $d=p^k$ with $k>0$ then $\Ind_d(E^0(B\CV_*))$ is the cokernel of
 the transfer map $E^0(BG_{p^{k-1}}^p)\to E^0(BG_{p^k})$, whereas
 $\Prim_d(E^0(B\CV_*))$ is the kernel of the restriction map
 $E^0(BG_{p^k})\to E^0(BG_{p^{k-1}}^p)$.  
\end{lemma}
\begin{proof}
 We will give the proofs for $\Ind$; the arguments for $\Prim$ are
 very similar.

 Put $P_j=G_j\tm G_{d-j}$ for $0<j<d$ and
 $J=\Dec_d(E^0(B\CV_*))=\sum_j\img(\tr_{P_j}^{G_d})$. 
 
 Write $d=\sum_id_ip^i$ in base $p$, and put
 $H=\prod_iG_{p^i}^{d_i}\leq G_d$.  Proposition~\ref{prop-non-p-power}
 tells us that $|G_d/H|$ is coprime to $p$, so
 Proposition~\ref{prop-inc-split} tells us that
 $\tr_H^{G_d}\:E^0(BH)\to E^0(BG_d)$ is surjective.  If $d$ is not a
 power of $p$ then the product defining $H$ has more than one factor,
 so $H$ is contained in some group $P_j$ with $0<j<d$, so
 $\tau_j$ is surjective, so $J=E^0(BG_d)$ and $\Ind_d(E^0(B\CV_*))=0$.

 The claim for $d=1$ is clear from the definitions.
 
 Now suppose that $d=p^k$ with $k>0$, and put $L=G_{p^{k-1}}^p<G_d$.
 Put $J'=\img(\tr_L^{G_d})$.  It is clear that $J'\leq J$, and we need
 to prove that $J'=J$.  Let us say that a \emph{block subgroup} of
 $G_d$ is a subgroup $M$ that is conjugate to $\prod_iG_{p^i}^{m_i}$
 with $\sum_ip^im_i=d$ and $M<G_d$.  The \emph{rank} of such a
 subgroup is the minimal $i$ such that $m_i\neq 0$.  If the rank is
 $i$ and $i<k-1$ then $m_i$ must be divisible by $p$, so we can use
 the inclusion $G_{p^i}^{m_i}\to G_{p^{i+1}}^{m_i/p}$ to include $M$
 in a block subgroup of strictly larger rank.  By iterating this, we
 see that every block subgroup is subconjugate to $L$, so
 $\img(\tr_M^{G_d})\leq J'$.  On the other hand, if $0<j<d$ then we
 can apply Proposition~\ref{prop-non-p-power} to $G_j$ and $G_{d-j}$
 to get a block subgroup $M\leq P_j$ such that $\tr_M^{P_j}$ is
 surjective.  From this it is clear that $\img(\tau_j)\leq J'$, as
 required.
\end{proof}

\section{Annihilators and socles}
\lbl{sec-ann}

Another aspect of duality theory involves the structure of
annihilators and socles, as we now explain.

\begin{convention}
 All group(oid)s mentioned in this section are assumed to be finite
 and even, as in Definition~\ref{defn-even}.
\end{convention}

We will prove various facts about $E^0(B\CG)$, which will have
consequences for $K^0(B\CG)$ and $\Q\ot E^0(B\CG)$.  One can check
that these consequences are valid even without the evenness
assumption, but we will not need that.

\begin{definition}\lbl{defn-ann}
 Given a subset $X\sse E^0(B\CG)$ we put  
 \begin{align*}
  X^\perp
   &= \{y\in E^0(B\CG)\st \ip{x,y}=0 \text{ for all } x\in X\} \\
   &= \{y\in E^0(B\CG)\st \ep_!(Xy)=0\} \\
  \ann(X) &= \{y\in E^0(B\CG)\st Xy=0\}.
 \end{align*} 
\end{definition}

\begin{lemma}\lbl{lem-ann}\leavevmode
 \begin{itemize}
  \item[(a)] $\ann(X)$ is always an ideal.
  \item[(b)] If $X$ is an ideal then $X^\perp=\ann(X)$ (and so $X^\perp$ is
   an ideal).
  \item[(c)] If $X$ is a summand in $E^0(B\CG)$ then so is
   $X^\perp$, and $X^{\perp\perp}=X$.
  \item[(d)] If $X$ is both an ideal and a summand, then so is $\ann(X)$,
   and $\ann(\ann(X))=X$.
 \end{itemize} 
\end{lemma}
\begin{proof}
 Straightforward.
\end{proof}

\begin{proposition}\lbl{prop-res-tr}\leavevmode
 Let $\phi\:\CH\to\CG$ be a functor between even finite groupoids, and
 use the resulting ring map $\phi^*\:E^0(B\CG)\to E^0(B\CH)$ to regard
 $E^0(B\CH)$ as a module over $E^0(B\CG)$.
 \begin{itemize}
  \item[(a)] The sets $\ker(\phi^*)$, $\img(\phi_!)$,
   $\ker(\phi^*)^\perp$ and $\img(\phi_!)^\perp$ are ideals in
   $E^0(B\CG)$. 
  \item[(b)] The sets $\ker(\phi_!)$, $\img(\phi^*)$, 
   $\ker(\phi_!)^\perp$ and $\img(\phi^*)^\perp$ are
   $E^0(B\CG)$-submodules of $E^0(B\CH)$.
  \item[(c)] We always have $\ker(\phi^*)=\img(\phi_!)^\perp$ and
   $\ker(\phi_!)=\img(\phi^*)^\perp$ and 
   $\img(\phi_!)\leq\ker(\phi^*)^\perp$ and
   $\img(\phi^*)\leq\ker(\phi_!)^\perp$. 
  \item[(d)] Suppose that either $\img(\phi^*)$ or $\img(\phi_!)$ is a
   summand.  Then all the sets mentioned in~(a) and~(b) are summands,
   and the inequalities in~(c) are actually equalities.
 \end{itemize}
\end{proposition}
\begin{proof}
 Frobenius reciprocity says that $\phi_!$ is $E^0(B\CG)$-linear, and
 $\phi^*$ is also $E^0(B\CG)$-linear by construction.  Claims~(a)
 and~(b) follow from this, together with Lemma~\ref{lem-ann}.
 
 Now suppose that $b\in E^0(B\CG)$.  We have $b\in\img(\phi_!)^\perp$
 iff $\ip{\phi_!(a),b}_G=0$ for all $a$, or equivalently
 $\ip{a,\phi^*(b)}_H=0$ for all $a$.  As the pairing on $E^0(B\CH)$ is
 perfect, this means that $\phi^*(b)=0$.  We now see that
 $\img(\phi_!)^\perp=\ker(\phi^*)$, and it follows that
 $\img(\phi_!)\leq\img(\phi_!)^{\perp\perp}=\ker(\phi^*)^\perp$.  This
 proves half of~(c), and the other half can be proved by the same
 method.

 Now suppose that the module $M=\img(\phi^*)$ is a summand in
 $E^0(B\CH)$, so we can choose a splitting $E^0(B\CH)=M\oplus N$.  This in
 particular means that $M$ is a free $E^0$-module, so the epimorphism
 $\phi^*\:E^0(B\CG)\to M$ can be split.  This means that $\ker(\phi^*)$
 is a summand, and we can identify $\phi^*$ with the projection
 $\ker(\phi^*)\oplus M\to M$ followed by the inclusion
 $M\to M\oplus N$.  It follows that the dual map $(\phi^*)^\vee$ can
 be identified with the projection $M^\vee\oplus N^\vee\to M^\vee$
 followed by the inclusion $M^\vee\to\ker(\phi^*)^\vee\oplus M^\vee$,
 so the kernel and image of $(\phi^*)^\vee$ are summands.  However, we
 can use the inner products and the adjunction formula to identify
 $\phi_!$ with $(\phi^*)^\vee$, so $\ker(\phi_!)$ and $\img(\phi_!)$
 are summands.  Recall also that if $X$ is a summand then so is
 $X^\perp$ and we have $X^{\perp\perp}=X$.  All claims are now clear
 for the case where $\img(\phi^*)$ is a summand.  The other case
 (where $\img(\phi_!)$ is a summand) is similar.
\end{proof}

\begin{example}\lbl{eg-bad-restriction}
 For an example where the above conditions do not hold, take $p=2$ and
 let $G$ be the quaternion group of order $8$.  (Of course this
 violates our standing assumption that $p>2$.  We expect that similar
 examples can be given for odd primes, but we have not checked.)  Let
 $H$ be the centre, which has order $2$.  Then $H$ has a single
 nontrivial linear character $\dl$, with Euler class $t$ say, so
 $E^0(BH)=E^0\psb{t}/[2](t)$.  On the other hand, $G/H$ has linear
 characters $\al,\bt,\gm$, with Euler classes $x$, $y$ and $z=x+_Fy$,
 so
 \begin{align*}
  E^0(B(G/H)) 
  &= E^0\psb{x,y}/([2](x),[2](y)) \\
  &= E^0\psb{x,y,z}/([2](x),[2](y),[2](z),x+_Fy+_Fz).
 \end{align*}
 We can regard $\al$, $\bt$ and $\gm$ as characters of $G$, and thus
 we can regard $x$, $y$ and $z$ as elements of $E^0(BG)$.  In
 $E^0(B(G/H))$ we have $xyz\neq 0$.  However, on any abelian subgroup
 of $G$ at least one of $\al$, $\bt$ and $\gm$ will vanish, so $xyz$
 will be zero.  Generalised character theory tells us that the
 restriction maps to abelian subgroups are jointly injective, so
 $xyz=0$ in $E^0(BG)$.  (We will review the relevant theory in
 Section~\ref{sec-hkr}; homomorphisms from $\Z_p^n\to G$ play a
 central role, and the key point here is just that any such
 homomorphism must factor through an abelian subgroup of $G$.)  We
 also have $z=x+_Fy=x-_Fy$ and so $z$ is a unit multiple of $x-y$ so
 $xy(x-y)=0$.

 The group $G$ has one more nontrivial irreducible character $\rho$ of
 dimension two, which satisfies $\rho|_H=2\dl$.  We write $e$ for the
 Euler class of $\rho$.  It is known that $E^0(BG)$ is generated by
 $x$, $y$ and $e$, subject to relations that we will not record here.
 Restriction to $H$ sends $x$, $y$ and $z$ to zero, and $e$ to $t^2$.
 Because $[2](t)$ contains both odd and even powers of $t$, it is not
 easy to read off directly the subring of $E^0(BH)$ generated by
 $t^2$.  However, it is easy to see that the map
 $[\Z_p^n,H]\to[\Z_p^n,G]$ is injective, so the restriction map is
 rationally surjective (by using generalised character theory again).
 On the other hand, in Morava $K$-theory we have
 $K^0(BH)=\F_2[t]/t^{2^n}$ so it is clear that $t^2$ generates a
 proper subring, which we will call $R$.  As the rational $E$-theory
 image and Morava $K$-theory image have different ranks, we see that
 the $E$-theory image is not a summand.  We have also shown elsewhere
 that when $n>1$ the Frobenius form $\tht\:K^0(BH)\to K^0$ sends
 $t^{2^n}-1$ to $1$ and all other powers $t^i$ to zero.  Using this we
 see that $R=R^\perp$.  On the other hand, the Morava $K$-theory
 analogue of Proposition~\ref{prop-res-tr} tells us that
 $R^\perp=\ker(\tr\:K^0(BH)\to K^0(BG))$.  This means that $\tr(R)=0$,
 and in particular the element $s_{G,H}=\tr(1)\in K^0(BG)$ is zero.
 Alternatively, it will follow from Lemma~\ref{lem-soc-quot} below
 that $\tr(1)$ is the same as the pullback along $\pi\:G\to G/H$ of
 the class $s_{G/H}$.  We have seen that the map $\pi^*$ is not
 injective, and it follows by a standard argument that it must send
 the socle to zero, so we again have $s_{G,H}=\pi^*(s_{G/H})=0$.  More
 explicitly, we have $s_{G/H}=(xy)^{2^n-1}$.  In $K^0(BG)$ we have
 mentioned that $xy(x-y)=0$, and we can multiply this by
 $(xy)^{2^n-3}y$ to get $(xy)^{2^n-1}=0$.
\end{example}

\begin{remark}\lbl{rem-bad-restriction}
 The above example can be partially generalised as follows.  Let $G$
 be any even finite group, and let $H$ be an abelian normal
 $p$-subgroup.  Consider the following conditions:
 \begin{itemize}
  \item[(a)] $H$ is central in $G$.
  \item[(b)] The map $\Hom(\Tht^*,H)=[\Tht^*,H]\to[\Tht^*,G]$ is
   injective.
  \item[(c)] The map $\Q\ot E^0(BG)\to\Q\ot E^0(BH)$ is surjective.
  \item[(d)] The map $E^0(BG)\to E^0(BH)$ is surjective.
  \item[(e)] The map $K^0(BG)\to K^0(BH)$ is surjective.
  \item[(f)] $H\cap [G,G]=1$.
 \end{itemize}
 Using generalised character theory we see that~(a), (b) and~(c) are
 equivalent.  Standard arguments also show that~(d) and~(e) are
 equivalent, and it is clear that they imply~(a) to~(c).  We also
 claim that~(f) implies~(d) (and so implies all of~(a) to~(e)).
 Indeed, if~(f) holds then we have an inclusion $H\to G/[G,G]$ of
 abelian groups, which gives an epimorphism $G^*=(G/[G,G])^*\to H^*$
 of character groups.  As $E^0(BH)$ is generated by Euler classes of
 characters, it follows that~(d) holds.  From these arguments we see
 that~(f) implies~(a); this can also be shown by a direct
 group-theoretic argument, as the normality of $H$ implies
 $[G,H]\leq H\cap [G,G]$.  We suspect that~(d) implies~(f) but we have
 not proved this.  However, (a) does not imply~(f), as we can see by
 considering the group of $3\tm 3$ upper unitriangular matrices, for
 example. 
\end{remark}

\begin{corollary}\lbl{cor-res-tr}
 Let $\phi\:\CG\to\CH$ be a functor such that the resulting map
 $\phi^*\:E^0(B\CG)\to E^0(B\CH)$ is surjective.  Then
 \begin{itemize}
  \item[(a)] The image of $\phi_!\:E^0(B\CH)\to E^0(B\CG)$ and the
   kernel of $\phi^*\:E^0(B\CG)\to E^0(B\CH)$ are both ideals and
   summands in $E^0(B\CG)$, and they are annihilators of each other.
  \item[(b)] The image of $\phi_!$ is the principal ideal generated by
   $\phi_!(1)$, so $\ker(\phi^*)$ is also the annihilator of
   $\phi_!(1)$.  
  \item[(c)] The map $\phi_!$ is a split monomorphism of
   $E^0$-modules. 
 \end{itemize}
\end{corollary}
\begin{proof}
 As $\phi^*$ is surjective, the image is certainly a summand, so most
 claims are immediate from Proposition~\ref{prop-res-tr}.  Note also
 that for $a\in E^0(B\CH)$ we can choose $b\in E^0(B\CG)$ with
 $\phi^*(b)=a$, and this gives
 $\phi_!(a)=\phi_!(\phi^*(b).1)=b\phi_!(1)$; this proves that
 $\img(\phi_!)=E^0(B\CG).\phi_!(1)$ and so
 $\ann(\phi_!(1))=\ann(\img(\phi_!))=\ker(\phi^*)$.
\end{proof}

\begin{remark}\lbl{rem-res-tr}
 The corollary clearly applies if there is a functor $\psi\:\CG\to\CH$
 with $\psi\phi\simeq 1\:\CH\to\CH$.  In particular, it applies if
 $\CH\simeq 1$.
\end{remark}

\begin{definition}\lbl{defn-soc}
 Suppose that $\CH$ is a subgroupoid of $\CG$ such that the map
 $\pi_0(\CH)\to\pi_0(\CG)$ is bijective.  We write
 $\soc_{\CG,\CH}=i_!(1)$, where $i$ is the inclusion.  For any $\CG$
 we can regard the set $\pi_0(\CG)$ as a groupoid with only identity
 morphisms.  We then have a functor $\CG\to\pi_0(\CG)$, and this has a
 section $j\:\pi_0(\CG)\to\CG$, which is unique up to natural
 isomorphism.  We call $j(\pi_0(\CG))$ the \emph{spine} of $\CG$.  We
 define $\soc_{\CG}=j_!(1)$.  If $G$ and $H$ are groups (regarded as
 one-object groupoids) this just means that $\soc_{G,H}=\tr_H^G(1)$
 and $\soc_G=\tr_1^G(1)$.
\end{definition}

\begin{remark}\lbl{rem-soc}
 Using Corollary~\ref{cor-res-tr} we see that if $i^*$ is surjective,
 then $\soc_{\CG,\CH}$ generates the annihilator of the kernel of
 $i^*$.  We call that annihilator the \emph{relative socle}.  In
 particular, $\soc_{\CG}$ generates the annihilator of the kernel of
 the map $E^0(B\CG)\to\Map(\pi_0(\CG),E^0)$, which we call the
 \emph{socle}.  We also make the obvious parallel definitions in
 Morava $K$-theory (as opposed to $E$-theory).
\end{remark}

\begin{remark}\lbl{rem-soc-terminology}
 Our definitions are only really standard in the case of $K^0(B\CG)$ where
 $|\pi_0(\CG)|=1$, so $K^0(B\CG)$ is a zero-dimensional local ring,
 and the socle is the annihilator of the maximal ideal.
\end{remark}

\begin{lemma}\lbl{lem-soc-res}
 Let $G$ be a finite group and let $H$ be a subgroup.  Then
 $\res^G_H(\soc_G)=|G/H|\soc_H$.  In particular, if $|G/H|$ is not
 divisible by $p$ then $\res^G_H(\soc_G)$ is an invertible multiple of
 $\soc_H$.
\end{lemma}
\begin{proof}
 In the double coset formula for $\res^G_H(\tr_1^G(1))$ the terms are
 indexed by the set $1\backslash G/H=G/H$, and it is easy to see that
 all terms are equal to $s_H$. 
\end{proof}

\begin{lemma}\lbl{lem-soc-quot}
 Suppose that $H$ is a normal subgroup of $G$ with quotient $Q=G/H$,
 and let $\pi\:G\to Q$ be the projection.  Then $\soc_{G,H}=\pi^*(\soc_Q)$.
\end{lemma}
\begin{proof}
 Let $P$ be the homotopy pullback of the inclusion $1\to Q$ and the
 projection $\pi\:G\to Q$.  Then the object set of $P$ is $Q$, with
 morphisms $P(x,y)=\{g\in G\st \pi(g)x=y\}$.  All objects of $P$ are
 isomorphic, so it is equivalent to the group $P(1,1)=H$.  We can
 therefore apply the Mackey property for this pullback to the element
 $1\in E^0(B1)$; this gives $\tr_H^G(1)=\pi^*(\tr_1^Q(1))$ as
 required.  
\end{proof}

\begin{lemma}\lbl{lem-soc-prod}
 Suppose that $H\leq G$ and $t\in E^0(BG)$ with $\res^G_H(t)=\soc_H$;
 then $\soc_G=t\,\soc_{G,H}$.
\end{lemma}
\begin{proof}
 \[ t\,\soc_{G,H}= t\,\tr_H^G(1) = \tr_H^G(\res^G_H(t)) = 
     \tr_H^G(\soc_H) = \tr_H^G\tr_1^H(1) = \tr_1^G(1) = \soc_G.
 \]
\end{proof}

\begin{lemma}\lbl{lem-soc-prod-semi}
 Suppose that $H$ is normal in $G$ and that $Q$ is a subgroup of $G$
 that maps isomorphically to $G/H$ (so $G$ is a semidirect product).
 Then $s_G=s_{G,H}.s_{G,Q}$.
\end{lemma}
\begin{proof}
 We have $Q\backslash G/H=1$ and $Q\cap H=1$ so the double coset
 formula for the element $\res^G_H(s_{G,Q})=\res^G_H(\tr_Q^G(1))$
 reduces to $\tr_1^H(\res^Q_1(1))=\tr_H^1(1)=s_H$.  Thus, the claim
 follows from Lemma~\ref{lem-soc-prod}.
\end{proof}

\section{The theorem of Tanabe}
\lbl{sec-tanabe}

We next recall the outline of Tanabe's proof of
Theorem~\ref{thm-tanabe}, and give an alternative approach to part of
the argument, which may be of independent interest.

\begin{proposition}
 Suppose that $D\in\Div_d^+(\HH)^\Gm$, and that $a\in D$.  Then
 $\prod_{k=1}^{d}x((q^k-1)a)=0$. 
\end{proposition}
\begin{remark}
 In any context where we can use generalised character theory the
 proof is easy: the $\Gm$-orbit of $a$ must have size $k$ with
 $1\leq k\leq d$, and then $(q^k-1)a=0$.  However, the proposition
 implicitly refers to the situation where $D$ and $a$ are defined over
 an arbitrary $E^0$-algebra, which may have torsion and nilpotents, so
 the generalised character theory will not be applicable.  To hande
 this situation, we will freely make use of the framework of formal
 group schemes and their divisors, as in~\cites{st:fsfg,st:fsf}.
\end{remark}
\begin{proof}
 Let $Y_k$ denote the scheme of pairs $(D,a)$ where
 $D\in\Div_d^+(\HH)^\Gm$ and $\sum_{i=0}^{k-1}[q^ia]\leq D$.  More
 explicitly, we start with the ring 
 \[ R = E^0\psb{c_1,\dotsc,c_d}/(r_1,\dotsc,r_d)
      = \CO_{\Div_d^+(\HH)^\Gm}, 
 \]
 where $r_i=\phi^*(c_i)-c_i$.  Put $f(t)=\sum_{i=0}^dc_it^{d-i}\in
 R\psb{t}$.  We then have $Y_1=\spf(Q_1)$, where $Q_1=R\psb{x}/f(x)$.
 The claim is that $\prod_{k=1}^d[q^k-1](x)$ is zero in $Q_1$.

 More generally, we put $g_k(t)=\prod_{i=0}^{k-1}(t-[q^i](x))$ (so
 $g_0(t)=1$ and $g_1(t)=t-x$).  Let $f_k(t)$ and $s_k(t)$ be the
 quotient and remainder when we divide $f(t)$ by $g_k(t)$, so $s_k(t)$
 is a polynomial in $t$ of degree less than $k$ over $R\psb{x}$.  Let
 $J_k$ be the ideal generated by the coefficients of $s_k(t)$ and put
 $Q_k=R\psb{x}/J_k$; we find that $Y_k=\spf(Q_k)$.  Now put
 $m_k=f_k([q^k](x))$ (so $m_0=f(x)$).  We claim that
 $J_k=(m_i\st i<k)$ and that $[q^k-1](x)m_k\in J_k$.  

 To prove this, it is convenient to restate it in more geometric
 language.  Consider a point $(D,a)\in Y_k$.  Put
 $E=\sum_{i=0}^{k-1}[q^ia]$, so $E\leq D$ by hypothesis, so there is a
 unique $D'$ with $D=E+D'$.  The divisors $E$ and $D'$ have equations
 $g_k(t)$ and $f_k(t)$, respectively.  By construction, we have
 $(D,a)\in Y_{k+1}$ iff $q^ka\in D'$ iff $m_k=f_k([q^k](a))=0$.  This
 shows that $J_{k+1}=J_k+(m_k)$.  It follows inductively that
 $J_l=(m_i\st i<l)$ for all $l$ as claimed.

 Next note that $\phi(D)=\phi(E)+\phi(E')$, but $\phi(D)=D$ and
 $\phi(E)=\sum_{i=1}^k[q^ia]$.  If we put
 $E_1=\sum_{i=1}^{d-1}[q^ia]$, this becomes
 $E_1+[a]+D'=E_1+[q^ka]+\phi(D')$, so $[a]+D'=[q^ka]+\phi(D')$, so
 $q^ka\in [a]+D'$, so $([q^k](x)-x)f_k([q^k](x))=0$.  Recall that
 $y-x$ is a unit multiple of $y-_Fx$ in $E^0\psb{y,x}$, so
 $[q^k](x)-x$ is a unit multiple of $[q^k-1](x)$.  Moreover,
 $f_k([q^k](x))$ is just $m_k$, so we have $[q^k-1](x)m_k=0$.  This
 was under the assumption that $(D,a)\in Y_k$, so we have really only
 proved that $[q^k-1](x)m_k\in J_k$.  As $J_{k+1}=J_k+(m_k)$ we see
 that $[q^k-1](x).J_{k+1}\leq J_k$.  This proves that the element
 $u=\prod_{k=1}^d[q^k-1](x)$ satisfies $uJ_{d+1}\leq J_1$.  However, as
 $D$ has degree $d$ we see that $Y_{d+1}=\emptyset$ so $1\in J_{d+1}$
 so $u\in J_1$ as required.
\end{proof}

\begin{corollary}\lbl{cor-regular}
 The ring of functions on $\Div_d^+(\HH)^\Gm$ is a finitely generated
 free module over $E^0$.  Moreover, the sequence $r_1,\dotsc,r_d$ is
 regular in the ring of functions on $\Div_d^+(\HH)$.
\end{corollary}
\begin{proof}
 Put $Y=\Div_d^+(\HH)^\Gm$ and 
 \[ Z=\{a\in\HH^d\st \sum_i[a_i]\in Y\}, \]
 so $\CO_Z$ is a quotient of $E^0\psb{x_1,\dotsc,x_d}$.   Put
 $g(x)=\prod_{i=0}^{d-1}[q^i-1](x)$; the proposition tells us that
 $g(x_i)=0$ in $\CO_Z$ for all $i$.  Here $[q^i-1](x)$ is a unit
 multiple of $[p^v](x)$, where $v=v_p(q^i-1)=r+v_p(i)$.  It is thus a
 unit multiple of the Weierstrass polynomial $h_v(x)$ from
 Proposition~\ref{prop-p-series}, which has degree $p^{nv}$.  This
 means that the ring $E^0\psb{x}/h_v(x)$ is a
 finitely generated free module over $E^0$, and thus that
 $\CO_Z$ is also a finitely generated module over $E^0$.  Next,  
 the map $\HH^d\to\HH^d/\Sg_d=\Div_d^+(\HH)$ is a finite flat map of
 degree $d!$, so $\CO_Z\simeq\CO_Y^{d!}$ as $E^0$-modules, so $\CO_Y$
 is also a finitely generated module over $E^0$.  It follows that the
 quotient $\CO_Y/I_n$ is finitely generated as a module over
 $E^0/I_n=\Fp$ (and we can use the divisor $d[0]$ to see that it is
 nontrivial).  In other words, if we start with the complete regular
 local ring 
 \[ E^0(B\bG_d)=\Zp\psb{u_1,\dotsc,u_{n-1},c_1,\dotsc,c_d} \]
 (of Krull dimension $n+d$) and kill the sequence
 $r_1,\dotsc,r_d,u_0,\dotsc,u_{n-1}$ (of length $n+d$) we obtain a
 quotient ring of Krull dimension zero.  By a standard result in
 commutative algebra~\cite{ma:crt}*{Section 14} it follows that the sequence
 is regular.  This in turn means that $r_1,\dotsc,r_d$ is regular on
 $E^0(B\bG_d)$, and that the sequence $u_0,\dotsc,u_{n-1}$ is regular
 on the quotient $\CO_Y=E^0(B\bG_d)/(r_1,\dotsc,r_d)$, which implies
 that $\CO_Y$ is free as an $E^0$-module.
\end{proof}

We now recall the structure of Tanabe's proof.  It is not hard to
check that there is a homotopy cartesian square
\begin{center}
 \begin{tikzcd} 
  BG_d \arrow[r] \arrow[d] & B\bG_d \arrow[d,"{(1,1)}"] \\
   B\bG_d \arrow[r,"{(1,\phi)}"'] & B(\bG_d\tm\bG_d).
 \end{tikzcd}
\end{center}
Tanabe shows that this gives a spectral sequence of Eilenberg-Moore
type: 
\[ \Tor^{K^*(B\bG_d^2)}_{**}(K^*(B\bG_d),K^*(B\bG_d)) 
    \convto K^*(BG_d).
\]
Corollary~\ref{cor-regular} (which Tanabe proved in a slightly
different way) can be used to show that the higher $\Tor$ groups are
zero, so the spectral sequence collapses, and the main claim follows
from this.

\section{The Atiyah-Hirzebruch spectral sequence}
\lbl{sec-ahss}

We now turn to the case of $B\CV$ rather than $B\bCV$, where the
Atiyah-Hirzebruch spectral sequence has many differentials.  We will
first define and analyse a certain spectral sequence $EM^*_{***}$
given by an explicit algebraic construction.  We will then prove that
the Atiyah-Hirzebruch spectral sequence
\[ H_*(B\CV;K_*) \convto K_*(B\CV) \]
is isomorphic to $EM^*_{***}$.  As is usual with spectral sequence
calculations, this will not give unambiguous generators for
$K_*(B\CV)$.  However, it will enable us to prove that $K_0(B\CV)$ is
a polynomial algebra, which will be a crucial piece of information
that we need to prove the completeness of generators and relations
that we can produce by other means.

We will need the following general construction.
\begin{lemma}\lbl{lem-divided-diff}
 Let $A_*$ be a differential graded algebra over $\Fp$ with a
 differential $\dl$ of odd degree.  Then there is a ring homomorphism
 $\phi\:A_{\text{even}}\to H(A)_{\text{even}}$ given by
 $\phi(a)=[a^p]$.  There is also an additive map
 $\al\:A_{\text{even}}\to H(A)_{\text{odd}}$ given by
 $\al(a)=[a^{p-1}\,da]$, which satisfies
 $\al(ab)=\phi(a)\al(b)+\al(a)\phi(b)$.   
\end{lemma}
\begin{proof}
 The key point is the additivity of $\al$.  For this, consider the
 ring $\widetilde{A}=P_\Z[a,b]\ot E_\Z[da,db]$ with the obvious
 differential, graded over $\Z/2$ with $|a|=|b|=0$ and $|da|=|db|=1$.
 Put $\widetilde{\al}(x)=x^{p-1}dx$, so $\al(x)=[\widetilde{\al}(x)]$.
 Put $c=((a+b)^p-a^p-b^p)/p$, which is well-known to lie in
 $\widetilde{A}$.  We find that
 \[ p(\widetilde{\al}(a+b)-\widetilde{\al}(a)-\widetilde{\al}(b)) =
      d((a+b)^p-a^p-b^p) = p\,dc,
 \]
 and $\widetilde{A}$ is torsion-free so
 $\widetilde{\al}(a+b)=\widetilde{\al}(a)+\widetilde{\al}(b)+dc$. 
 This equation in $\widetilde{A}$ gives an equation in
 $\widetilde{A}/p$, but that is a universal example so it holds for
 any DGA over $\Fp$.  It follows that $\al(a+b)=\al(a)+\al(b)$ as
 claimed.  All other claims are straightforward.
\end{proof}

\begin{proposition}\lbl{prop-AHSS-L}
 There is an Atiyah-Hirzebruch spectral sequence
 \[ E_2^{**} = H^*(B\CV(k)_1;K^*) = P[u^{\pm 1},x]\ot E[a] 
     \convto K^*(B\CV(k)_1),
 \]
 with $u\in E_2^{0,-2}$ and $x\in E_2^{2,0}$ and $a\in E_2^{1,0}$.
 The differential $d^r\:E_r^{s,t}\to E_r^{s+r,t+1-r}$ is zero except
 when $r=2p^{n(r+k)}-1$.  In that case we have 
 has $d^r(a)=u^{-1}(\pwr{(ux)}{p^{n(r+k)}})$ and $d^r(x)=0$ and $d^r(u)=0$.
\end{proposition}
\begin{proof}
 Recall that $GL_1(F(k))$ is cyclic of order $q^{p^k}-1$, so the
 $p$-torsion part has order $p^{r+k}$.  It follows that
 \[ K^*(B\CV(k)_1)=K^*\psb{x_K}/([p^{r+k}](x_K)) =
     K^*\psb{x_K}/(\pwr{x_K}{p^{n(r+k)}}). 
 \] 
 It is well-known and easy to see that there is only one possible
 pattern of AHSS differentials that is compatible with this: we must
 have $d^{2p^{n(r+k)}-1}(a)=t\,u^{-1}(\pwr{(ux)}{p^{n(r+k)}})$ for some
 $t\in\Fp^\tm$.  The value of $t$ is not really important, but it
 simplifies bookkeeping if we can pin it down.  In
 Appendix~\ref{apx-ahss} we will check that $t=1$.
\end{proof}

We now consider the dual homological spectral sequence.  Recall that
the elements $b^{(k)}_i\in H_{2i}(B\CV(k)_1)$ and
$e^{(k)}_i\in H_{2i+1}(B\CV(k)_1)$ are dual to $x^i$ and $x^ia$,
respectively.  The element $u\in K_2(S^0)=K^{-2}(S^0)$ has degree $-2$
in the previous cohomological spectral sequence, but it has degree $2$
in the homological version.  Recall also that we put
$b^{(k)}(s)=\sum_ib_is^i$ and $e(s)=\sum_ie^{(k)}_is^i$.

\begin{corollary}\lbl{cor-AHSS-L}
 The Atiyah-Hirzebruch spectral sequence
 \[ H_*(B\CV(k)_1;K_*) =
     K_*\{b^{(k)}_i\st i\geq 0\} \oplus K_*\{e^{(k)}_i\st i\geq 0\}
      \convto K_*(B\CV(k)_1)
 \]
 has
 \begin{align*}
   d_{2p^{n(r+k)}-1}(b^{(k)}_i) &= 
    \begin{cases}
     \pwrbb{u}{p^{n(r+k)}-1}\;e_{i-p^{n(r+k)}} & \text{ if } i \geq p^{n(r+k)} \\
     0 & \text{ if } i<p^{n(r+k)}
    \end{cases} \\
   d_{2p^{n(r+k)}-1}(e_i) &= 0.  
 \end{align*}
 Moreover, all other differentials are zero.
\end{corollary}
\begin{proof}
 Dualise Proposition~\ref{prop-AHSS-L}.
\end{proof}

\begin{corollary}\lbl{cor-AHSS-L-series}
 We have
 \begin{align*}
   d_{2p^{n(r+k)}-1}(b^{(k)}(s)) &=
     u^{-1}(\pwr{(us)}{p^{n(r+k)}}) e^{(k)}(s) \\
   d_{2p^{n(r+k)}-1}(e^{(k)}(s)) &= 0,
 \end{align*}
 and all other differentials are zero.
\end{corollary}
\begin{proof}
 This is a straightforward translation of the previous result into the
 language of formal power series.
\end{proof}

\begin{corollary}\lbl{cor-AHSS-V-diffs}
 In the Atiyah-Hirzebruch spectral sequence
 \[ E^2_{**} = H_*(B\CV;K_*) =
     K_*\ot P[b_i\st i\geq 0]\ot E[e_i\st i\geq 0]
     \convto K_*(B\CV),
 \]
 the series $b(s)^{p^k}$ survives to $E^{2p^{n(k+r)}-1}$ where we have
 \[ d_{2p^{n(k+r)}-1}(\pwr{b(s)}{p^k}) =
     u^{-1}(\pwr{(us)}{p^{n(k+r)}})(\pwrb{b(s)}{p^k-1}) e(s).
 \]
\end{corollary}
\begin{proof}
 The map $\rho\:\CV(k)_1\to\CV_{p^k}\subset\CV$ gives a morphism of
 spectral sequences, so the claim follows from
 Proposition~\ref{prop-rho-lower-star}. 
\end{proof}

\begin{definition}\lbl{defn-CN}
 As in the introduction, we put $N_0=p^{nr}$, and for $k>0$ we put
 \[ N_k = p^{(n-1)k+n(r-1)}(p^n-1). \]
 Next, for $k\geq 0$ we put $\ov{N}_k=\sum_{j=0}^kN_j$ and
 $N^*_k=\bN_k-p^{nk+nr-k}$.  We then put
 \[ \CN_{k} = \begin{cases}
          \{i\in\N\st i<\bN_0\}
           & \text{ if } k = 0 \\
          \{i\in\N\st \bN_{k-1}\leq i<\bN_k\} 
           & \text{ if } k > 0,
         \end{cases}
 \]
 so $|\CN_{k}|=N_k$ and $\N=\coprod_k\CN_{k}$. Let $s_i$ be the index
 such that $i\in \CN_{s_i}$.
\end{definition}

\begin{lemma}\lbl{lem-N-identities}
 For all $k\geq 0$ we have
 \begin{align*}
   p^kN_k                     &=
     \begin{cases}
      p^{nr} & \text{ if } k = 0 \\
      p^{n(r+k)}-p^{n(r+k-1)} & \text{ if } k > 0 
     \end{cases}  \tag{A} \\
   \sum_{j=0}^{k}p^jN_j       &= p^{n(r+k)}                \tag{B} \\
   p^k(\bN_k - N_k^*)         &= p^{n(k+r)}                \tag{C} \\
   p^{k+1}(\bN_k - N_{k+1}^*) &= p^{n(k+r)}                \tag{D} \\
   p^{k+1}N^*_{k+1}           &= (p-1)p^k\bN_k + p^kN^*_k. \tag{E}
 \end{align*}
\end{lemma}
\begin{proof}
 Straightforward expansion of the definitions gives~(A), and~(B)
 follows by induction.  Equation~(C) is immediate from the definition
 of $N_k^*$.  Now replace $k$ by $k+1$ in~(C) and substitute
 $\bN_{k+1}=N_{k+1}+\bN_k$ to get
 \[ p^{k+1}N_{k+1} + p^{k+1}(\bN_k-N_{k+1}^*) = p^{n(k+1+r)}. \]
 We can use~(A) to rewrite $p^{k+1}N_{k+1}$ as
 $p^{n(k+1+r)}-p^{n(k+r)}$ and rearrange to get~(D).  We can then
 subtract~(D) from~(C) and rearrange to get~(E).

\end{proof}

\begin{definition}\lbl{defn-Sk}
 We define a sequence of trigraded rings $S[k]$ as follows.  There is
 a invertible generator $u$ with $|u|=(0,2,0)$.
 \begin{itemize}
  \item[(a)] For $0\leq m<k$ and $i\in\CN_m$ we have a polynomial
   generator $b_{mi}$ with degree $(2p^mi,0,p^m)$.  Note here that the
   condition $i\in\CN_m$ is equivalent to $s_i=m$ or
   $\bN_{m-1}\leq i<\bN_m$.  As $m<k$ this implies $i<\bN_{k-1}$.
  \item[(b)] For $i\geq\bN_{k-1}$ we have a polynomial generator
   $b_{ki}$ with degree $(2p^ki,0,p^k)$.
  \item[(c)] For $i\geq N^*_k$ we have an exterior generator $e_{ki}$
   with degree $(2p^ki+1,0,p^k)$.
 \end{itemize}
 In the case $k=0$ we take $\bN_{-1}$ to be $0$, and note that
 $N^*_0=0$.  We write $b_i$ and $e_i$ for $b_{0i}$ and $e_{0i}$, so
 \[ S[0] = P[b_i\st i\geq 0]\ot E[e_i\st i\geq 0], \]
 with $|b_i|=(2i,0,1)$ and $|e_i|=(2i+1,0,1)$.  We also define
 $S[\infty]$ to be the polynomial algebra on generators $b_{mi}$ for
 all $m\geq 0$ and $i\in\CN_i$, with no exterior generators.
\end{definition}

\begin{remark}\lbl{rem-bmi-extended}
 According to the above definition, $b_{ni}$ is defined as a generator
 of $S[k]$ only when $n=\min(s_i,k)$.  However, we will extend the notation
 by putting $b_{mi}=b_{ni}^{p^{m-n}}$ for $m\geq n=\min(s_i,k)$.  This
 can be restated as follows.  First suppose that we fix $i$.
 \begin{itemize}
  \item[(a)] If $s_i<k$ (or equivalently $i<\bN_{k-1}$) then $b_{ni}$
   is defined for $n\geq s_i$, and is a generator iff $n=s_i$.
  \item[(b)] If $s_i\geq k$ (or equivalently $i\geq\bN_{k-1}$) then
   $b_{ni}$ is defined for $n\geq k$, and is a generator iff $n=k$.
 \end{itemize}
 Suppose instead that we fix $n$.
 \begin{itemize}
  \item[(c)] If $n<k$ then $b_{ni}$ is defined for all $i<\bN_n$, and
   is a generator iff $\bN_{n-1}\leq i<\bN_n$.
  \item[(d)] If $n=k$ then $b_{ni}$ is defined for all $i$, and is
   a generator iff $i\geq\bN_{k-1}$.
  \item[(e)] If $n>k$ then $b_{ni}$ is defined for all $i$, and is
   never a generator.
 \end{itemize}
 Whenever $b_{mi}$ is defined we have $|b_{mi}|=(2p^mi,0,p^m)$.  We
 also take $e_{ki}=0$ for $i<N^*_k$. 
\end{remark}
 
\begin{remark}\lbl{rem-trigraded-signs}
 We will need the usual kind of sign rules, which depend on the total
 degree of the elements involved.  For an element of degree
 $(i,j,k)$, we officially take $i+j$ to be the total degree.  However,
 in all cases that we consider $j$ will be even, so the signs will
 only depend on the parity of $i$.
\end{remark}

\begin{definition}\lbl{defn-Sk-differential}
 We define a map $\dl_k\:S[k]\to S[k]$ of degree
 $(1-2p^{n(k+r)},2p^{n(k+r)}-2,0)$ as follows.  First, we put
 $\dl_k(b_{mi})=0$ for all $m\leq k$ and $i\in\CN_m$.  This covers~(a)
 and part of~(b) in Definition~\ref{defn-Sk}.  In particular, we have
 $\dl_k(b_{ki})=0$ for $i<\bN_k$.  Next, for $i\geq 0$ we put
 \[ \dl_k(b_{k,i+\bN_k})= \pwrb{u}{p^{n(k+r)}-1}e_{k,i+N^*_k}. \]
 (One can check using Lemma~\ref{lem-N-identities}(A) that this has
 the required degree.)  We also put $\dl_k(e_{ki})=0$ for all $i$.
 Finally, we extend $\dl_k$ over the whole ring $S[k]$ by the Leibniz
 rule. 
\end{definition}

\begin{remark}
 The Leibniz rule implies that $\dl_k(a^p)=0$ for all $a$.  In
 particular, if $b_{mi}$ is one of the extra elements defined in
 Remark~\ref{rem-bmi-extended}, then $\dl_k(b_{mi})=0$.
\end{remark}

\begin{proposition}\lbl{prop-Sk-homology}
 The map $\dl_k$ satisfies $\dl_k^2=0$, so we have a homology ring
 $H(S[k];\dl_k)=\ker(\dl_k)/\img(\dl_k)$.  Moreover, there is an
 isomorphism $\tht_k\:S[k+1]\to H(S[k];\dl_k)$ of trigraded
 algebras over the ring $K_*=P[u,u^{-1}]$, defined as follows:
 \begin{itemize}
  \item[(a)] If $m\leq k$ and $i\in\CN_m$ then $\tht_k(b_{mi})=[b_{mi}]$. 
  \item[(b)] If $i\geq\bN_k$ then $\tht_k(b_{k+1,i})=[b_{ki}^p]$.
  \item[(c)] If $i\geq 0$ then
   $\tht_k(e_{k+1,i+N^*_{k+1}})=[b_{k,i+\bN_k}^{p-1}e_{k,i+N^*_k}]$.
 \end{itemize}
 Also, if we interpret $b_{k+1,i}\in S[k+1]$ and $b_{ki}\in S[k]$ as
 in Remark~\ref{rem-bmi-extended}, then the relation
 $\tht_k(b_{k+1,i})=[b_{ki}^p]$ is valid for all $i\geq 0$.
\end{proposition}
\begin{proof}
 The graded Leibniz rule for $\dl_k$ shows that $\ker(\dl_k^2)$ is a
 subring of $S[k]$, and it contains all the generators $b_{mi}$,
 $e_{ki}$ and $u$, so it is the whole of $S[k]$.  Thus, the homology
 ring is defined.  Now let $A$ be the subring of $S[k]$ generated over
 $K_*$ by $\{b_{mi}\st m\leq k,\;i\in\CN_m\}$, and let $B_i$ be the
 subring generated over $K_*$ by $\{b_{k,i+\bN_k},e_{k,i+N^*_k}\}$, so
 $S[k]$ is the tensor product of all these subrings.  Then $\dl_k$ is
 zero on $A$, and preserves $B_i$.  The K\"unneth theorem therefore
 tells us that $H_*(S[k];\dl_k)$ is the tensor product of $A$ with all
 the rings $H(B_i;\dl_k)$.  As $\dl_k(b_{k,i+\bN_k})$ is an invertible
 multiple of $e_{k,i+N^*_k}$, it is easy to see that
 $H(B_i;\dl_k)=P[b'_i]\ot E[e'_i]$, where $b'_i=[b_{k,i+\bN_k}^p]$ and 
 $e'_i=[b_{k,i+\bN_k}^{p-1}e_{k,i+N^*_k}]$.  Thus, we have
 \[ H_*(S[k];\dl_k) = A\ot P[b'_i\st i\geq 0] \ot E[e'_i\st i\geq 0].
 \]
 It is clear that
 $|b_{k+1,i+\bN_k}|=(2p^{k+1}(i+\bN_k),0,p^{k+1})=|b'_i|$.  Also, one
 can use Lemma~\ref{lem-N-identities}(E) to check that
 $|e_{k+1,i+N^*_{k+1}}|=|e'_i|$.  It follows that our definition of
 $\tht_k$ is compatible with our trigradings.  Now
 consider an element $b_{k+1,i}\in S[k+1]$.  If $i\geq\bN_k$ then
 $\tht_k(b_{k+1,i})=[b_{ki}^p]$ by part~(b) of the definition.  If
 $i<\bN_k$ then $b_{k+1,i}=\pwr{b_{ni}}{p^{k+1-i}}$, where $n=s_i\leq k$.
 We therefore have $\tht_k(b_{k+1,i})=[b_{ni}^{p^{k+1-i}}]$ by
 part~(a) of the definition, but in $S[k]$ we have
 $b_{ki}=b_{ni}^{p^{k-i}}$ so this can be rewritten as
 $\tht_k(b_{k+1,i})=[b_{ki}^p]$ as claimed.
\end{proof}

\begin{definition}\lbl{defn-S-al}
 Using Lemma~\ref{lem-divided-diff}, we define
 $\al_k\:S[k]_{\text{even}}\to S[k+1]_{\text{odd}}$ by 
 $\al_k(a)=\tht_k^{-1}[a^{p-1}\,\dl_k(a)]$.  We thus have
 $\al_k(b_{ki})=0$ for $i<\bN_k$, and
 Proposition~\ref{prop-Sk-homology}(c) gives
 \[ \al_k(b_{k,i+\bN_k}) = \pwrb{u}{p^{n(k+r)}-1} e_{k+1,i+N^*_{k+1}} \]
 for $i\geq 0$.
\end{definition}

\begin{proposition}\lbl{prop-S-conv}
 The bigraded group $S[k]_{**m}$ coincides with $S[\infty]_{**m}$ when
 $p^k>m$.  
\end{proposition}
\begin{proof}
 All the exterior generators of $S[k]$ lie in $S[k]_{**p^k}$ and so
 are irrelevant here.  From the definitions we see that the polynomial
 generators in $S[k]_{**,<p^k}$ are $\{b_{mi}\st m<k,\;i\in\CN_m\}$,
 and these are the same as the polynomial generators in
 $S[\infty]_{**,<p^k}$.  The claim is clear from this.
\end{proof}

The above essentially says that the rings $S[k]$ and differentials
$\dl_k$ give a spectral sequence converging to the polynomial ring
$S[\infty]$, but the grading behaviour is different from what one
would expect for the Atiyah-Hirzebruch spectral sequence.  To fix this
we merely need to insert additional pages with trivial differential.
The result is as follows:

\begin{corollary}\lbl{cor-model-ahss}
 There is a trigraded spectral sequence $EM^*_{***}$ given by
 \[ EM^u_{***} = \begin{cases}
     S[0] & \text{ if } u < 2p^{nr} \\
     S[k] & \text{ if } 2p^{n(r+k-1)} \leq u < 2p^{n(r+k)}
            \text{ with } k>0.
    \end{cases}
 \]
 The differentials $d^u$ have degree $(u,1-u,0)$ and are given by
 $d^u=\dl_k$ when $u=2p^{n(r+k)}-1$, and $d^u=0$ in all other cases.
 Moreover, the spectral sequence converges to $S[\infty]$.
 \qed   
\end{corollary}

\begin{remark}\lbl{rem-ss-language}
 Consider elements $x_0,y_0\in S[0]$.  In the ordinary language used to
 describe spectral sequences, we might say that $x_0$ survives to
 $S[k]$ and supports a differential $\dl_k(x_0)=y_0$.  Because we have
 defined our pages $S[k]$ as independent objects and introduced the
 homomorphisms $\tht_k$ explicitly, we need to use slightly different
 language.  The equivalent statement is that there are elements
 $x_i,y_i\in S[i]$ for $i=0,\dotsc,k$ such that
 \begin{itemize}
  \item For $i<k$ we have $\dl_i(x_i)=0$ and $[x_i]=\tht_i(x_{i+1})$
  \item For $i<k$ we have $\dl_i(y_i)=0$ and $[y_i]=\tht_i(y_{i+1})$
  \item $\dl_k(x_k)=y_k$.
 \end{itemize}
\end{remark}

In Proposition~\ref{prop-rho-lower-star} we described the transfer
maps $H_*(B\CV(n+m))\to H_*(B\CV(n))$ in terms of formal power
series.  It will be convenient to have similar formulae in
$EM^*_{***}$. 
\begin{definition}\lbl{defn-be-series-alt}
 We define $b_k(s),e_k(s)\in S[k]\psb{s}$ by
 \begin{align*}
   b_k(s) &= \sum_{i\geq 0}b_{ki}s^{p^k i}\\
   e_k(s) &= \sum_{i\geq N^*_k} e_{ki}s^{p^k i}.
 \end{align*}
 Recall here that $b_{ki}$ is one of the generators of $S[k]$ when
 $i\geq\bN_{k-1}$, but is defined using Remark~\ref{rem-bmi-extended}
 (and so is a $p$'th power) when $0\leq i<\bN_{k-1}$.  We use the
 grading $|s|=(-2,0,0)$ so that $|b_k(s)|=(0,0,p^k)$ and
 $|e_k(s)|=(1,0,p^k)$.  We define $\dl_k$ and $\tht_k$ on power series
 by the rule $\dl_k(s\,a)=s\,\dl_k(a)$ and
 $\tht_k(s\,a)=s\,\tht_k(a)$; the function $\al_k$ then satisfies
 $\al_k(s\,a)=s^p\,\al_k(a)$.
\end{definition}

\begin{lemma}\lbl{lem-AHSS-model-diffs}
 In $S[k]\psb{s}$ we have
 \begin{align*}
   \dl_k(b_k(s)) &= u^{-1}\pwrbb{(us)}{p^{n(k+r)}}e_k(s) \\
   \al_k(b_k(s)) &= u^{-1}\pwrbb{(us)}{p^{n(k+r)}}e_{k+1}(s) \\
   \tht_k(b_{k+1}(s)) &= [b_k(s)^p] \\
   \tht_k(e_{k+1}(s)) &= [b_k(s)^{p-1}e_k(s)].
 \end{align*}
\end{lemma}
\begin{proof}
 We have $\dl_k(b_k(s))=\sum_{i\geq 0}\dl_k(b_{ki})s^{p^ki}$.  For
 $i<\bN_{k-1}$ the element $b_{ki}$ is defined by
 Remark~\ref{rem-bmi-extended} and is a $p$'th power, so
 $\dl_k(b_{ki})=0$.  For $\bN_{k-1}\leq i<\bN_k$ the element $b_{ki}$
 is a polynomial generator of $S[k]$, but we still have
 $\dl_k(b_{ki})=0$ by definition.  Using this and the identity
 $p^k(\bN_k-N_k^*)=p^n(k+r)$ we get
 \begin{align*}
  \dl_k(b_k(s)) &=
     s^{p^k\bN_k}
      \sum_{i\geq 0} \dl_k(b_{k,\bN_k+i}) s^{p^ki} \\
  &= s^{p^k\bN_k} \pwrb{u}{p^{n(k+r)}-1}
      \sum_{i\geq 0} e_{k,N_k^*+i} s^{p^ki} \\
  &= s^{p^k(\bN_k-N_k^*)} \pwrb{u}{p^{n(k+r)}-1}
      \sum_{i\geq N_k^*} e_{k,i}) s^{p^ki} \\
  &= s^{p^k(\bN_k-N_k^*)} \pwrb{u}{p^{n(k+r)}-1} e_k(s) \\
  &= u^{-1}(us)^{p^{n(k+r)}}e_k(s).
 \end{align*}
 We can analyse $\al_k(b_k(s))$ in a similar way, remembering the rule
 $\al_k(s\,a)=s^p\,\al_k(a)$ and the formulae in
 Definition~\ref{defn-S-al} and the identity
 $p^{k+1}(\bN_k-N_{k+1}^*)=p^{n(k+r)}$.  We get  
 \begin{align*}
  \al_k(b_k(s)) &=
     \pwrb{s}{p^{k+1}\bN_k}
      \sum_{i\geq 0} \al_k(b_{k,\bN_k+i}) \pwrb{s}{p^{k+1}i} \\
  &= s^{p^{k+1}\bN_k} \pwrb{u}{p^{n(k+r)}-1}
      \sum_{i\geq 0} e_{k+1,N_{k+1}^*+i} \pwrb{s}{p^{k+1}i} \\
  &= \pwrb{s}{p^{k+1}(\bN_k-N_{k+1}^*)} \pwrb{u}{p^{n(k+r)}-1}
      \sum_{i\geq N_{k+1}^*} e_{k+1,i} \pwr{s}{p^{k+1}i} \\
  &= \pwrb{s}{p^{k+1}(\bN_k-N_{k+1}^*)} \pwrb{u}{p^{n(k+r)}-1} e_{k+1}(s) \\
  &= u^{-1}\pwrb{(us)}{p^{n(k+r)}}e_{k+1}(s).
 \end{align*}
 Next, using the last part of Proposition~\ref{prop-Sk-homology} we
 get
 \[
  \tht_k(b_{k+1}(s)) =
    \sum_{i\geq 0}\tht_k(b_{k+1,i}) \pwr{s}{p^{k+1}i}
    = \sum_i b_{ki}^p s^{p^{k+1}i} = b_k(s)^p.
 \]
 Finally, for any $a$ we have $\tht_k(\al_k(a))=a^{p-1}\dl_k(a)$ by
 definition.  We can take $a=b_k(s)$ and use our formulae for
 $\al_k(b_k(s))$ and $\dl_k(b_k(s))$ to get
 $z\,\tht_k(e_{k+1}(s))=z\,[b_k(s)^{p-1}\,e_k(s)]$, where
 $z=u^{-1}\pwrb{(us)}{p^{n(k+r)}}$.  This is not a zero divisor, so
 $\tht_k(e_{k+1}(s))=[b_k(s)^{p-1}\,e_k(s)]$ as claimed.
\end{proof}

\begin{corollary}\lbl{cor-AHSS-model-diffs}
 In the sense described in Remark~\ref{rem-ss-language}, the series
 $b(s)^{p^k}\in S[0]\psb{s}$ survives to $S[k]$ and then
 supports a differential
 \[ b(s)^{p^k}\mapsto
     u^{-1}\pwrb{(us)}{p^{n(k+r)}} \pwrb{b(s)}{p^k-1}e(s). \] 
\end{corollary}
\begin{proof}
 For $0\leq i\leq k$, put $x_i=b_i(s)^{p^{k-i}}$ and
 $y_i=u^{-1}(us)^{p^{n(k+r)}}b_i(s)^{p^{k-i}-1}e_i(s)$.  These
 elements have the properties specified in Remark~\ref{rem-ss-language}.
\end{proof}

\begin{definition}\lbl{defn-AHSS-ET}
 We write $ET^*_{***}$ for the Atiyah-Hirzebruch spectral sequence for
 $B\CV$, trigraded so that $ET^*_{**d}$ is the Atiyah-Hirzebruch
 spectral sequence for $B\CV_d$:
 \[ ET^2_{ijd} = H_i(B\CV_d;K_j) \convto K_{i+j}(B\CV_d) \]
\end{definition}

\begin{definition}\lbl{defn-PSV}
 We put $PS(\CV)(t)=\prod_{k\geq 0}(1-t^{p^k})^{-N_k}\in\N\psb{t}$.
\end{definition}

\begin{theorem}\lbl{thm-AHSS-main}
 There is an isomorphism $EM\to ET$ of trigraded spectral sequences.
 Thus, $K_*(B\CV)$ is isomorphic to $S[\infty]$ as a bigraded ring
 (where tridegree $(i,j,d)$ contributes to bidegree $(i+j,d)$).  In
 particular, $K_*(B\CV)$ is a polynomial ring, with Poincar\'e series 
 \[ \sum_d\dim_{K_*}(K_*B\CV_d) = PS(\CV)(t). \]
\end{theorem}
\begin{proof}
 It is clear that the two spectral sequences have the same $E^2$ page,
 and by comparing Corollary~\ref{cor-AHSS-V-diffs} with
 Lemma~\ref{lem-AHSS-model-diffs}, we see that the differentials are
 the same.  It follows in a standard way that the morphism gives an
 isomorphism on every page, so the $E^\infty$ page of the AHSS is
 isomorphic to $S[\infty]$.  This means that $K_*(B\CV)$ has a
 multiplicative filtration whose associated graded ring is a
 polynomial algebra, and by choosing representatives of the generators
 it follows that $K_*(B\CV)$ is isomorphic to the associated graded
 ring, and thus to $S[\infty]$.  Recall that $S[\infty]$ has
 generators $b_{ki}$ of tridegree $(*,*,p^k)$ for all $i\in\CN_k$, and
 $|\CN_k|=N_k$.  Each such generator contributes a factor of
 $\sum_jt^{p^kj}=(1-t^{p^k})^{-1}$ to the Poincar\'e series, so the
 full Poincar\'e series is $PS(\CV)(t)$.
\end{proof}

The following result can be deduced from the theorem of Tanabe, but we
now have an independent proof.

\begin{corollary}\lbl{cor-EV-free}
 For each $d$, the ring $E^*(B\CV_d)$ is a finitely generated free
 module over $E^*$, and is concentrated in even degrees, and the
 natural map $E^*(B\CV_d)/I_n\to K^*(B\CV_d)$ is an isomorphism.
 Similarly, $E^\vee_*(B\CV_d)$ is a finitely generated free module
 over $E_*$, and the natural map $E^\vee_*(B\CV_d)/I_n\to K_*(B\CV_d)$
 is an isomorphism.  In fact, there are natural isomorphisms
 \[ E^\vee_0(B\CV_d)\simeq E^0(B\CV_d)\simeq
    \Hom_{E_0}(E^\vee_0(B\CV_d),E_0).
 \]
\end{corollary}
\begin{proof}
 The theorem shows that $K_*(B\CV_d)$ is concentrated in even degrees,
 and is free of finite rank over $K_*$.  Given this, everything
 follows from the results of~\cite{host:mkl}*{Section 8}.
\end{proof}

We next want to show that $E^\vee_0(B\CV_*)$ and $E^0(B\CV_*)$ and
$K^0(B\CV_*)$ are also graded polynomial rings, under the products
discussed in Definition~\ref{defn-various-products}.  Recall that
these are induced by the direct sum functor
$\sg\:\CV_i\tm\CV_j\to\CV_{i+j}$,  corresponding to the standard
inclusion $GL_i(F)\tm GL_j(F)\to GL_{i+j}(F)$.  Recall also that the
homological and cohomological versions are essentially the same, as
explained in Remark~\ref{rem-bialgebra-duality}.

\begin{proposition}\lbl{prop-more-poly-rings}
 The graded rings $K_*(B\CV_*)$, $K^*(B\CV_*)$, $E^\vee_*(B\CV_*)$ and
 $E^*(B\CV_*)$ are all polynomial over $K_*$ or $E_*$ as appropriate,
 and concentrated in even degrees, and in each case the Poincar\'e
 series is $PS(\CV)(t)$.  Thus, the module $\Ind_{p^k}(E^0(B\CV_*))$
 is free of rank $N_k$ over $E^0$. 
\end{proposition}
\begin{proof}
 We have already seen that $K_*(B\CV_*)$ is polynomial, so we can
 choose a system of generators $\{x_i\st i\in I\}$.  As
 $K_*(B\CV_d)=E^\vee_*(B\CV_d)/I_n$, we can choose lifts
 $x_i\in E^\vee_0(B\CV_*)$.  These give a map
 $\phi\:P_{E_*}[x_i\st i\in I]\to E^\vee_*(B\CV_*)$, which gives an
 isomorphism modulo $I_n$.  As the source and target of $\phi$ are
 free and finitely generated in each degree, it follows that $\phi$
 itself is an isomorphism.  It follows by duality that $K^*(B\CV_*)$
 and $E^*(B\CV_*)$ are polynomial under the convolution products.
 Using duality theory and the fact that
 $K_*(B\CV_*)=E^\vee_*(B\CV_*)/I_n$ we also see that all four rings
 have the same Poincar\'e series, namely $PS(\CV)(t)$.
\end{proof}

\begin{remark}\lbl{rem-not-hopf-b}
 We now outline a different approach that one could hope to use to
 prove that $K_*(B\CV)$ is polynomial.  The work of Tanabe shows that
 $K^*(B\CV)$ is a quotient of $K^*(B\bCV)$, so $K_*(B\CV)$ is a
 subring of $K_*(B\bCV)$, which is a polynomial ring on even
 generators.  It follows that $K_*(B\CV)$ is concentrated in even
 degrees, and that the only nilpotent element is zero.  If we knew
 that $K_*(B\CV)$ was a Hopf algebra, with a coproduct that respects
 gradings in an appropriate way, then we could use the Milnor-Moore
 structure theorem to show that $K_*(B\CV)$ is polynomial.  However,
 as we mentioned in Remark~\ref{rem-not-hopf}, the obvious coproduct
 on $E^*(B\CV_*)$ (derived from the direct sum) does not make
 $E^*(B\CV)$ into a Hopf algebra, and one can check that this problem
 persists if we use $K^*(B\CV_*)$ or $K_*(B\CV_*)$.  If we instead use
 the diagonal map of the space $B\CV$, then we get an alternative
 coproduct that does make $K_*(B\CV_*)$ into a Hopf algebra, but the
 grading behaviour is incompatible with the Milnor-Moore theorem.
 Although we had some success in generalising the Milnor-Moore
 argument, we were not able to complete a proof along those lines.
\end{remark}

We next introduce an element of $K^0(B\CV_{p^k})$ which will turn out
to be a generator of the socle.  For the moment, we will just prove
that it is nonzero.

\begin{definition}\lbl{defn-sk}
 We put $s_k=c_{p^k}\in K^0(B\CV_{p^k})$. 
\end{definition}

\begin{proposition}\lbl{prop-sk-height}
 $s_k^{\bN_k-1}$ is nonzero in $K^0(B\CV_{p^k})$.
\end{proposition}
\begin{proof}
 We have homological and cohomological Atiyah-Hirzebruch spectral
 sequences for $B\CV$ and $B\bCV$.  We denote the final pages by
 $E^\infty(\CV)$, $E_\infty(\CV)$, $E^\infty(\bCV)$ and
 $E_\infty(\bCV)$.  Put 
 \[ v=c_{H,p^k}^{\bN_k-1}\in H^*(B\bCV;K^*)
       =E_2(\bCV)=E_\infty(\bCV).
 \]
 It will suffice to show that $v$ has nontrivial restriction in
 $E_\infty(\CV)$, or that $\ip{m,v}\neq 0$ for some $m$ in the image of the
 map 
 \[ E^\infty(\CV)\to E^\infty(\bCV) = K_*\ot P[b_{Hi}\st i\geq 0].
 \]
 From our earlier analysis, it is clear that this image is the subring
 generated by classes $b_{Hi}^{p^k}$ with $\bN_{k-1}\leq i<\bN_k$.  In
 particular, the class $m=b_{H,\bN_k-1}^{p^k}$ lies in this image.
 Now put $m'=b_{H,\bN_k-1}^{\ot p^k}\in H_*(B\bCV_1^{p^k})$, so the
 usual map $\bCV_1^{p^k}\to\bCV_{p^k}$ sends $m'$ to $m$.  The
 restriction of $v$ to $\bCV_1^{p^k}$ is the element
 $v'=(x_H^{\ot p^k})^{\bN_k-1}$, so we have $\ip{m,v}=\ip{m',v'}$.
 However, it is immediate from the definitions that $\ip{m',v'}=1$, so
 the claim follows.
\end{proof}

\section{Generalised character theory}
\lbl{sec-hkr}

We next recall some ideas about the generalised character theory of
Hopkins, Kuhn and Ravenel~\cite{hokura:ggc}.  We will use the two
kinds of duality introduced in Definition~\ref{defn-duals} and
Remark~\ref{rem-duals}. 

\begin{definition}\lbl{defn-HKR-D}
 Let $D$ denote the algebraic extension of $E_0$ obtained by adjoining
 a full set of roots for $[p^k](x)$ for all $k\geq 0$.  The set $\Tht$
 of all these roots is naturally a group under the formal sum, and an
 isomorphism $(\Z/p^\infty)^n\to\Tht$ is built in to the construction.
 The dual group $\Tht^*=\Hom(\Tht,\Z/p^\infty)$ is thus identified with
 $\Zp^n$.  We regard $\Tht^*$ as a groupoid with one object.  We also
 put $D'=D[p^{-1}]$.
\end{definition}

We will give the next definition in a form that works well for all
groupoids (including $\bCV$, which is not hom-finite).  However, we
will then show that it simplifies for hom-finite groupoids.
\begin{definition}\lbl{defn-functor-groupoid}
 For any groupoid $\CG$, we write $[\Tht^*,\CG]$ for the groupoid of
 functors $\Tht^*\to\CG$ that factor through $\Tht^*/p^k$ for some
 $k\geq 0$.  Equivalently, an object of $[\Tht^*,\CG]$ consists of an
 object $a\in\CG$ and a homomorphism $\al\:\Tht^*\to\CG(a)$ with
 $\al(p^k\Tht^*)=1$ for some $k$.  An isomorphism from $(a,\al)$ to
 $(b,\bt)$ is a morphism $u\in\CG(a,b)$ such that
 $\bt(x)=u\al(x)u^{-1}$ for all $x\in\Tht^*$.

 We will sometimes also write $[\Tht^*,\CG]$ for the set of
 isomorphism classes in the above groupoid, relying on the context to
 distinguish between the two different meanings.
\end{definition}

\begin{remark}\lbl{rem-Lie-topology}
 The condition $\al(p^k\Tht^*)=1$ just means that $\al$ is continuous
 with respect to the $p$-adic topology on $\Tht^*$ and the discrete
 topology on $\CG(a)$.  We might want to consider cases such as the
 groupoid of finite-dimensional complex vector spaces, in which
 $\CG(a)$ has a natural structure as a Lie group.  For any Lie group
 $G$, there is an open neighbourhood $U$ of the identity such that the
 only subgroup contained in $U$ is the trivial one.  Using this, we
 see that a homomorphism $\Tht^*\to\CG(a)$ is continuous with respect
 to the Lie topology iff it is continuous with respect to the discrete
 topology. 
\end{remark}

\begin{lemma}\lbl{lem-Tht-star-quotient}
 Let $\CG$ be a hom-finite groupoid, and suppose we have an object
 $a\in\CG$ and a homomorphism $\al:\Tht^*\to\CG(a)$.  Then there
 always exists $k\geq 0$ with $\al(p^k\Tht^*)=1$.  Thus,
 $[\Tht^*,\CG]$ is the category of all functors $\Tht^*\to\CG$.
\end{lemma}
\begin{proof}
 Put $A=\al(\Tht^*)\leq\CG(a)$.  As $\CG(a)$ is finite, we see that
 $A$ is finite.  As $\Tht^*$ is abelian, we see that $A$ is abelian.
 As every integer coprime to $p$ acts isomorphically on $\Tht^*$, the
 same is true of $A$.  It follows that $A$ is a finite abelian
 $p$-group, so it must have exponent $p^k$ for some $k$.  This implies
 that $\al(p^k\Tht^*)=1$ as claimed.
\end{proof}

\begin{definition}\lbl{defn-HKR-functors}
 For any hom-finite groupoid we put
 \begin{align*}
  D'_0\CG   &= \bigoplus_{\CC\in\pi_0(\CG)} D'_0\ot_{E_0}E^\vee_0(B\CG) \\
  (D')^0\CG &= \prod_{\CC\in\pi_0(\CG)} D'_0\ot_{E_0}E^0(B\CG).
 \end{align*}
 If $\CG$ is actually finite, this simplifies to
 \begin{align*}
  D'_0\CG   &= D'\ot_{E^0}E^\vee_0(B\CG) \\
  (D')^0\CG &= D'\ot_{E^0}E^0(B\CG).
 \end{align*}
\end{definition}

\begin{theorem}[Hopkins-Kuhn-Ravenel]\lbl{thm-HKR}
 For any hom-finite groupoid $\CG$, there are natural isomorphisms
 \begin{align*}
  D'_0\CG   &= D'\{[\Tht^*,\CG]\} \\
  (D')^0\CG &= \Map([\Tht^*,\CG],D')
 \end{align*}
\end{theorem}
\begin{proof}
 For the case of a finite group (or equivalently, a groupoid with only
 one isomorphism class), this is the main theorem of~\cite{hokura:ggc}.
 The general case follows easily.
\end{proof}

\begin{remark}\lbl{rem-char-vals}
 Suppose we have an element $f\in E^0(B\CG)$, and an object
 $a\in[\Tht^*,\CG]$.  We then have an element $1\ot f\in (D')^0(\CG)$,
 and we can apply the generalised character map to get a function
 $[\Tht^*,\CG]\to D'$, which we can evaluate at $a$ to get an element
 of $D'$.  We will usually just write $f(a)$ for this character value.
\end{remark}

We consider the various groupoids defined in
Definition~\ref{defn-cats}.  We will use the functors $A^*$ and $A^\#$
from Definition~\ref{defn-duals}, and the invertible $\Zp$-module $T$
introduced in Remark~\ref{rem-duals}, so that $T\ot A^*$ is naturally
isomorphic to $A^\#$.

\begin{definition}\lbl{defn-Phi}
 We define $\Phi$ to be the group of continuous homomorphisms
 $\Tht^*\to\bF^\tm$, or in other words $\Phi=\Tht^{*\#}=T\ot\Tht$, so
 $\Phi^\#=\Tht^*$.  

 Note that $\Phi$ is a $p$-torsion group.  We also put
 $\Phi[m]=\{\phi\in\Phi\st m\phi=0\}$ (and observe that this only
 depends on the $p$-adic valuation of $m$).
\end{definition}

\begin{remark}\lbl{rem-Phi}
 A choice of generator $t\in T$ gives isomorphisms
 $\Phi\simeq\Tht=(\Z/p^\infty)^n$ and $\Phi^*\simeq\Tht^*=\Z_p^n$.  If
 we write $\Sub(\Phi)$ for the set of finite subgroups of $\Phi$, then
 the first isomorphism gives a bijection
 $\Sub(\Phi)\simeq\Sub(\Tht)$.  Any other generator of $t$ has the
 form $t'=tu$ for some $u\in\Z_p^\tm$, and multiplication by $u$ sends
 every finite subgroup of $\Tht$ or $\Phi$ to itself.  Using this, we
 see that the bijection $\Sub(\Phi)\simeq\Sub(\Tht)$ is actually
 independent of any choices.
\end{remark}

\begin{example}\lbl{eg-HKR-L}
 The inclusion $F\to\bF$ gives an isomorphism
 $\Hom(\Tht^*,F^\tm)\simeq\Phi[q-1]=\Phi[p^r]$.  Recall that $\CL$ is
 the groupoid of one-dimensional vector spaces over $F$.  All such
 spaces are isomorphic, and the automorphism group of any one of them
 is naturally identified with $F^\tm$, so we have
 $[\Tht^*,\CL]=\Hom(\Tht^*,F^\tm)=\Phi[p^r]$.  Thus, we have
 $D'_0(\CL)=D'\{\Phi[p^r]\}$ and $(D')^0(\CL)=\Map(\Phi[p^r],D')$.

 We can also replace $F$ by $F(k)$ in the above analysis to get
 $D'_0(\CL(k))=D'\{\Phi[p^{r+k}]\}$.  
\end{example}

\begin{example}\lbl{eg-HKR-V}
 The main HKR theorem gives
 \[ D'_0(\CV)=D'\{[\Tht^*,\CV]\}=\bigoplus_dD'_0\{[\Tht^*,\CV_d]\}. \]
 We can identify $[\Tht^*,\CV]$ with the semiring of isomorphism
 classes of finite-dimensional $F$-linear representations of $\Tht^*$.
 We will denote this by $\Rep(\Tht^*)$ or $\Rep(\Tht^*;F)$.  We also
 write $\Rep_d(\Tht^*)$ for the subset of representations of dimension
 $d$, and $\Irr(\Tht^*)$ for the subset of irreducible
 representations, and $\Irr_d(\Tht^*)$ for
 $\Irr(\Tht^*)\cap\Rep_d(\Tht^*)$.  Lemma~\ref{lem-Tht-star-quotient}
 tells us that any $V\in[\Tht^*,\CV]$ can be regarded as an $F$-linear
 representation of a finite abelian $p$-group quotient of $\Tht^*$,
 and we can therefore use Lemma~\ref{lem-rep-th} to decompose it in an
 essentially unique way as a direct sum of irreducibles.  For any
 irreducible representation $S$ we write $v_S$ for $[S]$ regarded as
 an element of $\Irr(\Tht^*)$ or of $\Rep(\Tht^*)$ or of
 $D'\{\Rep(\Tht^*)\}$.  We then find that $\Rep(\Tht^*)$ is the free
 abelian monoid generated by the elements $v_S$, and thus that
 $D'\{\Rep(\Tht^*)\}$ is the polynomial ring
 $D'[v_S\st S\in\Irr(\Tht^*)]$.  We will study the structure of
 $\Irr(\Tht^*)$ in Proposition~\ref{prop-Irr-Phi} below.
\end{example}

\begin{example}\lbl{eg-HKR-X}
 Recall that $\CX$ is the groupoid of finite sets and bijections, so
 $[\Tht^*,\CX]$ is the set of isomorphism classes of finite sets with
 action of the group $\Tht^*=\Phi^\#$, or in other words the Burnside
 semiring of $\Phi^\#$.  Let $\Sub(\Phi)$ denote the set of finite
 subgroups of $\Phi$.  For any $A\in\Sub(\Phi)$ we have an epimorphism
 $\Phi^\#\to A^\#$ giving a transitive
 action of $\Phi^\#$ on $A^\#$, and any transitive finite
 $\Phi^\#$-set has this form for a unique group $A$.  We write $x_A$
 for $[A^\#]$ regarded as an element of $[\Tht^*,\CX]$ or
 of $D'_0(\CX)=D'\{[\Tht^*,\CX]\}$.  We find that $[\Tht^*,\CX]$ is
 the free abelian monoid generated by the elements $x_A$, and thus
 that $D'_0(\CX)$ is the polynomial ring $D'[x_A\st A\in\Sub(\Phi)]$.
 (It is more usual to formulate this in terms of subgroups of $\Tht$
 rather than subgroups of $\Phi$, but these are canonically
 identified, as discussed in Remark~\ref{rem-Phi}.)
\end{example}

\begin{example}\lbl{eg-HKR-XL}
 We also see that $[\Tht^*,\CXL]$ is the set of isomorphism classes of
 pairs $(X,L)$, where $X$ is a finite $\Tht^*$-set, and $L$ is an
 $F$-linear line bundle over $X$ with compatible action of $\Tht^*$.
 In Definition~\ref{defn-cats} we explained how to regard $\CXL$ as a
 symmetric bimonoidal category.  This makes $[\Tht^*,\CXL]$ into
 another commutative semiring.  We can construct some elements of
 $[\Tht^*,\CXL]$ as follows.  Consider a finite subgroup $A<\Phi$,
 which gives a subgroup $\ann(A)=\ker(\Phi^\#\to A^\#)$ of finite
 index in $\Phi^\#$.  Consider a homomorphism $a\:\ann(A)\to F^\tm$,
 or equivalently an element $a$ in the group $\ann(A)^\#=\Phi/A$ with
 $(q-1)a=0$, or equivalently a coset $\tilde{a}+A\subset\Phi$ with
 $(q-1)\tilde{a}\in A$.  The map $a$ gives an action of $\ann(A)$ on
 $F$, and this gives a line bundle $L_a$ with total space
 $(F\tm\Phi^\#)/\ann(A)$ over $\Phi^\#/\ann(A)=A^\#$.  There is an
 evident action of the group $\Tht^*=\Phi^\#$ on this line bundle, so
 we have an element $x_{Aa}\in[\Tht^*,\CXL]$.  Given a general object
 $(X,L)\in[\Tht^*,\CXL]$ we can decompose $X$ into orbits, each of
 which must have the form $A^\#$ for some $A$, and then check that
 every equivariant line bundle over $A^\#$ arises as above.  From this
 we see that $D'_0(\CXL)$ is again a polynomial algebra over $D'$,
 generated by these elements $x_{Aa}$.
\end{example}

We now discuss a slightly different picture of $\Rep(\Tht^*)$.
\begin{remark}\lbl{rem-Gm-bar}
 Recall that $\Gm$ is the Galois group of $\bF$ over $F$, which is a
 compact Hausdorff topological group and is topologically generated
 by the Frobenius map $\phi\:x\mapsto x^q$.  This gives rise to an
 action of $\Gm$ on the group $\Phi=\Hom(\Tht^*,\bF^\tm)$ that we
 introduced in Example~\ref{eg-HKR-L}.  We will write the group
 structure on $\Phi$ additively, so the action of $\phi\in\Gm$ is just
 multiplication by $q$.  Note that $\Phi\simeq(\Z/p^\infty)^n$, so
 $\Phi$ is a $p$-torsion group, and is naturally a module over $\Zp$.
 If we let $\ov{\Gm}$ denote the closed subgroup of $\Zp^\tm$
 generated by $q$, it follows that the action of $\Gm$ on $\Phi$
 factors through an epimorphism $\Gm\to\ov{\Gm}$.
\end{remark}

\begin{lemma}\lbl{lem-Gm-bar}
 The group $\ov{\Gm}=\ov{\ip{q}}<\Zp^\tm$ is just $1+p^r\Zp$.
\end{lemma}
\begin{proof}
 Put $U_j=1+p^{r+j}\Zp$, so $q\in U_0$ and $U_0$ is compact so
 $\ov{\Gm}\leq U_0$.  Note also that $U_j/U_{j+1}\simeq\Z/p$.
 Lemma~\ref{lem-vp} tells us that $q^{p^j}\in U_j\sm U_{j+1}$ so
 $q^{p^j}$ generates $U_j/U_{j+1}$.  It follows by induction that the
 map $\ov{\Gm}\to U_0/U_j$ is surjective for all $j$, and we can pass
 to the inverse limit using compactness to see that $\ov{\Gm}=U_0$.
\end{proof}

\begin{lemma}\lbl{lem-Gm-orbit}
 Suppose that $\al\in\Phi$.  Recall that $\Phi\simeq(\Z/p^\infty)^n$,
 so the order of $\al$ must be $p^t$ for some $t\geq 0$.  Consider the
 orbit $\Gm\al\subset\Phi$.
 \begin{itemize}
  \item[(a)] If $t\leq r=v_p(q-1)$, then $\Gm\al=\{\al\}$ and so
   $|\Gm\al|=1$.  
  \item[(b)] If $t\geq r$ then
   $\Gm\al=\{(1+p^rk)\al\st 0\leq k<p^{t-r}\}$ and so $|\Gm\al|=p^{t-r}$.
 \end{itemize}
\end{lemma}
\begin{proof}
 Clear from Remark~\ref{rem-Gm-bar} and Lemma~\ref{lem-Gm-orbit}.
\end{proof}

We next give a slightly different perspective on the semiring
$\Rep(\Tht^*)=\Rep(\Tht^*;F)$, by comparing it with the semiring
$\Rep(\Tht^*;\bF)=[\Tht^*,\bCV]$.  (Recall that this is a well-defined
algebraic object, even though the HKR theorem does not apply to
$\bCV$, because it is not hom-finite.)

\begin{proposition}\lbl{prop-Irr-Phi}
 There are natural isomorphisms
 \begin{align*}
   \Irr(\Tht^*;\bF) &\simeq \Phi &
   \Rep(\Tht^*;\bF) &\simeq \N[\Phi] \\
   \Irr(\Tht^*;F)   &\simeq \Phi/\Gm &
   \Rep(\Tht^*;F)   &\simeq \N[\Phi]^\Gm=\N[\Phi/\Gm].
 \end{align*}
\end{proposition}
\begin{proof}
 First, the semiring $\Rep(\Tht^*;\bF)$ is by definition the union of
 the semirings $\Rep(\Tht^*/p^k;\bF)$.  As $\Tht^*/p^k$ is a finite
 group of order coprime to the characteristic of $\bF$,
 Lemma~\ref{lem-rep-th} is applicable.  This means that every
 representation can be decomposed in an essentially unique way as a
 direct sum of irreducibles, and the irreducibles biject with the
 homomorphisms $\Tht^*/p^k\to\bF^\tm$, or equivalently the elements of
 $\Phi\ip{p^k}$.  By taking the colimit over $k$, we obtain the
 claimed description of $\Irr(\Tht^*;\bF)$ and $\Rep(\Tht^*;\bF)$.

 Next, we have a semiring map $\xi\:\Rep(\Tht^*;F)\to\Rep(\Tht^*;\bF)$
 given by $[V]\mapsto [\bF\ot_FV]$.  Note that $\Gm$ acts on
 $\Rep(\Tht^*;\bF)$ by Definition~\ref{defn-galois-twist}, and the
 image of $\xi$ lies in $\Rep(\Tht^*;\bF)^\Gm$ by
 Remark~\ref{rem-untwisted}.  Consider a $\Gm$-orbit
 $C\subset\Phi$.  This is finite, with $|C|=p^t$ for some $t$, by
 Lemma~\ref{lem-Gm-orbit}.  We can therefore define
 $e_C=\sum_{\al\in C}[\al]\in\N[\Phi]^\Gm$.  It is clear that
 $\N[\Phi]^\Gm$ is the free abelian monoid on the elements of this
 form, so $\N[\Phi]^\Gm\simeq\N[\Phi/\Gm]$.

 Now consider an element $\al\in\Phi$, of order $p^t$ say.  This means
 that $\al(\Tht^*)=\mu_{p^t}(\bF)$.  Proposition~\ref{prop-F-span}
 therefore tells us that the induced map $\al_*\:F[\Tht^*/p^t]\to\bF$
 has image $F(k)$, where $k=\max(t-r,0)$.  Let $V_\al$ denote $F(k)$,
 regarded as an $F$-linear representation of $\Tht^*$ via $\al_*$.
 Because the map $\al_*\:F[\Tht_*/p^t]\to F(k)$ is surjective, we see
 that the subrepresentations of $V_\al$ are $F(k)$-submodules, and so
 are either $0$ or $V_\al$.  Thus, $V_\al$ is irreducible.
 Lemma~\ref{lem-twist-split} shows that $\xi([V_\al])=\sum_j[q^j\al]$
 in $\Rep(\Tht^*;\bF)^\Gm=\N[\Phi]^\Gm$.  This is just the basis
 element corresponding to the orbit $\Gm\al$.  This shows in
 particular that the map $\xi\:\Rep(\Tht^*;F)\to\N[\Phi/\Gm]$ is
 surjective. 

 Now note that if $\bt=\phi^m\al=q^m\al$ then $\bt$ also has order
 $p^t$, and the map $\phi^m\:F(k)\to F(k)$ gives an isomorphism of
 representations between $V_\al$ and $V_\bt$.  Conversely, if
 $V_\al\simeq V_\bt$ then $\bF\ot_FV_\al\simeq\bF\ot_FV_\bt$, so it
 follows from the previous paragraph that $\Gm\al=\Gm\bt$.

 Now let $W$ be an arbitrary irreducible $F$-linear representation of
 $\Tht^*$.  Let $K$ be the image of $F[\Tht^*]$ in $\End(W)$.  As
 $\Tht^*$ is commutative, $K$ consists of equivariant endomorphisms of
 $W$, and any such endomorphism is zero or invertible by Schur's
 Lemma.  Any invertible endomorphism must have finite multiplicative
 order, so the inverse is also in $K$.  Thus, $K$ is a finite field
 extension of $F$.  Any $K$-submodule of $W$ is a subrepresentation,
 and so is zero or $W$; so we must have $\dim_K(W)=1$.  Thus, any
 choice of embedding of $K$ in $\bF$ gives an element $\al$ with
 $W\simeq V_\al$.  The claim now follows.
\end{proof}

\begin{proposition}\lbl{prop-Nk}
 The set $\Irr_d(\Tht^*)$ is empty unless $d$ is a power of $p$.
 Moreover, we have $|\Irr_{p^k}(\Tht^*)|=N_k$ for all $k\geq 0$ (where
 $N_k$ is as in Definition~\ref{defn-CN}).
\end{proposition}
\begin{proof}
 We have seen that $\Irr_d(\Tht^*)$ bijects with the set of
 $\Gm$-orbits of size $d$ in $\Phi$.  Lemma~\ref{lem-Gm-orbit} tells
 us that this is empty unless $d=p^k$ for some $k$.  The same lemma
 tells us that each $\al\in\Phi$ with $p^r\al=0$ gives an orbit
 $\{\al\}$, and these are all the orbits of size $1$.  As
 $\Phi\simeq(\Z/p^\infty)^n$, we see that the number of elements $\al$
 of this type is $p^{nr}=N_0$, as expected.  The lemma also tells us
 that each element of order precisely $p^{r+k}$ gives an orbit of size
 $p^k$.  The number of elements of order precisely $p^{r+k}$ is
 \[ p^{n(r+k)}-p^{n(r+k-1)} = (p^n-1)p^{nr+nk-n}
      = (p^n-1)p^{nk+n(r-1)}.
 \]
 to get the number of orbits, we divide by $p^k$, which gives $N_k$.    
\end{proof}

\begin{remark}
 This now gives us a useful consistency check.  The discussion above
 shows that $D'_0(\CV_*)$ is a polynomial ring over $D'$, with $N_k$
 generators in degree $p^k$, so the Poincar\'e series is
 \[ \sum_d\dim_{D'}(D'_0(\CV_d))t^d =
     \prod (1-t^{p^k})^{-N_k}.
 \]
 We also know that $E^\vee_0(B\CV_*)$ is free and of finite type as an
 $E_0$-module, and $D'_0(\CV_*)=D'\ot_{E_0}E^\vee_0(B\CV_*)$, so
 $E^\vee_0(B\CV_*)$ should have the same Poincar\'e series.  We saw
 this already in Proposition~\ref{prop-more-poly-rings}.
\end{remark}

We conclude this section with some results on the effect of the maps
$R\phi$ in generalised character theory.  These are slight extensions
to the results given in~\cite{st:kld}.  There we defined
$C\CG=\Q\{\pi_0\CG\}$.  Given $\phi\:\CG\to\CH$ we defined
$L\phi\:C\CG\to C\CH$ by $L\phi[a]=[\phi(a)]$ (where $[a]$ denotes the
basis element in $C\CG$ corresponding to the isomorphism class of
$a$).  Also, for $b\in\CH$ we defined
\[ (R\phi)[b] =
    \sum_{[a]|\phi(a)\simeq b} \frac{|\CH(b)|}{|\CG(a)|}[a].
\]
Theorem~\ref{thm-HKR} identifies $D'\hot_{E^0}E^\vee B\CG$ with
$D'\ot_{\Q}C[\Tht^*,\CG]$.  We proved in~\cite{st:kld} that this
identification is compatible with the constructions $L$ and $R$.  (The
statement for $L$ is clear by construction, and if $\phi$ is faithful,
then the statement for $R\phi$ is essentially contained
in~\cite{hokura:ggc}.) 

It will be convenient to give a slightly different formula for $R\phi$
in the case where $\phi$ is a fibration or a covering, as in the
following definition:
\begin{definition}\lbl{defn-covering}
 Consider a functor $\phi\:\CG\to\CH$ of groupoids.  We say that
 $\phi$ is a \emph{fibration} if for all $a\in\CG$ and all
 $h\:\phi(a)\to b'$ in $\CH$, there is an arrow $g\:a\to a'$ in $\CG$
 with $\phi(a')=b'$ and $\phi(g)=h$.  We say that $\phi$ is a
 \emph{covering} if the pair $(a',g)$ is always unique. 
\end{definition}

\begin{example}\lbl{eg-covering}
 As a typical example, consider the groupoids $\CV$ and $\CV^2$ and
 \[ \CW = \{(V;V_0,V_1)\st V\in\CV,\;V_0,V_1\leq V,\;
                V_0+V_1=V,\; V_0\cap V_1=0\}.
 \]
 The external direct sum operation gives a functor $\sg\:\CV^2\to\CV$
 that is not a fibration.  The construction $(V;V_0,V_1)\mapsto V$
 gives a functor $\sg'\:\CW\to\CV$, which is a covering.  The
 forgetful functor $\pi\:\CW\to\CV^2$ is an equivalence with
 $\sg\pi\simeq\sg'$, so $\sg$ and $\sg'$ are in some sense
 equivalent, in particular
 $R\sg=R\sg'\circ(R\pi)^{-1}=R\sg'\circ L\pi$.  The best way to
 understand $R\sg$ is to use this expression, combined with the
 description of $R\sg'$ given by Proposition~\ref{prop-fib-R} below. 
\end{example}

\begin{definition}\lbl{defn-fib-K}
 Suppose that $\phi\:\CG\to\CH$ is a fibration, and that $b$ is an
 object of $\CH$.  Consider the set of objects $a\in\CG$ with
 $\phi(a)=b$ (on the nose), and the morphisms $u\:a\to a'$ with
 $\phi(a)=\phi(a')=b$ and $\phi(u)=1_b$.  These give a category, which
 we denote by $\phi^{-1}\{b\}$.  For $a\in\phi^{-1}\{b\}$ we also put 
 \[ K_a = \ker(\phi\:\CG(a)\to\CH(b)) \]
 (which is trivial if $\phi$ is a covering). 
\end{definition}

\begin{proposition}\lbl{prop-fib-R}
 If $\phi\:\CG\to\CH$ is a fibration, then 
 \[ (R\phi)[b] = \sum_{[a]\in\pi_0(\phi^{-1}\{b\})} |K_a|^{-1}[a].\]
 In particular, if $\phi$ is a covering then 
 \[ (R\phi)[b] = \sum_{[a]\in\pi_0(\phi^{-1}\{b\})} [a].\]
\end{proposition}
\begin{proof}
 Let $\CA$ be the full subcategory of objects $a\in\CG$ with
 $\phi(a)=b$.  We will also write $\CB$ for $\phi^{-1}\{b\}$, so $\CB$
 has the same objects as $\CA$, but fewer morphisms.  Because $\phi$
 is a fibration, every isomorphism class that maps to $[b]$ has a
 representative in $\CA$.  Thus, from the original definition,
 $(R\phi)[b]$ can be written as a sum over $\pi_0\CA$.  Choose a list
 $a_1,\dotsc,a_m$ containing precisely one representative of each
 isomorphism class in $\CA$.  From the original definition of $R\phi$,
 we have 
 \[ (R\phi)[b]=\sum_i|\CH(b)|^{-1}|\CH(a_i)|[a_i]. \]
 Now let $\CB_i$ be the full subcategory of $\CB$ consisting of
 objects $a'$ such that $a'\simeq a_i$ in $\CA$, so
 $\CB=\coprod_i\CB_i$.  Put
 $H_i=\phi(\CG(a_i))\simeq\CG(a_i)/K_{a_i}$ and
 $Q_i=\CH(b)/H_i$, so
 $|Q_i|=|K_{a_i}||\CH(b)||\CG(a_i)|^{-1}$.  For $a'\in\CB_i$ we
 can choose $u\in\CG(a_i,a')$, and then get $\phi(u)\in\CH(b)$.  The
 coset $\phi(u)H_i$ is easily seen to depend only on
 the isomorphism class of $a'$ in $\CB_i$, so this construction gives
 a map $\dl\:\pi_0\CB_i\to Q_i$.  Given a coset represented by
 $v\in\CH(b)$, the fibration property gives us an object $a'\in\CG$
 and a morphism $u\:a_i\to a'$ with $\phi(a')=b$ and $\phi(u)=v$; it
 follows that $a'\in\CB_i$ and $\dl[a']=vH_i$.  On the other hand, if
 $a',a''\in\CB_i$ with $\dl[a']=\dl[a'']$ then we can choose
 $u\:a_i\to a'$ and $w\:a_i\to a''$ and $m\:a_i\to a_i$ with
 $\phi(u)=\phi(w)\phi(m)$, which means that $wmu^{-1}\:a'\to a''$ is an
 isomorphism in $\CB_i$.  We see from this that $\dl$ is a bijection.
 It follows that $[a_i]=|Q_i|^{-1}\sum_{[a]\in\pi_0\CB_i}[a]$.  Using
 the above formula for $|Q_i|$, the $i$'th term in $(R\phi)[b]$ now
 becomes $\sum_{[a]\in\CB_i}|K_{a_i}|^{-1}[a]$, and it is also easy to
 see here that $K_a$ is conjugate to $K_{a_i}$, so we can replace
 $|K_{a_i}|$ by $|K_a|$.  Taking the sum over $i$ now gives the
 formula 
 \[ (R\phi)[b] = \sum_{[a]\in\pi_0(\phi^{-1}\{b\})} |K_a|^{-1}[a] \]
 as claimed.
\end{proof}

To apply this in HKR theory, we also need the following:

\begin{proposition}\lbl{prop-fib-star}
 If $\phi\:\CG\to\CH$ is a fibration, then so is the induced functor
 $\phi_*\:[\Tht^*,\CG]\to[\Tht^*,\CH]$.  Similarly, if $\phi$ is a
 covering, then so is $\phi_*$.
\end{proposition}
\begin{proof}
 Any object in $[\Tht^*,\CG]$ has the form $(a,\al)$, where $a\in\CG$
 and $\al\:\Tht^*\to\CG(a)$.  Suppose we have such an object,
 together with an isomorphism $v\:(\phi(a),\phi\circ\al)\to(b,\bt)$ in
 $[\Tht^*,\CH]$.  Explicitly, this means that $v\:\phi(a)\to b$ in
 $\CH$, and that $\bt(\tht)=v.\phi(\al(\tht)).v^{-1}$ for all
 $\tht\in\Tht^*$.  Because $\phi$ is a fibration, we can choose
 $u\:a\to a'$ in $\CG$ such that $\phi(a')=b$ and $\phi(u)=v$.  We now
 define $\al'\:\Tht^*\to\CG(a',a')$ by
 $\al'(\tht)=u.\al(\tht).u^{-1}$.  Now $u$ can be regarded as a
 morphism $(a,\al)\to(a',\al')$ in $[\Tht^*,\CH]$, and we have
 $\phi_*(a',\al')=(b,\bt)$ and $\phi_*(u)=v$.  This is what is needed
 to show that $\phi_*$ is a fibration.  If $\phi$ is a covering then
 the choice of $(a',u)$ is unique, and it follows easily that the
 choice of lift for $\phi_*$ is also unique; we deduce that $\phi_*$
 is also a covering. 
\end{proof}

\begin{example}\lbl{eg-fib-star}
 Consider again the object
 $D^0(B\CV_k)\simeq\Map(\Rep_k(\Tht^*),D_0)$.  This has a ring
 structure arising from the diagonal map of the space $B\CV_k$; this
 just corresponds to the pointwise product of functions
 $\Rep_d(\Tht^*)\to D_0$, which will just be denoted by juxtaposition,
 so $(f_0f_1)(V)=f_0(V)f_1(V)$.  The direct sum functor
 $\sg\:\CV_i\tm\CV_j\to\CV_{i+j}$ gives a different kind of product 
 \[ \Map(\Rep_i(\Tht^*),D_0) \ot 
    \Map(\Rep_j(\Tht^*),D_0) \to
    \Map(\Rep_{i+j}(\Tht^*),D_0),
 \]
 which we denote by $f_0\tm f_1$.  Using the above propositions
 together with Example~\ref{eg-covering} we find that
 \[ (f_0\tm f_1)(V) = \sum_{V=V_0\oplus V_1} f_0(V_0)f_1(V_1). \]
 In more detail, this should be interpreted as follows.  We have a
 representation $V$ of $\Tht^*$ over $F$, and $(f_0\tm f_1)(V)$ refers
 to the value of $f_0\tm f_1$ on the isomorphism class of $V$.  The
 right hand side of the formula refers to the set of pairs
 $(V_0,V_1)$, where each $V_i$ is a subrepresentation of $V$, and
 $V_0\cap V_1=0$ and $V_0+V_1=V$.  (In particular, we have a sum over
 actual subrepresentations, not over isomorphism classes.)
\end{example}

\begin{remark}\lbl{rem-chi-V}
 As we have mentioned previously, the theory developed
 in~\cite{st:kld} gives a natural inner product on $E^\vee_0(B\CV_d)$,
 which allows us to identify it with $E^0(B\CV_d)$.  This gives rise
 to an inner product on $D'_0(\CV_d)=D_0\{\Rep_d(\Tht^*)\}$.  This is
 just given by $\ip{[W],[V]}=|\Iso(W,V)|$, where $\Iso(W,V)$ is the
 number of $\Tht^*$-equivariant $F$-linear isomorphisms from $W$ to
 $V$.  The standard product on $D'_0(\CV_*)$ is given by
 $[V_0][V_1]=[V_0\oplus V_1]$, and the convolution product on
 $(D')^0(\CV)$ is obtained from this by duality.  In other words, the
 isomorphism $D'_0(\CV_*)\to(D')^0(\CV_*)$ sends $[W]$ to the function
 $\chi_W\:\Rep(\Tht^*)\to\Z\subset D_0$ given by
 \[ \chi_W(V) = \ip{W,V} = |\Iso(W,V)|. \]
 Either by chasing through the duality isomorphisms, or by direct
 analysis of the formula in Example~\ref{eg-fib-star}, we find that
 $\chi_{W_0}\tm\chi_{W_1}=\chi_{W_0\oplus W_1}$.  Thus,
 $(D')^0(\CV_*)$ is polynomial on $\{\chi_S\st S\in\Irr(\Tht^*)\}$.
 This analysis also gives a natural isomorphism
 \[ D'\ot_{E^0} \Ind_{p^k}(\HH) \to \Map(\Irr_{p^k}(\Tht^*),D'). \]
\end{remark}

\begin{remark}\lbl{rem-HKR-coproduct}
 As discussed in Definition~\ref{defn-various-products}, the direct
 sum functor $\sg\:\CV^2\to\CV$ also gives a coproduct on
 $E^0(B\CV)$.  In generalised character theory, this corresponds to a
 map
 \[ \Map(\Rep_*(\Tht^*),D_0) \to \Map(\Rep_*(\Tht^*)^2,D_0). \]
 This is just given by $\sg^*(f)(V_0,V_1)=f(V_0\oplus V_1)$, and the
 counit is $f\mapsto f(0)$.
\end{remark}

\section{Indecomposables in \texorpdfstring{$E^0(B\CV)$}{E0(BV)}}
\lbl{sec-indec}

We now return to consideration of the ring $\Ind_{p^k}(E^0(B\CV_*))$,
and the associated formal scheme.

\begin{definition}\lbl{defn-IJ}
 In this section, we fix $k>0$ and put $G=GL_{p^k}(F)$ and
 $H=GL_{p^{k-1}}(F)^p<G$.  Using Lemma~\ref{lem-ind-degrees}, we put
 \begin{align*}
   R &= E^0(BG) \\
   I &= \Prim_{p^k}(E^0(B\CV_*)) = \ker(\res^G_H\:E^0(BG)\to E^0(BH)) \\
   J &= \Dec_{p^k}(E^0(B\CV_*)) = \img(\tr_H^G\:E^0(BH)\to E^0(BG)),
 \end{align*}
 so $R/J=\Ind_{p^k}(E^0(B\CV_*))$.
\end{definition}

We will use generalised character theory to help analyse the above
objects.  Recall that $D'$ is a flat extension of $E^0$, and
that the map $E^0\to D'$ is injective.  This justifies some implicit
identifications in the following definition.
\begin{definition}\lbl{defn-IJ-prime}
 We put
 \begin{align*}
   R' &= D'\ot_{E^0}R = (D')^0(\CV_{p^k}) =
           \Map(\Rep_{p^k}(\Tht^*),D') \\
   I' &= D'\ot_{E^0}I = \ker(\res\:(D')^0(G)\to (D')^0(H)) \\
   J' &= D'\ot_{E^0}J = \img(\tr\:(D')^0(H)\to (D')^0(G)).
 \end{align*}
\end{definition}

\begin{remark}\lbl{rem-IJ-prime}
 We can split $\Rep_{p^k}(\Tht^*)$ as $X\amalg Y$, where $X$ is the
 set of irreducibles, and $Y$ is the set of reducibles.  We find that
 \begin{align*}
   R' &= \Map(X,D')\tm\Map(Y,D') \\
   I' &= \Map(X,D')\tm 0 \\
   J' &= 0 \tm \Map(Y,D').
 \end{align*}
 (For the third statement, it is helpful to recall
 Remark~\ref{rem-chi-V}.) 
\end{remark}

We next want to discuss the annihilators of $I$ and $J$.  We will use
Lemma~\ref{lem-ann} without comment.

\begin{proposition}\lbl{prop-IJ}
 Both $I$ and $J$ are $E^0$-module summands in $R$, and they are
 annihilators of each other.  Thus, we have $\dim(I)+\dim(J)=\dim(R)$,
 so $\dim(I)=\dim(R/J)$ and $\dim(J)=\dim(R/I)$.
\end{proposition}
\begin{proof}
 As $E^0(B\CV_*)$ is a polynomial algebra over $E^0$, it follows
 easily that the decomposables form an $E^0$-module summand, so $J$ is
 an $E^0$-module summand in $R$.  Duality theory therefore tells us
 that $\ann(J)$ is also a summand, and that $\ann^2(J)=J$.  Consider
 an element $f\in E^0(BG)$, and let $f'$ be the image in $(D')^0(G)$.
 Because $E^0(BG)$ injects in $(D')^0(G)$ and similarly for $H$, we
 see that $fJ=0$ iff $f'J'=0$ iff $f'|_X=0$ iff $(f|_H)'=0$ iff
 $f|_H=0$ iff $f\in I$, so $\ann(J)=I$.  The claims about dimensions
 follow easily.
\end{proof}

\begin{definition}\lbl{defn-IJ-bar}
 We let $\mxi=(u_0,\dotsc,u_{n-1})$ denote the maximal ideal in $E^0$,
 and put $\bR=R/\mxi R$ and $\bI=I/\mxi I$ and $\bJ=J/\mxi J$.
\end{definition}

\begin{lemma}\lbl{lem-IJ-bar}
 \begin{align*}
  \bR &= K^0(BG) \\
  \bI &= \ker(\res^G_H\:K^0(BG)\to K^0(BH)) = \ann(\bJ) \\
  \bJ &= \img(\tr_H^G\:K^0(BH)\to K^0(BG)) = \ann(\bI).
 \end{align*}
\end{lemma}
\begin{proof}
 As $E^0(BG)$ is free and $E^1(BG)=0$, it is standard that
 $K^0(BG)=K^0\ot_{E^0}E^0(BG)=\bR$.  Similarly,
 $K^0(BH)=K^0\ot_{E^0}E^0(BH)$.  We have seen that $I$ and $J$ are
 summands in $R$, and so are free modules over $E^0$ of ranks $r_I$
 and $r_J$ say.  As $I$ and $J$ are annihilators of each other, we
 have $r_I+r_J=\dim_{E^0}(R)$.  It follows easily that $\bI$ and $\bJ$
 are summands in $\bR$, with ranks $r_I$ and $r_J$ over $K^0$.  As
 $IJ=0$ we have $\bI\,\bJ=0$, so $\bJ\leq\ann(\bI)$.  However, $\bR$
 is a Frobenius algebra, so 
 \[ \dim_{K^0}(\ann(\bI))=\dim_{K^0}(\bR)-\dim_{K^0}(\bI)=
     \dim_{K^0}(\bJ),
 \]
 so $\ann(\bI)=\bJ$.  Similarly, we have $\ann(\bJ)=\bI$.

 Next, as $J$ is a summand in $R$, we see that $\tr_H^G$ can be
 written as a split epimorphism of $E^0$-modules, followed by a split
 monomorphism.  The restriction map $\res^G_H$ is adjoint to $\tr_H^G$
 with respect to the standard inner products, and it follows that it
 can also be factored in the same way.  This means that the functor
 $K^0\ot_{E^0}(-)$ preserves the kernel, image and cokernel of
 $\tr_H^G$ and of $\res^G_H$, so $\bI$ is the kernel of the Morava
 $K$-theory restriction map, and $\bJ$ is the image of the Morava
 $K$-theory transfer.
\end{proof}

\begin{definition}\lbl{defn-Q}
 Put $C=GL_1(F(k))<G$, so $C$ is cyclic of order $m=q^{p^k}-1$.  We
 then have a restriction map 
 \[ E^0(BG) \to E^0(BC) = E^0\psb{x}/[m](x). \]
 Here $[m](x)$ is a unit multiple of $[p^{k+r}](x)$, so it is also a
 unit multiple of a Weierstrass polynomial of degree $p^{n(k+r)}$, as
 in Proposition~\ref{prop-p-series}.  The above ring map corresponds
 to a map of formal schemes  
 \[ \{a\in\HH\st q^{p^k}a=a\} \to \Div_{p^k}^+(\HH)^\Gm. \]
 Using Lemma~\ref{lem-twist-split}, we see that this is just 
 \[ a \mapsto \sum_{i=0}^{p^k-1} [q^ia]. \]

 Now put $m'=m/p$, so $[m](t)=[p]([m'](t))$.  As usual, we put
 $\ip{p}(t)=[p](t)/t$, which is a Weierstrass series of degree
 $p^n-1$.  We also put  
 \[ \tQ = E^0\psb{x}/(\ip{p}([m'](x))), \]
 so $\spf(\tQ)\subset\HH$ is the divisor of points of exact order
 $p^{k+r}$.

 The ring $\tQ$ has an action of the Galois group $\Gm$, satisfying
 $\phi^*(x)=[q](x)$.  We put $Q=\tQ^\Gm$, and we put
 \[ s = \prod_{i=0}^{p^k-1}[q^i](x) \in Q. \]

 It is easy to see that the composite
 $E^0(BG)\xra{\res}E^0(BC)\xra{}\tQ$ is $\Gm$-invariant and so lands
 in $Q$.  We define $\al$ to be the resulting map $R=E^0(BG)\to Q$.
\end{definition}

\begin{remark}\lbl{rem-Q}
 The $p$-torsion subgroup of $C$ is cyclic of order $p^{r+k}$.  Recall
 from Corollary~\ref{cor-cyclic-subgroup} that there is a unique
 conjugacy class of such subgroups in $G$, or in any subgroup of $G$
 that contains the monomial matrices.  This is in some sense the real
 reason for the importance of $C$ and $Q$, although that is not
 completely visible in our current approach.
\end{remark}

\begin{proposition}\lbl{prop-Q-regular}
 There is a Weierstrass polynomial $g_k(t)$ of degree $N_k$ such that 
 \[ Q = E^0\psb{s}/g_k(s) = E^0\{s^i\st 0\leq i<N_k\}. \]
 Moreover, both $\tQ$ and $Q$ are complete regular local Noetherian
 rings (so in particular they are unique factorisation domains and are
 Gorenstein). 
\end{proposition}
\begin{proof}
 Note that $[m'](t)$ is a unit multiple of $[p^{r+k-1}](t)$, so it has
 Weierstrass degree $p^{n(r+k-1)}$.  Moreover, $\ip{p}(t)$ has
 Weierstrass degree $p^n-1$, so the Weierstrass degree of the series
 $f_k(t)=\ip{p}([p^{r+k-1}](t))$ is $(p^n-1)p^{n(r+k-1)}=p^kN_k$. 

 Now recall that the maximal ideal $\mxi<E^0$ is generated by a
 regular sequence $u_0,\dotsc,u_{n-1}$ with $u_0=p$.  Note that
 $f_k(0)=p$ so $\tQ/x=E^0/g(0)=E^0/p$ so
 $\tQ/(x,u_1,\dotsc,u_{n-1})=\Fp$.  We now see that $\tQ$ is a
 complete local Noetherian ring of Krull dimension $n$ in which the
 maximal ideal can be generated by a sequence of length $n$.  It
 follows as in~\cite{ma:crt}*{Section 14} that the sequence
 $x,u_1,\dotsc,u_{n-1}$ is regular, and that $\tQ$ is a regular local
 ring.  In particular, it has unique factorisation, and so is
 integrally closed in its field of fractions.

 Next, the above discussion of Weierstrass degrees shows that
 $\tQ/\mxi\tQ=\Fp[x]/x^{p^kN_k}$.  On the other hand, as
 $q=1\pmod{p}$ we have $[q^i](x)=x\pmod{x^2}$ for all $i$, so $s$ is a
 unit multiple of $x^{p^k}$.  Now put
 \begin{align*}
   BX &= \{x^j\st 0\leq j<p^k\} \\
   BY &= \{s^m\st 0\leq m<N_k\} \\
   B  &= BY.BX = \{x^js^m\st 0\leq j<p^k,\; 0\leq m<N_k\}.
 \end{align*}
 We find that $B$ is a basis for $\tQ/\mxi\tQ$ over $\Fp$.  By
 Nakayama's Lemma, we see that $B$ is also a basis for $\tQ$ over $E^0$.  

 Now let $Q_0$ be the subring of $Q$ generated by $E^0$ and $s$.  As
 $B$ is a basis for $\tQ$ over $E^0$, we see that $BX$ generates
 $\tQ$ as a $Q_0$-module.

 Now let $K_0$, $K$ and $\tK$ be the fields of fractions of $Q_0$, $Q$
 and $\tQ$.  Put $\ov{\Gm}=\Gal(F(k)/F)=\ip{\phi\st\phi^{p^k}=1}$, so
 $|\ov{\Gm}|=p^k$ and $\ov{\Gm}$ acts on $\tQ$.  We
 then have $K=\tK^{\ov{\Gm}}$, so classical Galois theory of fields
 tells us that $[\tK:K]=p^k$.  On the other hand, as $BX$ generates
 $\tQ$ as a $Q_0$-module, we have $[K:K_0][\tK:K]=[\tK:K_0]\leq p^k$.
 For this to be consistent, we must have $K=K_0$, and the natural map
 $K\{BX\}\to\tK$ must be an isomorphism.  From this it follows that
 $Q\{BX\}\to\tQ$ is injective, but we saw previously that it is
 also surjective, so $BX$ is a basis for $\tQ$ over $Q$.  This
 allows us to identify the isomorphism $E^0\{B\}\to\tQ$ as a direct
 sum of copies (indexed by $BX$) of the map $E^0\{BY\}\to Q$.  It
 follows that the map $E^0\{BY\}\to Q$ is also an isomorphism, so
 $BY$ is a basis for $Q$ over $E^0$.  This also shows that $Q=Q_0$.
 By writing $s^{N_k}$ in terms of the basis $BY$, we obtain a
 Weierstrass polynomial $g_k(t)$ of degree $N_k$ such that
 $Q=E^0\psb{s}/g_k(s)$.   

 Now note that $s$ can be written as $h_k(x)$ for some series $h_k(t)$
 of Weierstrass degree $p^k$ with $h_k(0)=0$.  It follows that
 $g_k(h_k(t))$ and $f_k(t)$ both have Weierstrass degree $p^kN_k$.  As
 $g_k(s)=0$ we must have $g_k(h_k(x))=0$ in $E^0\psb{x}/f_k(x)$, so
 $g_k(h_k(t))$ is divisible by $f_k(t)$.  As both series have the same
 Weierstrass degree, we see that $g_k(h_k(t))$ is actually a unit
 multiple of $f_k(t)$.  By putting $t=0$, we see that $g_k(0)$ is a
 unit multiple of $f_k(0)=p$.  We can now repeat the argument that we
 gave for $\tQ$, and conclude that $Q$ is also a regular local ring.
\end{proof}

\begin{corollary}\lbl{cor-al-epi}
 The map $\al\:R\to Q$ is surjective, with kernel $J$, so
 $Q=R/J=\Ind_{p^k}(E^0(BGL_*(F)))$.  We therefore have
 $\dim_{E^0}(I)=\dim_{E^0}(Q)=N_k$. 
\end{corollary}
\begin{proof}
 Recall that $R$ is a quotient of $E^0\psb{c_1,\dotsc,c_{p^k}}$, and
 we previously defined $s_k=c_{p^k}$.  In
 $\tQ$ we find that $c_j$ maps to $(-1)^j$ times the $j$'th elementary
 symmetric function in $\{[q^i](x)\st 0\leq i<p^k\}$.  In particular,
 $s_k$ maps to $\pm s$, which makes it clear that $\al$ is
 surjective.  As $Q$ is a free module of rank $N_k$ over $E^0$, we see
 that $\ker(\al)$ must be a summand in $R$, with
 $\dim_{E^0}(\ker(\al))=\dim_{E^0}(R)-N_k$.  Next, it is clear by
 construction that $D'\ot_{E^0}Q=\Map(\Irr_{p^k}(\Tht^*),D')=R'/J'$,
 and it follows that $J\leq\ker(\al)$.  However, $J$ is also a
 summand, with $\dim_{E^0}(R/J)=\dim_{D'}(R'/J')=N_k$, so the relation
 $J\leq\ker(\al)$ implies that $J=\ker(\al)$.  We also noted in
 Proposition~\ref{prop-IJ} that $\dim_{E^0}(I)=\dim_{E^0}(R/J)$.
\end{proof}

\begin{corollary}\lbl{cor-bQ}
 The ring $\bQ=K^0\ot_{E^0}Q$ is just $K^0[s]/(s^{N_k})$, and $\al$
 induces an isomorphism $\bR/\bJ\to\bQ$.
\end{corollary}
\begin{proof}
 Clear.
\end{proof}

\begin{remark}\lbl{rem-Q-moduli}
 It is now clear that $\spf(Q)$ deserves to be thought of as the
 subscheme of irreducibles in $\Div_{p^k}^+(\HH)^\Gm$.  We pause to
 explain this in more detail.  In the case of a discrete group one
 would define the irreducibles to be the complement of the image of
 the addition map on nontrivial representations.  We regard formal
 schemes as a subcategory of functors from rings to sets, but in this
 context, the image of a morphism of formal schemes is not usually a
 formal scheme, and there are similar problems with complements.
 Nonetheless, we may be willing to regard certain subschemes as images
 or complements if they have weaker properties of the same general
 nature.  There are various different possible weakenings, and a
 complex web of implications between them.  One version is as
 follows.  Suppose we have morphisms
 \[ W \xra{f} Y \xra{i} X \xla{j} Z \]
 of formal schemes over $S$, such that $f^*\:\CO_Y\to\CO_W$ is a split
 monomorphism of $\CO_S$-modules, and $i^*\:\CO_X\to\CO_Y$ is a split
 epimorphism of $\CO_S$-modules.  This means that $i$ is a regular
 monomorphism of formal schemes, and that this property is stable under
 pullback along an arbitrary morphism $X'\to X$.  It also means that
 $f$ is an epimorphism of formal schemes, and that this property is
 stable under pullback along an arbitrary morphism $S'\to S$, although
 not in general under pullback along an arbitrary morphism $Y'\to Y$.  
 Thus, it is reasonable to regard the closed subscheme $Y$ as the
 image of $if$.  Now suppose that $j^*\:\CO_X\to\CO_Z$ is also a split
 epimorphism of $\CO_S$-modules.  If the combined map 
 $(i,j)\:Y\amalg Z\to X$ is an isomorphism (or equivalently, the
 combined map $(i^*,j^*)\:\CO_X\to\CO_Y\tm\CO_Z$ is an isomorphism of
 rings) then it is clearly appropriate to regard $Y$ and $Z$ as
 complements of each other.  However, this situation occurs rarely in
 examples arising from Morava $E$-theory of classifying spaces.  A
 more typical situation is that $(i^*,j^*)\:\CO_X\to\CO_Y\tm\CO_Z$ is
 a rational isomorphism, and that the ideals $\ker(i^*)$ and
 $\ker(j^*)$ are $\CO_S$-module summands that are annihilators of each
 other.  In this situation, it is natural to regard $Z$ as the best
 possible analog of $X\setminus Y$ in the category of formal schemes.
 As a special case of this, $\spf(Q)$ can be regarded as the
 complement of the image of the addition map 
 \[ \coprod_{i=0}^{p^k-1}
     \Div_i^+(\HH)^\Gm \tm_S \Div_{p^k-i}^+(\HH)^\Gm \tm_S \to 
      \Div_{p^k}^+(\HH)^\Gm,
 \]
 or in other words, as the subscheme of irreducible $\Gm$-invariant
 divisors of degree $p^k$.
\end{remark}

\begin{proposition}\lbl{prop-I-gen}
 The ideal $\bI$ is principal, and is a free module of rank one over
 $\bR/\bJ=\bQ$.  More precisely:
 \begin{itemize}
  \item[(a)] Any element $t\in\bI$ is annihilated by $\bJ$ and so
   induces a map $\mu_t\:\bQ=\bR/\bJ\to\bI$ sending $z+\bJ$ to $zt$.
  \item[(b)] The element $s_k^{N_k-1}t$ always lies in the socle of
   $\bR$, and $\mu_t$ is an isomorphism iff $s_k^{N_k-1}t\neq 0$. 
  \item[(c)] There exist elements $t\in I$ such that~(b) holds.
 \end{itemize}
 Moreover, the ideal $I$ is also principal, and is a free module of
 rank one over $R/J=Q$.
\end{proposition}
\begin{proof}
 Claim~(a) is clear from Lemma~\ref{lem-IJ-bar}.  Also,
 $\mu_t\:\bQ\to\bI$ is an isomorphism iff it is an epimorphism iff $t$
 generates $\bI$, by comparison of dimensions.

 Now put $J'=\bJ+\bR.s_k^{N_k-1}$.  As ideals containing $\bJ$ biject
 with ideals in $\bR/\bJ$, we see that $J'$ is the smallest ideal that
 is strictly larger than $\bJ$.  We also see that $J'/\bJ$ has
 dimension one over $K^0$, so it must be annihilated by the maximal
 ideal $\mxi$, so $\mxi.s_k^{N_k-1}\leq\bJ$, so
 $\mxi.s_k^{N_k-1}\bI\leq\bJ\bI=0$, so $s_k^{N_k-1}\bI$ is contained in
 the socle.  

 Next, as $\bR$ is a Frobenius algebra, the map $L\mapsto\ann(L)$ is
 an order-reversing permutation of the set of ideals.  Thus, for any
 $t\in\bI$ we have $\ann(\bR t)\geq\ann(\bI)=\bJ$, with equality iff
 $t$ generates $I$.  As $J'$ is the smallest ideal strictly exceeding
 $\bJ$, we see that $t$ generates $\bI$ iff $J't\neq 0$ iff
 $s_k^{N_k-1}t\neq 0$.  This proves~(b).  Moreover, as
 $s_k^{N_k-1}\not\in\bJ=\ann(\bI)$ we see that $s_k^{N_k-1}\bI\neq 0$;
 this proves~(c).

 Finally, Nakayama's Lemma tells us that if we lift any generator of
 $\bI$ to $I$, we get a generator of $I$.  As $\ann(I)=J$, it follows
 that the lifted element generates $I$ freely as a module over
 $R/J=Q$. 
\end{proof}

\begin{proposition}\lbl{prop-socle}
 $\bI$ is generated by $s_k^{\bN_{k-1}}$, and the socle of $\bR$ is
 generated by $s_k^{\bN_k-1}$ (so $s_k^{\bN_k-1}$ is a
 $K_0^\tm$-multiple of the standard socle generator
 $\tr_1^{GL_{p^k}(F)}(1)$). 
\end{proposition}
\begin{proof}
 In this section we have excluded the case $k=0$, but here we
 partially reinstate it.  The ideal $\bI$ is not defined, but we have
 $K^0(B\CV_{p^0})=K^0\psb{s_0}/[p^r](s_0)$ and
 $[p^r](s_0)=s_0^{p^{nr}}=s_0^{\bN_0}$, so the socle is generated by
 $s_0^{\bN_0-1}$ as claimed.

 Now suppose that $k>0$, and that we have already proved all claims
 for $\CV_{p^{k-1}}$.  Put
 $w=\res^G_H(s_k)=s_{k-1}^{\ot p}\in K^0(BH)$.  The induction
 hypothesis means that $w^{\bN_{k-1}-1}$ generates the socle of
 $K^0(BH)$, and that $w^{\bN_{k-1}}=0$.  This in turn means that the
 element $t=s_k^{\bN_{k-1}}$ has $\res^G_H(t)=0$, or in other words
 $t\in\bI$.  On the other hand, we have $s_k^{N_k-1}t=s_k^{\bN_k-1}$,
 which is nonzero by Proposition~\ref{prop-sk-height}.  It follows
 from Proposition~\ref{prop-I-gen} that $t$ generates $\bI$, and that
 $s_k^{\bN_k-1}$ is a nonzero element of the socle.  As $\bR$ is a
 Frobenius algebra, we see that the socle has dimension one over
 $K^0$, so it is generated by $s_k^{\bN_k-1}$.
\end{proof}

\section{Further relations in \texorpdfstring{$E$}{E}-theory}
\lbl{sec-relations}

In this section we will prove some additional interesting relations in
$E^0(B\CV)$.  In particular, we will give a formula for the socle
generator, in the sense of Definition~\ref{defn-soc}.

We first need to generalise the definition of the map $\phi$ given
by the Galois action.
\begin{definition}\lbl{defn-adams-op}
 For any $k\in\N$, we define $\psi^k\:\bF^\tm\to\bF^\tm$ by
 $\psi^k(u)=u^k$.  We also write $\psi^k$ for the induced self-map of
 $\HH=\spf(E^0(B(\bF^\tm)))$.  This in turn gives self-maps of
 $\Div_d^+(\HH)=\HH^d/\Sg_d$ and $\Div_d^+(\HH)^\Gm$, which we also
 denote by $\psi^k$.  We call these maps \emph{Adams operations}.
\end{definition}

\begin{remark}\lbl{rem-adams-lambda}
 The Adams operation $\psi^q$ is just the same as the operation $\phi$
 coming from $\Gm=\Gal(\bF/F)$.  For general $k\in\N$, we can use
 lambda operations (as in the original work of Adams) to define an
 operator $\psi^k$ on virtual representations, or a corresponding map
 $\Div(\HH)\to\Div(\HH)$ of schemes, and we can then check that this
 extends our definition on $\Div_d^+(\HH)$.  We do not need this so we
 will not give further details here, but they can be found
 in~\cite{st:cag}. 
\end{remark}

\begin{remark}\lbl{rem-adams-frob}
 The operation $\psi^p$ induces an endomorphism of the ring
 \[ K^0(B\bG_d) =\Fp\psb{c_1,\dotsc,c_d} 
     = \Fp\psb{x_1,\dotsc,x_d}^{\Sg_d}.
 \]
 This sends $x_i$ to $[p](x_i)=x_i^{p^{nv}}$, and it follows that
 $\psi^p$ is just the same as the Frobenius endomorphism, sending $a$
 to $a^{p^{nv}}$ for all $a$.
\end{remark}

\begin{remark}\lbl{rem-adams-ab}
 Any $\bF$-linear representation $V$ of an abelian group $A$ gives a
 divisor $D(V)$ defined over $\spf(E^0(BA))=\Hom(A^*,\HH)$.
 Specifically, we can write $V$ as a direct sum of one-dimensional
 representations corresponding to characters $\al_i\in A^*$, and we
 have a tautological homomorphism $\phi\:A^*\to\HH$ defined over
 $\Hom(A^*,\HH)$, giving the divisor $D(V)=\sum_i[\phi(\al_i)]$.  We
 can pull back the representation $V$ along the homomorphism
 $p^k.1_A\:A\to A$, and it is clear from the above discussion that
 $\psi^k(D(V))=D((k.1_A)^*V)$.  In particular, if the exponent of $A$
 divides $k$ we find that $\psi^k(D(V))=\dim(V).[0]$.
\end{remark}

If we want to do computer calculations of $K^0(BG_d)$, then we need
a formula for $[q](x)$, and then we need to do some manipulations with
symmetric functions based on that.  To make the calculation finite, we
need to truncate our power series at an appropriate level.  The
following result will help us to decide which level is appropriate.
(However, the numbers that emerge are very large, so we have had
limited success with explicit calculations.)

\begin{proposition}\lbl{prop-adams-null}
 Let $k$ be the largest integer such that $p^k\leq d$.  Then for
 $D\in\Div_d^+(\HH)^\Gm$ we have $\psi^{p^{k+r}}(D)=d[0]$.  Thus, for
 $u$ in the maximal ideal of $K^0(BG_d)$ we have $u^{p^{n(k+r)}}=0$.
\end{proposition}
\begin{proof}
 Remark~\ref{rem-adams-ab} and Proposition~\ref{prop-exponent} show
 that $\psi^{p^{k+r}}(D)$ becomes equal to $d[0]$ over $E^0(BA)$, for
 any abelian $p$-subgroup $A\leq G_d$.  Equivalently, the Chern
 classes $c_1,\dotsc,c_d$ of $\psi^{p^{k+r}}(D)$ map to zero in
 $E^0(BA)$.  However, generalised character theory tells us that the
 restriction maps to abelian subgroups are jointly injective, so 
 $\psi^{p^{k+r}}(D)=d[0]$ already over $E^0(BG_d)$.  On the other
 hand, Remark~\ref{rem-adams-frob} tells us that $\psi^{p^{k+r}}$ acts
 on $K^0(BG_d)$ as $u\mapsto u^{p^{n(k+r)}}$.  This operator must
 therefore kill the ideal generated by the Chern classes, which is the
 whole of the maximal ideal.
\end{proof}

\begin{definition}\lbl{defn-euler}
 We let $\euler\:\Div^+(\HH)\to\aff^1$ be the map given by $c_d$ on
 $\Div_d^+(\HH)$.  Equivalently, this is the unique map satisfying
 $\euler(0)=1$ and $\euler(D+E)=\euler(D)\euler(E)$ and
 $\euler([a])=x(a)$ for $a\in\HH$.  
\end{definition}

\begin{proposition}\lbl{prop-fix}
 There is a unique function $\fix\:\Div^+(\HH)\to\aff^1$ such that 
 \[ \euler(\psi^q(D)) = \fix(D)\,\euler(D) \]
 for all $D\in\Div^+(\HH)$.  Moreover, this is invertible and
 satisfies $\fix(0)=1$ and $\fix(D+E)=\fix(D)\fix(E)$ and
 $\fix([0])=q$.
\end{proposition}

\begin{remark}\lbl{rem-fix}
 We have defined $\fix(D)$ for all effective divisors $D$, in a way
 that implicitly depends on our choice of coordinate $x$.  We will
 primarily be interested in the case where $D$ is $\Gm$-invariant,
 and we will check later that it is independent of $x$ in that
 case. 

 The condition $\euler(\psi^q(D))=\fix(D)\,\euler(D)$ implies that
 $\fix(D)=1$ in any context where $D$ is $\Gm$-invariant and
 $\euler(D)$ is not a zero-divisor.  Of course this contrasts with the
 case $D=d[0]$, where $D=\psi^q(D)$ but $\euler(D)=0$ and
 $\fix(D)=q^d$.  One should think $\fix(D)$ as something like $q$
 raised to the power of the multiplicity of $[0]$ in $D$.
\end{remark}

\begin{remark}\lbl{rem-cannibalistic}
 In more traditional language, $\fix$ might be called a cannibalistic
 class and denoted by $\rho^q$.
\end{remark}

\begin{proof}[Proof of Proposition~\ref{prop-fix}]
 It is standard that $x(qa)=[q]_F(x(a))$ for some power series
 $[q]_F(t)$ of the form $qt+O(t^2)$.  We can thus write
 $[q]_F(t)=t\,\ip{q}_F(t)$ and $\fix_1(a)=\ip{q}_F(x(a))$ so
 $x(qa)=\fix_1(a)\,x(a)$ and $\fix_1(0)=q$.  Next, given a divisor $D$
 of degree $d$ on $\HH$ over a base scheme $T$ we have a norm map
 $N\:\CO_D\to\CO_T$ which we can apply to $\fix_1$ to get an element
 $\fix(D)\in\CO_T$.  Equivalently, we define
 $\fix\:\Div_d^+(\HH)=\HH^d/\Sg_d\to\aff^1$ to be the unique element
 such that $\fix(D)=\prod_i\fix_1(a_i)$ whenever
 $D=\sum_{i=1}^d[a_i]$.  It is clear that this satisfies 
 \[ \euler(\psi^q(D))=\euler\left(\sum_i[qa_i]\right)
     = \prod_i[q]_F(x(a_i)) 
     = \prod_i(\fix_1(a_i)x(a_i)) 
     = \fix(D)\euler(D).
 \] 
 As the element $c_d$ is not a zero divisor in the ring
 $\CO_{\Div_d^+(\HH)}=E^0(BGL_d(\bF))=E^0\psb{c_1,\dotsc,c_d}$, 
 the above equation characterises the map $\fix$ uniquely.  Note also
 that the element $\fix(d[0])=q^d$ is invertible in $E^0$ and all of
 $\Div_d^+(\HH)$ is infinitesimally close to $d[0]$ so $\fix$ is
 invertible on $\Div_d^+(\HH)$ as claimed.
\end{proof}

\begin{lemma}\lbl{lem-fix-adams}
 For $k\in\Zp$ with $k^{p-1}=1$ we have $\fix(\psi^kD)=\fix(D)$.  In
 particular, the divisor $\ov{D}=\psi^{-1}(D)$ has
 $\fix(\ov{D})=\fix(D)$. 
\end{lemma}
\begin{proof}
 Recall from Remark~\ref{rem-period} that for $k$ as above we have
 $x(ka)=[k]_F(x(a))=k\,x(a)$, and thus $x(kqa)=k\,x(qa)$.  It follows
 that for any $D\in\Div_d^+(\HH)$ we have
 $\euler(\psi^kD)=k^d\euler(D)$ and
 $\euler(\psi^q\psi^kD)=k^d\euler(\psi^qD)$.  By working in the
 universal case where $\euler(D)$ is not a zero divisor, we conclude
 that $\fix(D)=1$.
\end{proof}

\begin{definition}\lbl{defn-Fix}
 For any $V\in\Rep(\Tht^*)=[\Tht^*,\CV]$ we put
 \[ \Fix(V) = \{v\in V\st \tht.v=v \text{ for all } \tht\in\Tht^*\}.
 \]
\end{definition}

Note that this is a finite-dimensional vector space over $F$, so
$|\Fix(V)|$ is a power of $q$, and thus lies in $1+p^r\Z$.  Note also
that the element $\fix\in E^0(B\CV)$ gives rise to a map
$\Rep(\Tht^*)\to D'$, by generalised character theory.

\begin{proposition}\lbl{prop-fix-HKR}
 For any representation $V\in\Rep(\Tht^*)$, the character value
 $\fix(V)$ is just $|\Fix(V)|$.
\end{proposition}
\begin{proof}
 Because $\Fix(V\oplus W)=\Fix(V)\oplus\Fix(W)$ and
 $\fix(V\oplus W)=\fix(V)\fix(W)$ we can reduce to the case
 where $V$ is irreducible.  If $V$ is just $F$ (with trivial
 $\Tht^*$-action) then we have $\fix(V)=\ip{q}(0)=q=|\Fix(V)|$ as
 required.  Suppose instead that $V$ has nontrivial action, and let
 $L$ denote a one-dimensional subrepresentation of $\bF\ot_FV$.  We
 then see that $\bF\ot_FV\simeq\bigoplus_{i=0}^{k-1}L^{q^i}$ for some
 $k$ with $L^{q^k}\simeq L$.  Let $x$ be the image in $D'$ of the
 Euler class of $L$; we then get
 \[ \fix(V) = \prod_i\ip{q}([q^i](x)) \]
 As $L$ is nontrivial we see that $x\neq 0$.  As $D_A$ is an integral
 domain it will be harmless to work in the field of fractions where we
 have $\ip{q}([q^i](x))=[q^{i+1}](x)/[q^i](x)$.  The whole product
 therefore cancels to give $\fix(V)=1$, which is the same as
 $|\Fix(V)|$, as required.
\end{proof}

\begin{proposition}\lbl{prop-fix-soc}
 The element $s=\prod_{i=0}^{d-1}(\fix-q^i)\in E^0(B\CV_d)$ is the
 transfer of $1$ from the trivial groupoid, and so is the standard
 generator of the socle.
\end{proposition}
\begin{proof}
 Consider a representation $V\in\Rep(\Tht^*)$.  We then have
 $\fix(V)=q^{\dim(\Fix(V))}$ with $0\leq\dim(\Fix(V))\leq d$.  It
 follows that $s(V)=0$ iff $\dim(\Fix(V))<d$ iff $A$ acts
 nontrivially on $V$.  On the other hand, if $A$ acts trivially we
 have $s(V)=\prod_{i=0}^{d-1}(q^d-q^i)$, which is well-known to be
 the same as the order of $GL_d(F)$.  This means that $s$ has the same
 character values as the transfer of $1$, so it is the same as the
 transfer of $1$.
\end{proof}

\begin{corollary}\lbl{cor-soc-fix}
 The socle of $K^0(B\CV_d)$ is generated by $(\fix-1)^d$.  Thus, if
 $d=p^k$ then the socle is generated by $(\fix-1)^{p^k}=\fix^{p^k}-1$.
\end{corollary}
\begin{proof}
 This follows because $q^i=1\pmod{p}$ for all $i$.
\end{proof}

\begin{remark}\lbl{rem-soc-soc}
 By comparing Corollary~\ref{cor-soc-fix} with
 Proposition~\ref{prop-socle}, we see that $(\fix-1)^{p^k}$ must be a
 unit multiple of $s_k^{\bN_k-1}=\euler^{\bN_k-1}$ in
 $K^0(B\CV_{p^k})$.  We have not yet been able to find a direct,
 equational proof of this, but we have verified it by strenuous
 computer calculation in some very small cases.
\end{remark}

Recall that the socle is by definition the annihilator of the kernel
of the restriction map from $E^0(BGL_d(F))$ to
$E^0(\text{point})=E^0(BGL_0(F))$.  We are also interested in the
kernel of restriction to $E^0(BGL_{d-1}(F))$.

\begin{proposition}\lbl{prop-fix-transfer}
 If we define $\sg\:\CV\to\CV$ by $\sg(V)=V\oplus F$, then
 \begin{itemize}
  \item[(a)] The map $\sg^*\:E^0(B\CV)\to E^0(B\CV)$ is surjective,
   with kernel generated by the element $\euler\in E^0(B\CV)$.
  \item[(b)] The element $\fix-1\in E^0(B\CV)$ is a unit multiple of
   $\sg_!(1)$.
  \item[(c)] The ideal $\ann(\euler)=\ann(\ker(\sg^*))$ is generated
   by $\fix-1$.
 \end{itemize}
\end{proposition}

We will deduce this from a more elaborate statement involving some
auxiliary groupoids, which we now introduce.

\begin{definition}\lbl{defn-fix-transfer}
 We will construct a diagram of groupoids of the following shape:
 \begin{center}
  \begin{tikzcd}[sep=large]
   \CV \arrow[d,equal] &
   \hCV_1 
    \arrow[l,twoheadrightarrow,"\al"']
    \arrow[d,twoheadrightarrow,"\pi_1"'] 
    \arrow[r,rightarrowtail,"\text{inc}"] &
   \hCV 
    \arrow[dl,twoheadrightarrow,"\pi"] \\
   \CV
    \arrow[ur,rightarrowtail,"\bt"] 
    \arrow[r,rightarrowtail,"\sg"'] &
   \CV &
   \hCV_0
    \arrow[l,"\pi_0","\simeq"']
    \arrow[u,rightarrowtail,"\text{inc}"']
  \end{tikzcd}
 \end{center}
 \begin{itemize}
  \item The groupoid $\hCV$ consists of pairs $(V,v)$ with $V\in\CV$
   and $v\in V$.  The morphisms from $(V_0,v_0)$ to $(V_1,v_1)$ are
   $F$-linear isomorphisms $f\:V_0\to V_1$ with $f(v_0)=v_1$.  
   This splits as the disjoint union of subgroupoids
   $\hCV_0=\{(V,v)\st v=0\}$ and $\hCV_1=\{(V,v)\st v\neq 0\}$.
  \item The functor $\pi\:\hCV\to\CV$ is given by $\pi(V,v)=V$, and
   $\pi_0$ and $\pi_1$ are just restrictions of $\pi$.
  \item The remaining functors are $\sg(V)=V\oplus F$ and
   $\bt(V)=(V\oplus F,(0,1))$ and $\al(V,v)=V/Fv$.
 \end{itemize}
\end{definition}
\begin{remark}
 It is convenient to start with $\hCV$, because it has an evident
 symmetric monoidal direct sum operation making $\pi$ a symmetric
 monoidal functor.  However, we are primarily interested in the
 subgroupoid $\hCV_1$.
\end{remark}

Proposition~\ref{prop-fix-transfer} clearly follows from the following
extended version.

\begin{proposition}\lbl{prop-fix-transfer-extra}\leavevmode
 \begin{itemize}
  \item[(a)] The functor $\pi$ is a covering, and the restriction
   $\pi_0$ is an isomorphism of groupoids.
  \item[(b)] The functors $\al$ and $\bt$ are not equivalences, but
   nonetheless they give mutually inverse isomorphisms in Morava
   $E$-theory.
  \item[(c)] The maps $\pi_1^*\:E^0(B\CV)\to E^0(B\hCV_1)$ and
   $\sg^*\:E^0(B\CV)\to E^0(B\CV)$ are surjective, with kernel
   generated by $\euler$.
  \item[(d)] The element $\fix-1$ is equal to $(\pi_1)_!(1)$, and is a
   unit multiple of $\sg_!(1)$.
  \item[(e)] The ideal
   $\ann(\euler)=\ann(\ker(\sg^*))=\ann(\ker((\pi_1)^*))$ is generated 
   by $\sg_!(1)$ or by $\fix-1=(\pi_1)_!(1)$.
 \end{itemize}
\end{proposition}
\begin{proof}\leavevmode
 \begin{itemize}
  \item[(a)] If we have an object $(V,v)\in\hCV$ and a morphism
   $f\:V\to W$ in $\CV$ then $f$ counts as a morphism
   $(V,v)\to(W,f(v))$ in $\hCV$, and $f(v)$ is the unique element of
   $W$ such that this works.  This means that $\pi$ is a covering.  It
   is clear that $\pi_0$ is an isomorphism, with inverse
   $V\mapsto(V,0)$.
  \item[(b)] It is clear that $\al\bt\simeq 1$, and that $\al$ and
   $\bt$ give mutually inverse bijections between $\pi_0(\CV)$ and
   $\pi_0(\hCV)$.  Moreover, the map  
   \[ \bt\:\Aut_{\hCV}(V,v) \to \Aut_{\CV}(V/Fv) \]
   is easily seen to be surjective, with kernel of order coprime to
   $p$.  It follows by Proposition~\ref{prop-quot-equiv} that $\al$
   and $\bt$ give mutually inverse isomorphisms in Morava $E$-theory.
  \item[(c)] In view of~(b) it will suffice to show that $\sg^*$ has
   the stated properties.  Put $r_i=\phi^*(c_i)-c_i$ as usual, and
   \begin{align*}
    \ov{R}_d &= E^0(B\bCV_d) = E^0\psb{c_1,\dotsc,c_d} \\
    L_d &= (r_1,\dotsc,r_d) \leq \ov{R}_d \\
    R_d &= E^0(B\CV_d) = \ov{R}_d/L_d.
   \end{align*}
   The restriction map $\ov{R}_d\to\ov{R}_{d-1}$ sends $c_d$ to
   $0$ and $c_i$ to $c_i$ for $i<d$, so it is surjective.  It is
   compatible with the action of $\phi^*$, so it sends $r_d$ to $0$
   and $r_i$ to $r_i$ for $i<d$.  It follows that the induced map
   $L_d\to L_{d-1}$ is also surjective, as is the induced map
   $R_d\to R_{d-1}$.  A diagram chase therefore shows that the map
   \[ \ov{R}_d.c_d = \ker(\sg^*\:\ov{R}_d\to\ov{R}_{d-1}) \to
        \ker(\sg^*\:R_d\to R_{d-1})
   \]
   is again surjective, so $\ker(\sg^*\:R_d\to R_{d-1})$ is generated
   by $c_d$.  Moreover, $c_d$ is the same as $\euler$ on $\CV_d$, so
   claim~(c) follows. 
  \item[(d)] Consider a representation $V\in\Rep_d(\Tht^*)$.  Because
   $\pi_1$ is a covering, we see that $(\pi_1)_!(1)(V)$ is just the
   number of objects $(V,v)\in[\Tht^*,\hCV_1]$ lying over $V$, or in
   other words the number of nonzero $\Tht^*$-fixed points in $V$,
   which is $(\fix-1)(V)$.  As the generalised character map is
   injective, it follows that $\fix-1=(\pi_1)_!(1)$.

   Next, recall that $\al_!(1).\al^*(\bt_!(1))=\al_!(\bt_!(1))=1$, so
   $\al_!(1)$ is invertible in $E^0(B\CV)$.  We have seen that $\sg^*$
   is surjective, so we can choose $z$ with $\sg^*(z)=\al_!(1)$.  By
   considering the restriction to the spine, we see that $z$ is also
   invertible.  As $\bt_!$ is inverse to $\al_!$ and $\sg=\pi_1\bt$ we
   have $(\pi_1)_!=\sg_!\al_!$.  It follows that
   \[ z\,\sg_!(1)=\sg_!(\sg^*(z))=\sg_!\al_!(1)=(\pi_1)_!(1)=\fix-1,
   \]
   so $\sg_!(1)$ is a unit multiple of $\fix-1$ as claimed.
  \item[(e)] Claim~(e) now follows from Corollary~\ref{cor-res-tr}.
 \end{itemize}
\end{proof}

\begin{definition}\lbl{defn-hom}
 We define $\hom\:\Div^+(\HH)\tm\Div^+(\HH)\to\aff^1$ by 
 $\hom(D,E)=\fix(\ov{D}*E)$.  Here the map $D\mapsto\ov{D}$ is
 induced by $-1\:\HH\to\HH$, and $*$ denotes convolution of divisors,
 so
 \[ \hom\left(\sum_i[a_i],\sum_j[b_j]\right) =
      \fix\left(\sum_{i,j}[b_j-a_i]\right).
 \]
\end{definition}
\begin{remark}
 It is easy to see that this satisfies $\hom(D,E)=\hom(E,D)$ and 
 \[ \hom(D_0+D_1,E_0+E_1) = 
    \hom(D_0,E_0)\hom(D_0,E_1)\hom(D_1,E_0)\hom(D_1,E_1).
 \]
 In other words, it is a symmetric biexponential function.
\end{remark}

Using Proposition~\ref{prop-fix-transfer}, we can also describe the
element $\hom\in E^0(B\CV^2)$ as a transfer.  To explain the details,
we need some more auxiliary groupoids and functors.

\begin{definition}\lbl{defn-hom-aux}\leavevmode
 \begin{itemize}
  \item[(a)] Let $\CS$ be the groupoid of pairs $(U,W)$ with $W\in\CV$
   and $U\leq W$.  We will write such pairs as $(U\leq W)$.
  \item[(b)] Let $\CH$ be the groupoid of triples $(U,V,m)$ with
   $U,V\in\CV$ and $m\in\Hom_F(V,U)$.
  \item[(c)] Define functors as follows:
   \begin{center}
    \begin{tikzcd}[sep=2cm]
     \hCV \arrow[d,"\pi"'] & 
     \CH \arrow[l,"\hxi"'] \arrow[d,"\zt"'] \arrow[r,"\zt"] &
     \CV^2 \arrow[d,shift right,"\tau"'] \arrow[dr,"\sg"] \\
     \CV &
     \CV^2 \arrow[l,"\xi"] \arrow[r,"\tau"'] &
     \CS
      \arrow[u,shift right,"\al"']
      \arrow[r,shift left,"\rho"]
      \arrow[r,shift right,"\bt"'] &
     \CV
    \end{tikzcd}
   \end{center}
   \begin{align*}
    \zt(U,V,m) &= (U,V) &
    \hxi(U,V,m) &= (\Hom_F(V,U),m) \\
    \pi(V,v) &= V &
    \xi(U,V) &= \Hom_F(V,U) \\
    \al(U\leq W) &= (U,W/U) &
    \bt(U\leq W) &= W \\
    \tau(U,V) &= (U\leq U\oplus V) &
    \rho(U\leq W) &= U\oplus W/U \\
    \sg(U,V) &= U\oplus V. 
   \end{align*}
  \item[(d)] Given $(U,V,m)\in\CH$, define $n\in\Aut(U\oplus V)$ by
   $n(u,v)=(u+m(v),v)$.  This gives a natural automorphism of the
   functor $\tau\zt$.  
 \end{itemize}
\end{definition}

\begin{proposition}\lbl{prop-hom-aux}\leavevmode
 \begin{itemize}
  \item[(a)] $L\tau$ and $L\al$ are equivalences, and are inverse to
   each other.
  \item[(b)] $L\bt=L\rho=L\sg\,L\al\:L\CS\to L\CV$.
  \item[(c)] $R\al$ and $R\tau$ are equivalences, and are inverse to
   each other.
  \item[(d)] $R\bt=R\rho=R\al\,R\sg\:L\CV\to L\CS$.
  \item[(e)] The element $\hom\in E^0(B\CV^2)$ is the same as
   $\zt_!(1)$ or $\tau^*\tau_!(1)$ or $\al_!(1)^{-1}$.
 \end{itemize}
\end{proposition}
The message here is that in many contexts, it is harmless to assume
that any short exact sequence of $F$-vectors spaces comes equipped
with a specified splitting.
\begin{proof}
 First, it is easy to see that the functor $\pi$ is a covering and the
 left square is a pullback so it is also a homotopy pullback.  This
 implies that 
 \[ \zt_!(1) = \zt_!\hxi^*(1) = \xi^*\pi_!(1) = \xi^*(\fix) 
     = \hom.
 \]

 Next, we claim that the middle square is a homotopy pullback.
 Indeed, we certainly have a homotopy pullback square
 \begin{center}
  \begin{tikzcd}
   \CH'  \arrow[r,"\zt_0"] \arrow[d,"\zt_1"'] &
   \CV^2 \arrow[d,"\tau"] \\
   \CV^2 \arrow[r,"\tau"'] &
   \CS,
  \end{tikzcd}
 \end{center} 
 where
 \[ \CH' = \{(U_0,V_0,U_0,V_1,k) \st U_i,V_j\in\CV,\;
      k\:\tau(U_0,V_0)\xra{\simeq}\tau(U_1,V_1) \}
 \]
 and $\zt_i(U_0,V_0,U_1,V_1,k)=(U_i,V_i)$.  We can define
 $\phi\:\CH\to\CH'$ by $\phi(U,V,m)=(U,V,U,V,n)$, where $n$ is as in
 Definition~\ref{defn-hom-aux}(d).  In the opposite direction, if
 $(U_0,V_0,U_1,V_1,k)\in\CH'$ then $k$ is an isomorphism
 $U_0\oplus V_0\to U_1\oplus V_1$ that sends $U_0$ to $U_1$, so it can
 be decomposed into components $k_U:U_0\xra{\simeq}U_1$ and
 $k_V\:V_0\xra{\simeq}V_1$ and $k'\:V_0\to U_1$.  We can thus define
 $\psi(U_0,V_0,U_1,V_1,k)=(U_0,V_0,k_U^{-1}\circ k')$, and check that
 this gives a functor $\psi\:\CH'\to\CH$ that is inverse to $\phi$.
 We also have $\zt_i\phi=\zt$ for $i=0,1$.  It follows that the middle
 square is a homotopy pullback as claimed, and thus that
 $\zt_!(1)=\zt_!\zt^*(1)=\tau^*\tau_!(1)$, which proves most of~(e).  

 It is also easy to see that $\al\tau\simeq 1$ and $\rho=\sg\al$ and
 $\bt\tau=\sg$.  Next, we see that there is a split extension
 \begin{center}
  \begin{tikzcd}
   \Hom(V,U)
     \arrow[r,rightarrowtail] &
   \Aut(U\leq U\oplus V)
     \arrow[r,twoheadrightarrow,shift right] &
   \Aut(U)\tm\Aut(V)
     \arrow[l,shift right]
  \end{tikzcd}
 \end{center}
 Here $\Hom(V,U)$ is a vector space over $F$ and so has order coprime to
 $p$.  We can thus apply Proposition~\ref{prop-inc-split} to see that
 $L\tau$ is an equivalence, or we can apply
 Proposition~\ref{prop-quot-equiv} to see that $L\al$ is an
 equivalence.  As $\al\tau=1$, we see that $L\al$ and $L_\tau$ are
 inverse to each other.  As $\rho=\sg\al$ we have $L\rho=L\sg L\al$.
 As $\bt\tau=\sg$ we have $L\bt\,L\tau=L\sg$, and we can compose with
 $L\al$ to get $L\bt=L\sg\,L\al=L\rho$.  We have now proved~(a)
 and~(b); claims~(c) and~(d) follow by taking adjoints.  

 To complete the proof of~(e), we note that
 \[ \al_!(1).\hom = \al_!(\al^*(\hom)) =
     \al_!\al^*\tau^*\tau_!(1).
 \]
 We have seen that $\al^*$ is inverse to $\tau^*$, and $\al_!$ is
 inverse to $\tau_!$, so this simplifies to $\al_!(1).\hom=1$ as
 required. 
\end{proof}

\section{General theory of twisted products}
\lbl{sec-twist}

We have used the direct sum functor to make $E^0(B\CV)$ into a
bialgebra.  There is another natural way to make $E^0(B\CV)$ into a
bialgebra, using short exact sequences without a specified splitting.
This is related to the theory of Harish-Chandra induction in
representation theory.  We have also mentioned that the coproduct map
$E^0(B\CV)\to E^0(B\CV^2)$ is only a ring map if we modify the obvious
ring structure on the codomain.  Thus, we have various modified
products to consider.  It will turn out that they can all be
encompassed by a general framework that we will discuss in this
section. 

Let $\CA$ be a graded groupoid of finite type.  (We will primarily be
interested in the case $\CA=\CV$, and the cases $\CA=\CV^k$ will
appear for auxiliary reasons; but is is convenient to proceed more
generally.)  We assume that we have an object $0\in\CA$ and a functor
$\sigma\:\CA^2\to\CA$ that give a symmetric monoidal structure and are
compatible with the grading. We also assume that $\CA_0$ is the
category with only one object (namely $0$), and only the identity
morphism.   We write $U\oplus V$ for $\sg(U,V)$.  We also put
$R_k=E^0(B\CA_k)$. 

As with $\CV$, we have a diagonal product $\dl^*\:R_k\ot R_k\to R_k$,
a convolution product $\sg_!\:R_i\ot R_j\to R_{i+j}$, and a coproduct
$\sg^*\:R_k\to\bigoplus_{k=i+j}R_i\ot R_j$.  We will write $f\tm g$
for $\sg_!(f\ot g)$.  We will usually write $fg$ for $\dl^*(f\ot g)$,
but we may write $f\bl g$ when necessary to emphasise the
difference from $f\tm g$.  We say that an $E^0$-submodule $U\leq R$ is
a $\bl$-ideal if it satisfies $R\bl U\leq U$.  

We write $M_k=\spf(R_k)$ and $M=\coprod_kM_k$, so $M$ is a commutative
monoid in the category of formal schemes.  We will use
scheme-theoretic language, and treat elements of $R_k$ as functions on
$M_k$.  We will use integral notation for the standard inner product
on $R_k$, as discussed in Section~\ref{sec-morava}.  We will say that
a function $f$ on $M$ is \emph{finitely supported} if it vanishes on
$M_k$ for $k\gg 0$, or equivalently it lies in $\bigoplus_kR_k$; this
insures that $\int_Mf$ is defined.

For example, the convolution product is characterised by the identity
\[ \int (f\tm g)(m)u(m)\,dm =
      \iint f(m_1)g(m_2)u(m_1+m_2)\,dm_1\,dm_2
\]
for all finitely supported functions $u$.  The unit for the
convolution product is the function $[0]$ which is $1$ on $\CA_0$ and
$0$ on $\CA_k$ for $k>0$.  The coproduct is given by
$(\sg^*f)(m,n)=f(m+n)$, and the counit is $f\mapsto f(0)$.

\begin{definition}\lbl{defn-twisted-product}
 Suppose we have an element $t\in (R\hot R)^\tm$, or in other
 words, an invertible function on $M^2$.  We define a twisted
 convolution product on $R$ by the rule
 \[ f\tm_t g = \sg_!(t.(f\ot g)). \]
 This is characterised by 
 \[ \int (f\tm_tg)(m)u(m)\,dm =
     \iint f(m_1)g(m_2)t(m_1,m_2)u(m_1+m_2)\,dm_1\,dm_2
 \]
 for all finitely supported functions $u$ on $M$.   
\end{definition}

\begin{lemma}\lbl{lem-twist-alt}
 Suppose that $t$ is biexponential, in the sense that
 $t(0,m)=t(m,0)=1$ and $t(m+n,p)=t(m,p)t(n,p)$ and
 $t(m,n+p)=t(m,n)t(m,p)$ for all $n,m,p$.  Then the product $\tm_t$ is
 associative, with $[0]$ as a two sided unit.  Moreover, the induced
 $k$-fold product is given by
 \[ f_1 \tm_t \dotsb \tm_t f_k = \mu(t^{(k)}.(f_1\ot\dotsb\ot f_k)), \]
 where
 \[ t^{(k)}(m_1,\dotsc,m_k) = \prod_{1\leq i<j\leq k} t(m_i,m_j). \]
 Moreover, the product $\tm_t$ is commutative if $t$ is symmetric (in
 the sense that $t(m,n)=t(n,m)$).
\end{lemma}
\begin{proof}
 First, we have
 \[ \int ([0]\tm_tf)(m) u(m)\,dm = 
    \iint [0](m_1)f(m_2)t(m_1,m_2)u(m_1+m_2)\,dn\,dm.
 \]
 Here $[0]$ is the characteristic function of $M_0=\{0\}$, so this
 reduces to $\int f(m)t(0,m)u(m)\,dm$.  As $t(0,m)=1$ this is just
 $\int f(m)u(m)\,dm$.  By the perfectness of the inner product, we
 must therefore have $[0]\tm_tf=f$.  Essentially the same argument gives
 $f\tm_t[0]=f$.  Next, we have
 \begin{gather*}
  \int (f\tm_t(g\tm_th))(m)u(m)\,dm 
   = \iint f(m_1)(g\tm_th)(m_{23})t(m_1,m_{23})
            u(m_1+m_{23})\,dm_1dm_{23} \\ 
   = \iiint f(m_1)g(m_2)h(m_3) t(m_2,m_3)t(m_1,m_2+m_3)
             u(m_1+m_2+m_3) dm_1 dm_2 dm_3 \\
  \int ((f\tm_tg)\tm_th)(m)u(m)\,dm 
   = \iint (f\tm_tg)(m_{12})h(m_3)t(m_{12},m_3)
            u(m_{12}+m_3)\,dm_{12}dm_3 \\ 
   = \iiint f(m_1)g(m_2)h(m_3) t(m_1,m_2)t(m_1+m_2,m_3)
             u(m_1+m_2+m_3) dm_1 dm_2 dm_3.
 \end{gather*}
 The biexponential property gives
 \[ t(m_2,m_3)t(m_1,m_2+m_3) = t^{(3)}(m_1,m_2,m_3) = 
    t(m_1,m_2)t(m_1+m_2,m_3),
 \]
 and it follows that 
 \[ f\tm_t(g\tm_th) = \mu(t^{(3)}.(f\ot g\ot h)) = (f\tm_tg)\tm_t h. \]
 This can be extended inductively for products of more than three
 factors.  The commutativity statement is clear.
\end{proof}

We will assume from now on that our twisting function $t\in R\hot R$
is biexponential.

\begin{definition}
 Recall that we can regard $\pi_0(\CA)$ as a discrete subcategory of
 $\CA$, which we call the spine.  With respect to the usual topology
 on $E^0(B\CA)$, an element $f\in E^0(B\CA)$ is topologically
 nilpotent iff $f$ maps to zero in
 $K^0(B\pi_0(\CA))=\Map(\pi_0(\CA),K^0)$.   We also say that $f$ is a
 \emph{strong unit} if $f-1$ is topologically nilpotent.
\end{definition}

We next prove some results showing that the difference between $\tm$
and $\tm_t$ is irrelevant in various contexts.

\begin{lemma}\lbl{lem-twisted-ideal-product}
 If $U$ and $V$ are $\bl$-ideals in $R$ then $U\tm V$ is the same as
 $U\tm_tV$, and this is again a $\bl$-ideal.
\end{lemma}
\begin{proof}
 It is clear that $U\hot V$ is a $\bl$-ideal in $E^0(B\CA^2)=R\hot R$,
 and $t$ is a unit in $R\hot R$ so $U\hot V = t\bl(U\hot V)$.  By
 applying $\sg_!$ we deduce that $U\tm V=U\tm_tV$.  Also, if $f\in R$
 then $f\bl(U\tm V)=f\bl\sg_!(U\hot V)=\sg_!(\sg^*(f)\bl(U\hot V))$.
 As $U\hot V$ is a $\bl$-ideal, this lies in $\sg_!(U\hot V)=U\tm V$,
 as claimed. 
\end{proof}

\begin{corollary}\leavevmode
 \begin{itemize}
  \item[(a)] If $U$ is a $\bl$-ideal in $R$, then it is a
   $\tm$-subring iff it is a $\tm_t$-subring.
  \item[(b)] Suppose that~(a) holds, and that $V\leq U$, and that $V$
   is also a $\bl$-ideal.  Then $V$ is a $\tm$-ideal in $U$ iff it is a
   $\tm_t$-ideal.
  \item[(c)] Suppose that~(b) holds, so we have an induced $\tm$-product
   and an induced $\tm_t$-product on $U/V$, and also an induced
   $\bl$-product making $U/V$ into a module over $R$.  Suppose also
   that $t$ is a strong unit, and that $U/V$ is annihilated by
   topologically nilpotent elements.  Then the $\tm$-product and the
   $\tm_t$-product on $U/V$ are the same.
 \end{itemize}
\end{corollary}
\begin{proof}\leavevmode
 \begin{itemize}
  \item[(a)] The lemma gives $U\tm U=U\tm_tU$.
  \item[(b)] The lemma gives $U\tm V=U\tm_tV$.
  \item[(c)] Let $J$ be the ideal of topologically nilpotent elements
   in $R$, so we have $J\bl U\leq V$ by assumption.  Now put
   $V'=V\hot U+U\hot V$, so that $(U/V)^{\ot 2}=U^{\ot 2}/V'$.  
   The ideal of topologically nilpotent elements of $R\hot R$ is 
   $J'=J\hot R+R\hot J$, so $J'.U^{\ot 2}\leq V'$.  By assumption we
   have $t-1\in J'$, so for $f,g\in U$ we have
   $t\bl(f\ot g)\in (f\ot g)+V'$.  Applying $\sg_!$ to this gives
   $f\tm_tg\in f\tm g+V$, as required.
 \end{itemize}
\end{proof}

\begin{definition}\lbl{defn-quad-twist}
 We say that an element $u\in R^\tm$ is \emph{exponentially quadratic}
 if the map
 \[ (\dl u)(m_0,m_1)=u(m_0+m_1)u(m_0)^{-1}u(m_1)^{-1} \]
 is biexponential.
\end{definition}

\begin{proposition}\lbl{prop-quad-twist}
  Let $u$ be exponentially quadratic, and let $t$ be biexponential,
  and put $w=\dl(u)\bl t$.  Then the map $f\mapsto u\bl f$ gives an
  isomorphism $(R,\tm_t)\to (R,\tm_w)$.
\end{proposition}
\begin{proof}
 The condition $w=\dl(u)\bl t$ can be rearranged as
 $w\bl (u\ot u)=\sg^*(u)\bl t$.  This gives 
 \begin{align*}
      (u\bl f)\tm_w(u\bl g)
   &= \sg_!(w\bl (u\ot u)\bl (f\ot g))
    = \sg_!(\sg^*(u)\bl t\bl (f\ot g)) \\
   &= u\bl \sg_!(t\bl (f\ot g))
    = u\bl(f\tm_tg).
 \end{align*}
\end{proof}

\begin{corollary}\lbl{cor-root-untwist}
 Suppose that $t=r^2$ for some symmetric biexponential function $r$.
 Then $(R,\tm_t)\simeq(R,\tm)$.
\end{corollary}
\begin{proof}
 The function $u(m)=r(m,m)$ is exponentially quadratic with
 $\dl(u)=t$. 
\end{proof}

\begin{corollary}\lbl{cor-strong-unit-untwist}
 If $t$ is a symmetric biexponential function and is a strong unit,
 then $(R,\tm_t)\simeq(R,\tm)$.
\end{corollary}
\begin{proof}
 We have assumed that $p>2$, so the squaring map is an automorphism of
 the group of strong units.  If we let $r$ denote the unique strong
 unit with $r^2=t$, it follows easily that $r$ is again symmetric and
 biexponential, so we can use Corollary~\ref{cor-root-untwist}.
\end{proof}

\begin{proposition}\lbl{prop-twisted-hom}
 Suppose that $(\CA,t)$ is as above, and $(\CA',t')$ is of the same
 type.  Let $\phi\:\CA'\to\CA$ be a symmetric monoidal functor, so we
 can build a homotopy-commutative diagram as follows, in which the
 bottom right region is a homotopy pullback.
 \begin{center}
  \begin{tikzcd}
   (\CA')^2
    \arrow[rr,"\phi^2"]
    \arrow[dd,"\sg'"']
    \arrow[dr,"\kp"'] &&
   \CA^2
    \arrow[dd,"\sg"] \\ &
   \CP
    \arrow[ur,"\tphi"']
    \arrow[dl,"\tsg"] \\
   \CA'
    \arrow[rr,"\phi"'] &&
   \CA
  \end{tikzcd}
 \end{center}
 Suppose also that $\kp_!(t')=\tphi^*(t)\in E^0(B\CP)$.  Then $\phi^*$
 gives a ring map $(R,\tm_t)\to(R',\tm_{t'})$.  
\end{proposition}
\begin{proof}
 Consider elements $f_0,f_1\in R$ and put $f=f_0\ot f_1$.  We have
 \begin{align*}
      \phi^*(f_0\tm_tf_1)
   &= \phi^*\sg_!(t\bl f)
    = \tsg_!\tphi^*(t\bl f)
    = \tsg_!(\tphi^*(t)\bl\tphi^*(f)) \\
   &= \tsg_!(\kp_!(t')\bl\tphi^*(f))
    = \tsg_!(\kp_!(t'\bl\kp^*\tphi^*(f)))
    = (\tsg\kp)_!(t'\bl((\tphi\kp)^*(f))) \\
   &= \sg'_!(t'\bl(\phi^2)^*(f))
     = \sg'_!(t'\bl(\phi^*(f_0)\ot\phi^*(f_1)))
     = \phi^*(f_0)\tm_{t'}\phi^*(f_1).
 \end{align*}
\end{proof}

\section{The Harish-Chandra (co)product}
\lbl{sec-HC}

We have used the direct sum functor to make $E^0(B\CV)$ into a
bialgebra.  There is another natural way to make $E^0(B\CV)$ into a
bialgebra, related to the theory of Harish-Chandra induction in
representation theory, which we will explain in this section.
However, our main conclusion will be that the Harish-Chandra structure
is very closely related to our original structure.

\begin{definition}\lbl{defn-HC-chi}
 We define $\chi\:\CV\to\CV$ by $\chi(V)=V^*$ on objects, and
 $\chi(u)=(u^*)^{-1}$ on morphisms. 
\end{definition}

\begin{definition}\lbl{defn-HC-mu}
 We define $\mu\:L\CV^{(2)}\to L\CV$ by $\mu=\bt\al^!$, where $\al$
 and $\bt$ are as in Definition~\ref{defn-hom-aux}.  Dually, we define
 $\nu=\al\bt^!\:L\CV\to L\CV^{(2)}$.
\end{definition}

\begin{proposition}\lbl{prop-HC-assoc}
 $\mu$ is an associative and unital product on $L\CV$, with $\chi$ as
 an anti-involution.  Dually, $\nu$ is a coassociative and counital
 coproduct. 
\end{proposition}
\begin{proof}
 Let $\CS(r)$ be the groupoid of tuples $(V_1\leq\dotsb\leq V_r)$ in
 $\CV$.  Define functors $\CV^r\xla{\al(r)}\CS(r)\xra{\bt(r)}\CV$ by
 \begin{align*}
   \al(r)(V_1\leq\dotsb\leq V_r) &= (V_1,V_2/V_1,\dotsc,V_r/V_{r-1}) \\
   \bt(r)(V_1\leq\dotsb\leq V_r) &= V_r.
 \end{align*}
 Then put $\mu(r)=\bt(r)\al(r)^!\:L\CV^{(r)}\to L\CV$.  In terms of
 our earlier definitions, we have $\CS(2)=\CS$ and $\al(2)=\al$ and
 $\bt(2)=\bt$ and $\mu(2)=\mu$.  We claim that
 $\mu(r+s)=\mu(2)\circ(\mu(r)\Smash\mu(s))$.  To see this, consider
 the diagram
 \begin{center}
  \begin{tikzcd}[sep=huge]
   && \CV \\ &
   \CS(r+s)
    \arrow[ur,"\bt(r+s)"]
    \arrow[r,"{\bt(r,s)}"']
    \arrow[d,"{\al(r,s)}"] 
    \arrow[dl,"\al(r+s)"'] &
   \CS(2)
    \arrow[u,"\bt(2)"']
    \arrow[d,"\al(2)"] \\
   \CV^{r+s} &
   \CS(r)\tm\CS(s)
    \arrow[l,"\al(r)\tm\al(s)"]
    \arrow[r,"\bt(r)\tm\bt(s)"'] &
   \CV\tm\CV
  \end{tikzcd}
 \end{center}
 where 
 \begin{align*}
  \al(r,s)(V_1\leq\dotsb\leq V_{r+s}) &= 
   ((V_1\leq\dotsb\leq V_r),
    (V_{r+1}/V_r\leq\dotsb\leq V_{r+s}/V_r)) \\
  \bt(r,s)(V_1\leq\dotsb\leq V_{r+s}) &= (V_r\leq V_{r+s}).
 \end{align*}
 The top triangle and the square commute on the nose, and the left
 triangle commutes up to natural isomorphism.  Moreover, $\bt(r,s)$
 and $\bt(r)\tm\bt(s)$ are coverings, and the square is a pullback (on
 the nose), so it is also a homotopy pullback.  The generalised Mackey
 property~\cite{st:kld}*{Proposition 8.6} therefore gives 
 \[ \bt(r,s)\al(r,s)^!=\al(2)^!(\bt(r)\Smash\bt(s)). \]
 We now compose on the left by $\bt(2)$, and on the right
 by $\al(r)^!\Smash\al(s)^!$.  The two commutative triangles give
 $\al(r,s)^!(\al(r)^!\Smash\al(s)^!)=\al(r+s)^!$ and
 $\bt(2)\bt(r,s)=\bt(r+s)$ so we get 
 \[ \bt(r+s)\al(r+s)^!=
       \mu(2)(\bt(r)\Smash\bt(s))(\al(r)^!\Smash\al(s)^!),
 \]
 or in other words $\mu(r+s)=\mu(2)\circ(\mu(r)\Smash\mu(s))$ as
 claimed.  From this it follows easily that there is an associative
 ring structure with $\mu(0)$ as the unit, $\mu(1)$ as the identity
 map, $\mu=\mu(2)$ as the product, and $\mu(r)$ as the $r$-fold product
 for general $r$.
 
 All that is left is to prove that $\chi$ is an anti-involution for
 $\mu$, or in other words that the following diagram commutes, where
 $\tau$ is the twist map:
 \begin{center}
  \begin{tikzcd}[sep=large]
   L\CV\Smash L\CV \arrow[d,"\mu"'] \arrow[r,"\chi\Smash\chi"] &
   L\CV\Smash L\CV \arrow[r,"\tau"] &
   L\CV\Smash L\CV \arrow[d,"\mu"] \\
   L\CV \arrow[rr,"\chi"'] &&
   L\CV
  \end{tikzcd}
 \end{center}
 To prove this, we define $\chi\:\CS(2)\to\CS(2)$ by
 $\chi(V_1\leq V_2)=(\ann(V_1)\leq V_2^*)$.  On both $\CV$ and $\CS(2)$ we
 have $\chi^2\simeq 1$, so $\chi$ is an equivalence and
 $\chi^!=\chi^{-1}=\chi$.  Similarly, we have
 $\tau^!=\tau^{-1}=\tau$.  As $V_1^*/\ann(V_2)\simeq(V_2/V_1)^*$, we
 see that the following diagram commutes up to natural isomorphism:
 \begin{center}
  \begin{tikzcd}[sep=large]
    \CV\tm\CV \arrow[d,"(\chi\tm\chi)\tau"'] &
    \CS(2) \arrow[l,"\al"'] \arrow[d,"\chi"'] \arrow[r,"\bt"] &
    \CV \arrow[d,"\chi"] \\
    \CV\tm\CV &
    \CS(2) \arrow[l,"\al"] \arrow[r,"\bt"'] & 
    \CV
  \end{tikzcd}
 \end{center}
 The right square gives $\chi\bt=\bt\chi$, and the left square gives
 \[ \chi\al^! = (\al\chi)^! = ((\chi\tm\chi)\tau\al)^! = 
     \al^!\tau(\chi\tm\chi).
 \]
 Combining these gives
 \[ \chi\mu = \chi\bt\al^! = \bt\chi\al^! =
      \bt\al^!\tau(\chi\Smash\chi) = \mu\tau(\chi\Smash\chi)
 \]
 as required.
\end{proof}

\begin{corollary}\lbl{cor-HC-coassoc}
 $L\CV$ has a coassociative coring structure with
 $\nu(r)=\al(r)\bt(r)^!$ as the $r$-fold coproduct.
\end{corollary}
\begin{proof}
 This follows from the proposition by duality.
\end{proof}

\begin{definition}\lbl{defn-HC-product}
 Given $f,g\in E^0(B\CV)$ we define $f*g=\nu^*(f\hot g)$.
 Corollary~\ref{cor-HC-coassoc} tells us that this gives an
 associative and unital product on $E^0(B\CV)$, which we call the
 Harish-Chandra product.  Similarly, $\mu^*$ gives a coassociative and
 counital coproduct, which we call the Harish-Chandra coproduct.
\end{definition}

\begin{remark}\lbl{rem-hom-twist}
 Recall that in Definition~\ref{defn-hom} we defined an element
 $\hom\in E^0(B\CV^2)$ and observed that it is a symmetric
 biexponential function.  We can therefore define a twisted product
 $\tm_{\hom^{-1}}$ as in Definition~\ref{defn-twisted-product}.
\end{remark}

\begin{proposition}\lbl{prop-HC-product}
 The Harish-Chandra product $f*g$ is the same as $f\tm_{\hom^{-1}}g$.
 Moreover, the ring $(E^0(B\CV),*)$ is isomorphic to
 $(E^0(B\CV),\tm)$. 
\end{proposition}
\begin{proof}
 Consider the element $h=(f\ot g).\hom^{-1}\in E^0(B\CV^2)$.
 We can apply Proposition~\ref{prop-twisting-element} to the functors
 $\al$ and $\tau$ to see that $\tau_!(h)=\al^*(h).v$, where
 $v=\tau_!(1)$.  We have seen that $\al^*$ is inverse to $\tau^*$, so
 $v=\al^*\tau^*\tau_!(1)$, and this is the same as $\al^*(\hom)$ by
 Proposition~\ref{prop-hom-aux}.  We therefore have
 $\tau_!(h)=\al^*(h).\al^*(\hom)=\al^*(f\ot g)$.  We now apply $\bt_!$
 to this equation.  As $\bt\tau=\sg$, the left hand side becomes
 $\sg_!((f\ot g).\hom^{-1})=f\tm_{\hom^{-1}}g$, whereas the right hand
 side becomes $f*g$.  This proves the first claim.

 Next, it is clear that if we restrict $\hom$ to the basepoint in
 $B\CV_i\tm B\CV_j$ we get $q^{ij}\in E^0$, which maps to $1$ in
 $K^0$.  It follows that $\hom$ is a strong unit, so we can appeal to
 Corollary~\ref{cor-strong-unit-untwist} to see that
 $(E^0(B\CV),*)\simeq(E^0(B\CV),\tm)$. 
\end{proof}

\begin{proposition}\lbl{prop-HC-HKR}
 For the functors $\al\:\CS(2)\to\CV^2$ and $\bt\:\CS(2)\to\CV$, the
 corresponding maps of $D'_0(\CG)$ are 
 \begin{align*}
   (L\al)[U\leq W] &= [U,W/U] \\
   (L\bt)[U\leq W] &= [W] \\
   (R\al)[U,V] &= |\Hom(V,U)|^{-1}[U\leq U\oplus V]
     = |\Hom(V,U)|^{-1}(L\tau)[U,V] \\
   (R\bt)[V] &= \sum_{U\leq V}[U\leq V].
 \end{align*}
 In more detail:
 \begin{itemize}
  \item[(a)] All the objects $U,V$ and $W$ are
   finite-dimensional $F$-linear representations of $\Tht^*$.
  \item[(b)] The symbol $\Hom(V,U)$ refers to $\Tht^*$-equivariant
   $F$-linear homomorphisms.  
  \item[(c)] In the formula for $R\bt$, the sum is indexed by
   subrepresentations of $V$.
 \end{itemize}
\end{proposition}
\begin{proof}
 The formulae for $L\al$, $L\bt$ and $L\tau$ are immediate.  As $\bt$
 is a covering, the formula for $R\bt$ follows from
 Proposition~\ref{prop-fib-R}.  Next, $(R\al)[U,V]$ is by definition a
 sum over isomorphism classes that are sent by $\al$ to $[U,V]$.  As
 Maschke's theorem applies to $[\Tht^*,\CV]$, the only such
 isomorphism class is $[U\leq U\oplus V]$.  The formula also involves
 a numerical factor, namely $|\Aut(U,V)|/|\Aut(U\leq U\oplus V)|$,
 where the first $\Aut$ is in the category $[\Tht^*,\CV^2]$ and the
 second is in $[\Tht^*,\CS(2)]$.  As we mentioned above, there is a
 split extension 
 \begin{center}
  \begin{tikzcd}
   \Hom(V,U) \arrow[r,rightarrowtail] &
   \Aut(U\leq U\oplus V) \arrow[r,twoheadrightarrow,shift right] &
   \Aut(U)\tm\Aut(V) \arrow[l,shift right]
  \end{tikzcd}
 \end{center}
 and this shows that the numerical factor is $|\Hom(V,U)|$.
\end{proof}

\begin{corollary}\lbl{cor-HC-mu}
 The action of the Harish-Chandra product and coproduct on
 $(D')^0(\CV)$ is given by
 \begin{align*}
   (f*g)(W) &= \sum_{U\leq W} f(U)g(W/U)
             = \sum_{W=U\oplus V} |\Hom(V,U)|^{-1}f(U)g(V) \\
   (\mu^*h)(U,V) &= |\Hom(V,U)|^{-1}h(U\oplus V).
 \end{align*}
\end{corollary}
\begin{proof}
 We can dualise the proposition to get descriptions of the effect of
 $L\al$, $L\bt$,$R_\al$ and $R\bt$ on $(D')^0(\CV)$, then just use the
 definitions $\mu=(L\bt)(R\al)$ and $\nu=(L\al)(R\bt)$.  This gives
 the formula for $\mu^*h$, and also the first formula for $(f*g)(W)$.  
 Proposition~\ref{prop-HC-product} gives the second formula for
 $(f*g)(W)$ (or one can deduce it from the first formula by a little
 linear algebra).
\end{proof}

\begin{definition}\lbl{defn-coprod-prod}
 We define a non-symmetric biexponential function $t$ on
 $\spf(E^0(B\CV^2))=(\Div^+(\HH)^\Gm)^2$ by
 \[ t((D_0,D_1),(E_0,E_1)) = \hom(D_0,E_1)\hom(D_1,E_0)^2. \]
 Thus, the character values are
 \[ t((U_0,U_1),(V_0,V_1)) =|\Hom(U_0,V_1)||\Hom(U_1,V_0)|^2 \]
 for $U_0,U_1,V_0,V_1\in\Rep(\Tht^*)$.
\end{definition}

\begin{proposition}\lbl{prop-coprod-prod}
 The map $\sg^*$ gives a ring homomorphism
 $(E^0(B\CV),\tm)\to(E^0(B\CV^2),\tm_t)$
\end{proposition}
\begin{proof}
 We will need various auxiliary groupoids and functors between them.
 \begin{itemize}
  \item[(a)] $\CP_0$ is just $\CV^4$.  A typical object will be
   written as $(U_0,V_0,U_1,V_1)$.
  \item[(b)] $\CP_1$ is the groupoid of quadruples $(W_0,W_1,U,V)$
   where $W_0,W_1\in\CV$, and $U$ and $V$ are subspaces of
   $W_0\oplus W_1$ such that $W_0\oplus W_1$ is the internal direct
   sum of $U$ and $V$.
  \item[(c)] $\CP_2$ is the groupoid of triples $(W_0,W_1,U)$ where
   $U\leq W_0\oplus W_1$.
  \item[(d)] $\CP_3$ is the groupoid of quadruples $(W_0,W_1,U_0,U_1)$
   with $U_0\leq W_0$ and $U_1\leq W_1$.
 \end{itemize}
 We next need functors as shown below.  Not all of the squares will
 commute.
 \begin{center}
 \begin{tikzcd}
  \CV^2 \arrow[r,equal] & 
  \CV^2 \arrow[r,equal] & 
  \CV^2 \arrow[r,equal] & 
  \CV^2 \arrow[r,equal] & 
  \CV^2 \\
  \CP_0 \arrow[u,"\zt_0"] \arrow[d,"\xi_0"'] \arrow[r,"\phi_0"] &
  \CP_1 \arrow[u,"\zt_1"] \arrow[d,"\xi_1"'] \arrow[r,"\phi_1"] &
  \CP_2 \arrow[u,"\zt_2"] \arrow[d,"\xi_2"'] \arrow[r,"\phi_2"] &
  \CP_3 \arrow[u,"\zt_3"] \arrow[d,"\xi_3"'] \arrow[r,"\phi_3"] &
  \CP_0 \arrow[u,"\zt_0"] \arrow[d,"\xi_0"'] \\
  \CV^2 \arrow[r,equal] & 
  \CV^2 \arrow[r,equal] & 
  \CV^2 \arrow[r,equal] & 
  \CV^2 \arrow[r,equal] & 
  \CV^2 \\
   \end{tikzcd}
 \end{center}
 These are defined by 
 \begin{align*}
  \zt_0(U_0,V_0,U_1,V_1)  &= (U_0\oplus V_0,U_1\oplus V_1) & 
  \xi_0(U_0,V_0,U_1,V_1)  &= (U_0\oplus U_1,V_0\oplus V_1) \\
  \zt_1(W_0,W_1,U,V)      &= (W_0,W_1) &
  \xi_1(W_0,W_1,U,V)      &= (U,V) \\
  \zt_2(W_0,W_1,U)        &= (W_0,W_1) &
  \xi_2(W_0,W_1,U)        &= (U,(W_0\oplus W_1)/U) \\
  \zt_3(W_0,W_1,U_0,U_1)  &= (W_0,W_1) &
  \xi_3(W_0,W_1,U_0,U_1)  &= (U_0\oplus U_1,W_0/U_0\oplus W_1/U_1) \\
 \end{align*}
 \begin{align*}
  \phi_0(U_0,V_0,U_1,V_1) &=
   (U_0\oplus V_0,U_1\oplus V_1,U_0\oplus U_1,V_0\oplus V_1) \\
  \phi_1(W_0,W_1,U,V)     &= (W_0,W_1,U) \\
  \phi_2(W_0,W_1,U)       &= (W_0,W_1,U\cap W_0,\pi_{W_1}(U)) \\
  \phi_3(W_0,W_1,U_0,U_1) &= (U_0,W_0/U_0,U_1,W_1/U_1).
 \end{align*}
 We also put $\tht=\phi_3\phi_2\phi_1\:\CP_1\to\CP_0$.  It is easy to
 see that the composite
 $\tht\phi_0=\phi_3\phi_2\phi_1\phi_0\:\CP_0\to\CP_0$ is naturally
 equivalent to the identity. 
 
 Let us say that two parallel functors $\lm_0$ and $\lm_1$ are
 $L$-equivalent if $L\lm_0=L\lm_1$, which implies that
 $R\lm_0=R\lm_1$.  It is straightforward to check that the first three
 squares on the top commute up to natural isomorphism, as do the first
 and fourth squares on the bottom.  The second square on the bottom
 also commutes up to natural isomorphism, for a slightly less obvious
 reason.  Indeed, for $(W_0,W_1,U,V)\in\CP_1$ we have
 $\xi_1(W_0,W_1,U,V)=(U,V)$ and
 $\xi_2\phi_1(W_0,W_1,U,V)=(U,(W_0\oplus W_1)/U)$ but $W_0\oplus W_1$
 is assumed to be the internal direct sum of $U$ and $V$ so the
 evident composite $V\to W_0\oplus W_1\to (W_0\oplus W_1)/U$ gives the
 required natural isomorphism.

 We claim that the remaining two squares commute up to
 $L$-equivalence.  The key point is that the functors
 $\bt\:(U\leq W)\mapsto W$ and $\rho\:(U\leq W)\mapsto U\oplus W/U$
 are $L$-equivalent, as we saw in Proposition~\ref{prop-hom-aux}.  The
 functors $\zt_3$ and $\zt_0\phi_3$ are essentially $\bt^2$ and
 $\rho^2$, so the top right square commutes up to $L$-equivalence.
 Similarly, the functors $\xi_2$ and $\xi_3\phi_2$ are, up to
 isomorphism, the composites of $\bt^2$ and $\rho^2$ with the functor
 $\CP_2\to\CS^2$ given by
 \[ (W_0,W_1,U)\mapsto
  \left((U\cap W_0\leq U),
   \frac{U+W_0}{U}\leq \frac{W_0\oplus W_1}{U}\right).
 \]

 Now consider the diagram
 \begin{center}
  \begin{tikzcd}
   \CP_0
    \arrow[rr,"\zt_0"]
    \arrow[dd,"\xi_0"']
    \arrow[dr,"\phi_0"] &&
   \CV^2
    \arrow[dd,"\sg"] \\ &
   \CP_1
    \arrow[ur,"\zt_1"']
    \arrow[dl,"\xi_1"] \\
   \CV^2
    \arrow[rr,"\sg"'] &&
   \CV
  \end{tikzcd}
 \end{center}
 One can check that this commutes up to natural isomorphism, and that
 the bottom right region is a homotopy pullback, as in
 Proposition~\ref{prop-twisted-hom}.  However, we will not apply that
 Proposition directly, but will give a give a different argument of a
 similar kind.

 Suppose we have $f,g\in E^0(B\CV)$ and put $h=f\ot g$.  We then have
 $\sg^*(f\tm g)=\sg^*\sg_!(h)$.  The pullback property of the
 above diagram implies that $\sg^*\sg_!(h)=(\xi_1)_!\zt_1^*(h)$.  Because
 the previous diagram commutes up to $L$-equivalence, this is the same
 as $(\xi_0)_!\tht_!\tht^*\zt_0^*(h)$.  Here $\xi_0$ is just the
 direct sum functor for the category $\CV^2$, whereas $\zt_0$ is just
 $\sg^2$, so $\zt_0^*(h)=\sg^*(f)\ot\sg^*(g)$.  Moreover, we have a
 Frobenius reciprocity formula $\tht_!\tht^*(k)=s\bl k$ for all
 $k\in E^0(B\CP_0)$, where $s=\tht_!(1)$.  We therefore conclude that
 $\sg^*\sg_!(f\ot g)=\sg^*(f)\tm_s\sg^*(g)$, and it will suffice to
 check that $s$ is the same as $t$.  We can do this in generalised
 character theory.

 Firstly, $\phi_1$ is a covering.  Given
 $(W_0,W_1,U)\in[\Tht^*,\CP_2]$ we can always choose a
 subrepresentation $V_0$ complementary to $U$ in $W_0\oplus W_1$, and
 then any other complement is the graph of a homomorphism $V_0\to U$.
 The number of complements is thus equal to
 \begin{gather*}
  |\Hom(V_0,U)|=|\Hom(U,V_0)|=|\Hom(U,(W_0\oplus W_1)/U)| = \\
  |\Hom(U,W_0\oplus W_1)||\Hom(U,U)|^{-1}.
 \end{gather*}
 We therefore have
 \[ (\phi_1)_!(1)(W_0,W_1,U) = |\Hom(U,W)||\Hom(U,U)|^{-1}, \]
 where $W=W_0\oplus W_1$.
 
 Next, $\phi_2$ is also a covering.  If $(W_0,W_1,U)$ is in the
 preimage of $(W_0,W_1,U_0,U_1)$ then $U/U_0$ is the graph of a
 homomorphism $U_1\to W_0/U_0$, and this construction is bijective, so
 the number of preimages of $(W_0,W_1,U_0,U_1)$ is
 $|\Hom(U_1,W_0/U_0)|$.  If $(W_0,W_1,U)$ is one such preimage, then
 $U$ is isomorphic to $U_0\oplus U_1$, so the value of
 $(\phi_1)_!(W_0,W_1,U)$ is independent of the choice of preimage.
 Putting this together, we see that
 $(\phi_2\phi_1)_!(1)(W_0,W_1,U_0,U_1)$ is equal to
 \[ |\Hom(U,W)||\Hom(U,U)|^{-1}|\Hom(U_1,W_0/U_0)|, \]
 where $U=U_0\oplus U_1$ and $W=W_0\oplus W_1$.
 
 Next, $\phi_3$ is not a covering, but it is not hard to understand
 $(\phi_3)_!$ anyway.  Recall the general framework: for a functor
 $\al\:\CA\to\CB$, an element $f\in(D')^0(\CA)$ and an object
 $b\in[\Tht^*,\CB]$ we have
 \[ (\al_!f)(b) = \sum_{\al[a]=[b]}\frac{|\CB(b)|}{|\CA(a)|}f(a). \]
 Consider an object $P_0=(U_0,V_0,U_1,V_1)\in[\Tht^*,\CP_0]$, and put
 $W_i=U_i\oplus V_i$ and $P_3=(W_0,W_1,U_0,U_1)\in[\Tht^*,\CP_3]$.  We
 find that $[P_3]$ is the only isomorphism class mapping to $[P_0]$,
 and that
 $|\Aut(P_0)|/|\Aut(P_3)|=|\Hom(U_0,V_0)|^{-1}|\Hom(U_1,V_1)|^{-1}$. 
 This gives
 \[ \tht_!(1)(P_0) = \frac{|\Hom(U,W)||\Hom(U_1,V_0)|}{
    |\Hom(U,U)||\Hom(U_0,V_0)||\Hom(U_1,V_1)|},
 \]
 By expanding this out and cancelling in the obvious way, we get
 \[ \tht_!(1)(P_0) = |\Hom(U_0,V_1)||\Hom(U_1,V_0)|^2
      = t(P_0),
 \]
 as required.    
\end{proof}

\begin{remark}\lbl{rem-coprod-prod-general}
 Proposition~\ref{prop-coprod-prod} should ideally be be embedded in a
 larger context.  For any finite set $X$, let $\CV[X]$ denote the
 category of bundles of finite-dimensional $F$-vector spaces over
 $X$.  Any map $p\:X\to Y$ gives a functor $\sg_p\:\CV[X]\to\CV[Y]$ by
 $\sg_p(V)_y=\bigoplus_{p(x)=y}V_x$.  This in turn gives maps
 $\sg_p^*\:E^0(B\CV[Y])\to E^0(B\CV[X])$ and
 $\sg_{p!}\:E^0(B\CV[X])\to E^0(B\CV[Y])$.  Next, for any map
 $m\:X^2\to\Z$ we define $h_m\in E^0(B\CV[X])$ by
 $h_m(D)=\prod_{x,y}\hom(D_x,D_y)^{m(x,y)}$.  We then define
 $\mu_m\:E^0(B\CV[X])\to E^0(B\CV[X])$ to be multiplication by $h_m$.
 Maps of the form $\sg_p^*$ can be composed in an obvious way, as can
 maps of the form $\sg_{p!}$.  One can also check that
 $\mu_n\circ\sg_{p!}=\sg_{p!}\circ\mu_{p^*(n)}$ and
 $\sg_p^*\circ\mu_n=\mu_{p^*(n)}\circ\sg_p^*$.  Now suppose we have
 maps $X\xra{p}Y\xla{q}Z$, and we want to understand
 $\sg_q^*\circ\sg_{p!}$.  We can form a pullback square
 \begin{center}
  \begin{tikzcd}
   W \arrow[r,"i"] \arrow[d,"j"'] &
   X \arrow[d,"p"] \\
   Z \arrow[r,"q"'] &
   Y.
   \end{tikzcd}
 \end{center}
 We conjecture that there exists $m\:W^2\to\Z$ such that
 \[ \sg_q^*\sg_{p!}=\sg_{j!}\mu_m\sg_i^* \:
     E^0(B\CV[X]) \to E^0(B\CV[Z]).
 \]
 This $m$ will certainly not be unique, because of the fact that
 $\hom(D_0,D_1)=\hom(D_1,D_0)$.  One might hope that this was the only
 source of nonuniqueness, but one can find counterexamples to that.
 We expect that there should be some straightforward combinatorial
 characterisation of the set of maps $m$ for which this works, but so
 far we have not found one.  Assuming that this works, we find that
 the class of maps of the form $\sg_{f!}\mu_m\sg_g^*$ is closed under
 composition. 
\end{remark}

\section{Groupoids of line bundles}
\lbl{sec-line-bundles}

In this section, we will compare $E^0(B\CV)$ with $E^0(B\CXL)$, where
$\CXL$ is the groupoid of finite sets equipped with an $F$-linear line
bundle.  We also consider the groupoid $\CX$ of finite sets.  All of
these groupoids have evident finite type gradings, and all our
algebraic constructions will again be interpreted in the graded
category. 

We proved in~\cite{st:msg} that $E^0(B\CX)$ is a polynomial ring, and
that $\spf(\Ind(E^0(B\CX)))$ is the moduli scheme $\Sub(\GG)$ of
finite subgroups of the formal group $\GG=\spf(E^0(BS^1))$.  Recall that
we have chosen an embedding $\bF^\tm\to\mu_{\infty}(\C)$, which gives
an isomorphism $\HH\simeq\GG$.  This is unique up to the action of
$\Aut(\mu_\infty(\C))=\widehat{\Z}^\tm$, and every subgroup scheme is
necessarily preserved by that action.  We thus have a canonical
identification of $\spf(\Ind(E^0(B\CX)))$ with $\Sub(\HH)$.  We have a
functor $X\mapsto F[X]$ from $\CX$ to $\CV$, and the induced map
\[ \Sub(\HH) \to \spf(E^0(B\CX)) \to
     \spf(E^0(B\CV)) = \Div^+(\HH)^\Gm
\]
is just the evident inclusion.  We have an analogous statement for
$\CXL$ as follows:

\begin{definition}\lbl{defn-coset}
 We define $\Coset(\HH)$ to be the moduli scheme of pairs $(A,a)$
 where $A$ is a finite subgroup scheme of $\HH$, and $a\in\HH/A$.  The
 group $\Gm$ acts on this by $\phi.(A,a)=(A,qa)$, so
 \[ \Coset(\HH)^\Gm = \{(A,a)\st (q-1)a=0\}
                    = \{(A,a)\st p^ra=0\}.
 \]
\end{definition}

\begin{proposition}\lbl{prop-E-CXL}
 $E^0(B\CXL)$ is a polynomial ring, with
 $\spf(\Ind(E^0(B\CXL)))=\Coset(\HH)^\Gm$. 
\end{proposition}
\begin{proof}
 This relies on the theory of power operations in Morava $E$-theory.
 The paper~\cite{re:mic} by Rezk provides a good account of this.  

 The projection $\CXL\to\CX$ gives a map
 $\Ind(E^0(\CX))\to\Ind(E^0(\CXL))$ and thus a map
 $\spf(\Ind_d(E^0(\CXL)))\to\Sub_d(\HH)$.  (We can assume that $d$ is
 a power of $p$, because in other cases $\Sub_d(\HH)$ is empty; but it
 is notationally convenient to write $d$ rather than $p^k$.)

 The $E_\infty$ structure of the spectrum $E$ provides a power
 operation $P\:E^0(X)\to E^0(D_d(X))$ for any space $X$, where
 $D_d(X)=E\Sg_d\tm_{\Sg_d}X^d$ is the $d$'th extended power.  This is
 multiplicative but not additive.  However, there is a transfer ideal
 $T\leq E^0(D_d(X))$ such that the composite
 $\ov{P}\:E^0(X)\to E^0(D_d(X))/T$ is a ring homomorphism.  In
 particular, we can take $X$ to be a point, so $D_d(X)=B\Sg_d$, so
 $\spf(E^0(B\Sg_d)/T)=\Sub_d(\HH)$.  Put $S=\spf(E^0)$.  There is an
 obvious projection $\pi_0\:\Sub_d(\HH)\to S$, and there is also another
 map $\pi_1=\spf(\ov{P})$ in the same direction.  The key property of
 $\pi_1$ is that there is a canonical isomorphism
 $\pi_0^*(\HH)/A\to\pi_1^*(\HH)$, where $A$ is the tautological
 subgroup of $\pi_0^*(\HH)$ given by the universal property of
 $\Sub_d(\HH)$.  (This is proved by taking $X=\CPi$ in the above
 construction.) 

 We can now take $X=B\CL=BGL_1(F)$.  Standard methods then identify
 $D_d(X)$ with $B(GL_1(F)\wr\Sg_d)$ or with $B\CXL_d$.  The relevant
 ideal $T$ in $E^0(B\CXL_d)$ is just the ideal of decomposables, so we
 get a map $\ov{P}\:E^0(B\CL)\to\Ind_d(E^0(B\CXL))$.  By naturality,
 the projection $B\CL\to 1$ gives a commutative diagram of rings as
 shown on the left below, and thus a commutative diagram of formal
 schemes as shown on the right.
 \begin{center}
  \begin{tikzcd}
   \Ind_d(E^0(B\CXL)) &
   E^0(B\CL) \arrow[l,"\ov{P}"'] \\
   \Ind_d(E^0(B\CX)) \arrow[u] &
   E^0 \arrow[l,"\ov{P}"] \arrow[u] &
  \end{tikzcd}
  \begin{tikzcd}
   \spf(\Ind_d(E^0(B\CXL))) \arrow[r] \arrow[d] & 
   \HH[q-1] \arrow[d] \\
   \Sub(\HH) \arrow[r,"\pi_1"'] \arrow[r] &
   S
  \end{tikzcd}
 \end{center}
 The bottom map is $\pi_1$, which classifies the quotient $\HH/A$ as
 explained above, so the pullback in the right hand square is
 $\Coset_d(\HH)^\Gm$.  We therefore have a comparison map 
 \[ \spf(\Ind_d(E^0(B\CXL))) \to \Coset_d(\HH)^\Gm. \]

 The general theory of power operations shows that this is an
 isomorphism.  To make contact with the account by Rezk, we need to
 work with $E^\vee_0(B\CXL)$, but that is isomorphic to $E^0(B\CXL)$
 by duality theory, as we have discussed previously.  As $B\CXL$ is
 the total extended power of $B\CL$, Rezk shows that $E^\vee_0(B\CXL)$
 is obtained from $E_0(B\CL)$ by applying a functor $\TT$.  As
 $E^\vee_*(B\CL)$ is a free module over $E_*$ we are in the simplest
 case of the theory, where $\TT$ always produces polynomial algebras.
 We also have $\TT(M_*\oplus N_*)=\TT(M_*)\ot\TT(N_*)$, so
 $\Ind(\TT(M_*))$ is an additive functor of $M_*$.  Multiplication by
 an element $a\in E_0$ gives an endomorphism of $M_*$ and thus an
 endomorphism of $\Ind_d(\TT(M_*))$, and a naturality argument shows
 that this is multiplication by $\ov{P}(a)$.  Using this we find that
 $\Ind_d(\TT(M_*))=\Ind_d(\TT(E_*))\ot_{E_*}M_*$.  Here
 $\Ind_d(\TT(E_*))$ is the dual of $\CO_{\Sub_d(\HH)}$, and the tensor
 product is formed using $\ov{P}$.  After dualising this and unwinding
 we find that $E^0(D_d(B\CL))/T=\CO_{\Sub_d(\HH)}\ot_{E^0}E^0(B\CL)$,
 where again the tensor product is formed using $\ov{P}$.  Our main
 claim follows from this.
\end{proof}

\appendix

\section{The Atiyah-Hirzebruch spectral sequence for a cyclic group}
\lbl{apx-ahss}

Put $C=\{z\in\C\st z^{p^k}=1\}$.  There is a unique homomorphism
$C\to\Fp$ sending $e^{2\pi i/p^k}$ to $1$, and we write $a^H$ for the
corresponding class in $H^1(BC)$.  The inclusion $C\to S^1$ gives
a line bundle over $BC$ with Euler classes $x^H\in H^2(BC)$ and
$x^K\in K^0(BC)$.  We therefore have an AHSS
\[ H^*(BC;K^*) = P[x]\ot E[a]\ot K^* \convto K^*(BC) = 
    K^*\psb{x^K}/[p^k](x^K). 
\]
Here $x^K$ is represented by $ux^H$ in the AHSS.  As is well-known and
easy to see, the only possible way that the differentials can work is
to have $d_r=0$ for $r\neq 2p^{nk}-1$, and
$d_{2p^{nk}-1}(a)=tu^{-1}(ux)^{p^{nk}}$ for some $t\in\Fp^\tm$.  For
most purposes the value of $t$ is not important.  However, if we do
not know the value, then it creates some overhead of bookkeeping and
notation.  Because of this, the following result is convenient.

\begin{proposition}\lbl{prop-ahss-cyclic}
 In the above AHSS we have $d_{2p^{nk}-1}(a^H)=u^{-1}(ux^H)^{p^{nk}}$
 (so $t=1$). 
\end{proposition}

The proof will be given after some preliminary discussion.
\begin{definition}
 Let $L$ be a complex line bundle over a space $X$, equipped with a
 Hermitian inner product.  Let $S(L)$ be the unit circle bundle in
 $L$, and let $\pi_L\:S(L)\to X$ be the projection.  Let
 $X^L=E(L)\cup\{\infty\}$ be the Thom space, and let $\zt_L\:X\to X^L$
 be the zero section.  This can be identified with the cofibre of
 $\pi_L\:S(L)_+\to X_+$, so we have a connecting map
 $\dl_L\:X^L\to\Sg S(L)_+$.  To identify $X^L$ as a cofibre we
 implicitly need a homeomorphism $f\:[0,1]\to [0,\infty]$, which we
 can take to be $f(t)=t/(1-t)$, so $f^{-1}(t)=t/(1+t)$.  We also take
 $\Sg Y_+$ to be $([0,1]\tm Y)/(\{0,1\}\tm Y)$. The map $\dl_L$ is
 then given by $\dl_L(x,tv)=f^{-1}(t)\Smash(x,v)$ for $(x,v)\in S(L)$
 and $t\in[0,\infty]$.

 We also have a diagonal map $\Dl_L\:E(L)\to X\tm E(L)$ given by
 $\Dl_L(x,v)=(x,(x,v))$, and this has a unique continuous
 extension $\Dl_L\:X^L\to X_+\Smash X^L$.  We write $t(L)$ for the
 Morava $K$-theory Thom class in $\tK^0(X^L)$.  We can use $\Dl_L$ to
 make $\tK^*(X^L)$ into a module over $K^*(X)$, and the Thom
 isomorphism theorem tells us that it is freely generated as such by
 $t(L)$.
\end{definition}

\begin{definition}
 Now fix $m\geq 1$, and let $M$ be the $m$'th tensor power of $L$.
 Define $\psi\:S(L)\to S(M)$ by $\psi(x,v)=(x,v^{\ot m})$.  We also
 extend this to give a map $\psi\:X^L\to X^M$ by
 $\psi(x,tv)=(x,t\,v^{\ot m})$ for $(x,v)\in S(L)$ and
 $t\in[0,\infty]$. 
\end{definition}

\begin{remark}\lbl{rem-psi-diagram}
 It is easy to see that the following diagrams commute:
 \begin{center}
  \begin{tikzcd}[sep=large]
   S(L)_+ \arrow[r,"\pi_L"] \arrow[d,"\psi"'] &
   X_+    \arrow[r,"\zt_L"] \arrow[d,equal]   &
   S^L    \arrow[d,"\psi"]  \arrow[r,"\dl_L"] & 
   \Sg S(L)_+ \arrow[d,"\Sg\psi_+"] & 
   X^L    \arrow[d,"\psi"'] \arrow[r,"\Dl_L"] &
   X_+\Smash X^L \arrow[d,"1\Smash\psi"] \\
   S(M)                     \arrow[r,"\pi_M"'] &
   X_+                      \arrow[r,"\zt_M"'] &
   S^M                      \arrow[r,"\dl_M"'] &
   \Sg S(M)_+ &
   X^M \arrow[r,"\Dl_M"'] &
   X_+\Smash X^M.
  \end{tikzcd}
 \end{center}
\end{remark}

\begin{lemma}\lbl{lem-psi-tM}
 The map 
 \[ \psi^* \: K^*(X).t(M) = \tK^*(X^M) \to 
     \tK^*(X^L) = K^*(X).t(L)
 \]
 is given by 
 \[ \psi^*(a.t(M))= a.\ip{m}(e(L)).t(L) \]
 for all $a\in K^*(X)$.
\end{lemma}
\begin{proof}
 The right hand diagram in Remark~\ref{rem-psi-diagram} shows that
 $\psi^*$ is $K^*(X)$-linear, so we need only consider
 $\psi^*(t(M))$.  This must have the form $c\,t(L)$ for some unique
 element $a\in K^*(X)$.  As $\psi\zt_L=\zt_M$ it follows that 
 \[ \zt_M^*(t(M)) = \zt_L^*(c\,t(L)) = c\,e(L). \]
 On the other hand, we also have 
 \[ \zt_M^*(t(M)) = e(M) = [m](e(L)) = \ip{m}(e(L)).e(L), \]
 so $(c - \ip{m}(e(L)))\,e(L)=0$.  This does not immediately complete
 the proof because $e(L)$ may be a zero-divisor.  However, all our
 constructions are natural, so it will suffice to prove that
 $c=\ip{m}(e(L))$ in the universal case of the tautological bundle
 over $\C P^\infty$.  Here $K^*(X)=K^*\psb{e(L)}$ and so $e(L)$ is a
 regular element and we get $c=\ip{m}(e(L))$ as required.
\end{proof}

We now specialise to the case where $X=\C P^{r-1}=P(\C^r)$ for some
$r\leq\infty$, and $L$ is the tautological bundle.  Then the map
$v\mapsto(\C v,v)$ identifies the space $S^{2r-1}=S(\C^r)$ with
$S(L)$.  Similarly, the map $C_mv\mapsto(\C v,v^{\ot m})$ identifies
the space $S(\C^r)/C_m$ with $S(M)$.  From this point of view, the map
$\psi\:S(L)\to S(M)$ is just the projection $S(\C^r)\to S(\C^r)/C_m$.

\begin{lemma}\lbl{lem-taut-thom}
 The space $P(\C^r)^L$ can be identified with $P(\C^{r+1})$, in such a
 way that the zero section $\zt_L$ becomes the obvious inclusion
 $P(\C^r)\to P(\C^{r+1})$.
\end{lemma}
\begin{proof}
 Define $f\:S(\C^r)\tm\C\to E(L)$ by $f(v,z)=(S^1v,zv)$.  If we let
 $S^1$ act on $S(\C^r)\tm\C$ by $u.(v,z)=(uv,\ov{u}z)$ then this gives
 a homeomorphism $(S(\C^r)\tm\C)/S^1\to E(L)$.  Now define
 $g\:S(\C^r)\tm\C\to S(\C^{r+1})$ by 
 \[ g(v,z) = (v,\ov{z})/\sqrt{1+z\ov{z}}. \]
 This is $S^1$-equivariant and so induces a map
 $\ov{g}\:E(L)\to\C P^r$.  This is easily seen to be a homeomorphism
 from $E(L)$ to the complement of a single point in $\C P^r$, so we
 can pass to the one-point compactification to get a homeomorphism
 $(\C P^{r-1})^L\to\C P^r$ as claimed.
\end{proof}

\begin{lemma}\lbl{lem-BC-attach}
 The space $BC_m$ can be identified with $S(\C^\infty)/C_m$.  This has
 a CW structure where the $(2r-1)$-skeleton is $S(\C^r)/C_m$, and the
 $(2r)$-skeleton is the image of $S(\C^r\oplus\R)$ in
 $S(\C^{r+1})/C_m$.  Moreover, the $(2r)$-skeleton is also the cofibre
 of the map $\psi_+\:S(\C^r)_+\to (S(\C^r)/C_m)_+$.
\end{lemma}
\begin{proof}
 It is standard that $S(\C^\infty)$ is contractible and that the
 action of $C_m$ is free so that $S(\C^\infty)/C_m$ is a model for the
 homotopy type $BC_m$.  Now define $\phi\:[0,1]\tm S(\C^r)\to
 S(\C^{r+1})/C$ by 
 \[ \phi(t,v) = C.(\cos(\pi t/2)v,\sin(\pi t)/2). \]
 The image is the $2r$-skeleton of $BC_m$.  Moreover, we have
 $\phi(0,v)=\psi(v)$ and $\phi(1,v)=C.(0,1)$ for all $v$, but
 otherwise $\phi$ is injective.  This allows us to identify the
 $(2r)$-skeleton with the cofibre of $\psi_+$.  
\end{proof}

Recall that the differentials in the AHSS for $BC_m$ are determined by
the skeleta and the attaching maps.  The above lemma identifies the
quotient $\text{skel}^{2r}(BC_m)/\text{skel}^{2r-1}(BC_m)$ with 
$\Sg S(\C^r)$, which gives
\[ E_1^{2r,j-2r} = \tK^j(\Sg S(\C^r)) = \tK^{j-1}(S(\C^r)). \]

\begin{corollary}\lbl{cor-BC-attach}
 Suppose that $a\in\tK^1(S(\C^r)/C_m)$, and let $a_1$ be the
 restriction of $a$ in the group 
 \[ \tK^1(S(\C)/C_m)=\tK^1(\text{skel}^1(BC_m))= E_1^{1,0}. \]
 Note that 
 \[ \psi^*(a)\in\tK^1(S(\C^r)) = E_1^{2r,2-2r}. \]
 Then $a_1$ and $\psi^*(a)$ survive to the page $E_{2r-1}^{**}$ where
 we have $d_{2r-1}(a_1)=\psi^*(a)$.  
\end{corollary}
\begin{proof}
 Just unwind the definitions.
\end{proof}

We now specialise further to the case where $m=p^k$ and $r=p^{nk}$.
The cofibration 
\[ (S(\C^r)/C_m)_+ \xra{\pi_M} \C P^{r-1}_+ \xra{\zt_M}
    (\C P^{r-1})^M \xra{\dl_M} \Sg (S(\C^r)/C_m)_+
\]
gives an exact sequence relating $K^*(S(\C^r)/C_m)$ to the kernel and
cokernel of the map 
\[ \zt_M^* \: \tK^*((\C P^{r-1})^M) = 
    (K^*[x_K]/x_K^r).t(M) \to K^*[x_K]/x_K^r = K^*(\C P^{r-1}).
\]
Here $\zt_M^*(t(M))=e(M)=[p^k](x_K)=x^r=0$.  As $K^*(\C P^{r-1})$ is
in even degrees there are no extension problems, so there is a unique
element $a_K\in K^1(S(\C^r)/C_m)$ such that 
$\dl_M^*(\Sg a)=u^{-1}t(M)\in K^2((\C P^{r-1})^M)$.  We then find
that $K^*(S(\C^r)/C_m)$ is freely generated by $\{1,a_K\}$ as a module
over $K^*(\C P^{r-1})$.  We can perform the same analysis with $r=1$
to see that $(a_K)_1$ corresponds to $a_H\in H^1(BC_m;K^0)$.  We now
want to understand $\psi^*(a_K)\in\tK^1(S(\C^r))$.  We have seen that
there is a commutative diagram 
\begin{center}
 \begin{tikzcd}
  \C P^r = (\C P^{r-1})^L \arrow[rr,"\dl_L"] \arrow[d,"\psi"'] &&
  \Sg S(\C^r)_+ \arrow[d,"\Sg\psi_+"] \\
  (\C P^{r-1})^M \arrow[rr,"\dl_M"'] &&
  \Sg S(\C^r)/C_m.
 \end{tikzcd}
\end{center}
We have $\dl_M^*(a_K)=u^{-1}t(M)$ by the definition of $a_K$, and 
\[ \psi^*(t(M))=\ip{p^k}(x_K)e(L)=[p^k](x_K)=x_K^r, \]
so $\dl_L^*(\Sg\psi_+)^*(a_K)=u^{-1}x_K^r$.  On the other hand,
$\dl_L$ just pinches off the top cell of $\C P^r$ and so induces an
isomorphism $\tK^i(\Sg S(\C^r))\simeq H^{2r}(\C P^r;K^{i-2r})$.  This
identifies $x_K^r$ with $(ux_H)^r$, so we see that
$d_{2r-1}(a_H)=u^{-1}(ux_H)^r$.  This proves
Proposition~\ref{prop-ahss-cyclic}.

\section{Index of notation}
\label{apx-index}

\begin{itemize}
 \item $\tm$ (Convolution product): Definition~\ref{defn-various-products}
 \item $*$ (Harish-Chandra convolution product): Definition~\ref{defn-HC-product}
 \item $\al$ (Spectral sequence operator): Lemma~\ref{lem-divided-diff}
 \item $\al_k\:S[k]_{\text{even}}\to S[k+1]_{\text{odd}}$: Definition~\ref{defn-S-al}
 \item $\gm(V)$ (Galois twist of $V$): Definition~\ref{defn-galois-twist}
 \item $\Gm=\Gal(\bF/F)\simeq\widehat{Z}$: Definition~\ref{defn-F-bar}
 \item $\dl_k$ (Differential on $S[k]$): Definition~\ref{defn-Sk-differential}
 \item $\Tht=(\Z/p^\infty)^n$, $\Tht^*=\Z_p^n$: Definition~\ref{defn-HKR-D}
 \item $[\Tht^*,\CG]$ (Groupoid of functors): Definition~\ref{defn-functor-groupoid}
 \item $\tht_k\:S[k+1]\to H(S[k];\dl_k)$: Proposition~\ref{prop-Sk-homology}
 \item $\pi_0(\CG)$ (Set of isomorphism classes in a groupoid): Definition~\ref{defn-gpd-misc}
 \item $\Phi=\Tht^\#\simeq(\Z/p^\infty)^n$, $\Phi[m]<\Phi$: Definition~\ref{defn-Phi}
 \item $\phi$ (Frobenius automorphism): Definition~\ref{defn-F-bar}
 \item $\phi$ (Spectral sequence operator): Lemma~\ref{lem-divided-diff}
 \item $\psi^k$ (Adams operation): Definition~\ref{defn-adams-op}
 \item $\mu_{q'}(\C)$, $\mu_{p^\infty}(\C)$ (Groups of roots of unity): Proposition~\ref{prop-all-units}
 \item $\mu$, $\nu$ (Spectrum-level Harish-Chandra (co)product): Definition~\ref{defn-HC-mu}
 \item $\rho\:C=GL_1(F(p))\to GL_p(F)=G$: Used temporarily in Section~\ref{sec-ordinary}
 \item $\sg(V)=V\oplus F$: Proposition~\ref{prop-fix-transfer}
 \item $\sg\:\CV_*\tm\CV_*\to\CV_*$: Definition~\ref{defn-various-products}
 \item $\xi_d$ (Universal Brauer character): Definition~\ref{defn-brauer-character}
 \item $A^*=\Hom(A,\Z/p^\infty)$, $A^\#=\Hom(A,\mu_{p^\infty}(\bF))$: Definition~\ref{defn-duals}
 \item $a\in H^1(BGL_1(F))$: Definition~\ref{defn-aH}
 \item $\ann(X)$ (Annihilator): Definition~\ref{defn-ann}
 \item $b_k=b_k^H\in H_{2k}(BGL_1(F))$: Definition~\ref{defn-xH}
 \item $b^{(n)}(s)\in H_*(B\CV(n))\psb{s}$: Definition~\ref{defn-be-series}
 \item $b_i=b_{Ei}\in E_0(BGL_1(\bF))$, $b_i=b_{Ki}\in K_0(BGL_1(\bF))$: Definition~\ref{defn-bE}
 \item $b_k(s)\in S[k]\psb{s}$: Definition~\ref{defn-be-series-alt}
 \item $b_{mi}\in S[k]$: Definition~\ref{defn-Sk}, Remark~\ref{rem-bmi-extended}
 \item $C=GL_1(F(p))$: Used temporarily in Section~\ref{sec-ordinary}
 \item $C=GL_1(F(k))$ (Used temporarily in Section~\ref{sec-indec}): Definition~\ref{defn-Q}
 \item $\CC_k$ (Equivalent to a groupoid of field extensions): Definition~\ref{defn-field-cat}
 \item $\Coset(\HH)$ (Moduli scheme of cosets): Definition~\ref{defn-coset}
 \item $c_k=c_k^H\in H^{2k}(BGL_d(F))$: Definition~\ref{defn-H-gens-T}, Proposition~\ref{prop-HG}
 \item $c_j^{(k)}\in H^{2j}(BGL_d(F(k)))$: Remark~\ref{rem-intermediate-gens}
 \item $c_k=c_{Ek}\in E^0(BU(d))$ Remark~\ref{rem-E-BUd}
 \item $c_k=c_{Ek}\in E^0(BGL_1(\bF)^d)$, $c_k=c_{Kk}\in K^0(BGL_1(\bF)^d)$: Definition~\ref{defn-E-gens-T}
 \item $D$, $D'$ (Integral and rational HKR coefficient rings): Definition~\ref{defn-HKR-D}
 \item $(D')^0\CG$, $D'_0\CG$ (Generalised character ring and its dual): Definition~\ref{defn-HKR-functors}
 \item $\Dec_d(E^0(B\CV_*))$ (Module of $\tm$-decomposables): Definition~\ref{defn-Ind-Prim}
 \item $\Div_d^+(\HH)$ (Moduli scheme of effective divisors): Remark~\ref{rem-Div-scheme}
 \item $E$ (Morava $E$-theory spectrum): Definition~\ref{defn-E}
 \item $EM^*_{***}$ (Spectral sequence defined algebraically): Corollary~\ref{cor-model-ahss}
 \item $ET^*_{***}$ (Atiyah-Hirzebruch spectral sequence for $K_*(B\CV_*)$): Definition~\ref{defn-AHSS-ET}
 \item $\euler$ (Euler class for $\bF$-linear representations): Definition~\ref{defn-euler}
 \item $e^{(n)}(s)H_*(B\CV(n))\psb{s}$: Definition~\ref{defn-be-series}
 \item $e_j^{(k)}\in H_{2j+1}(BGL_d(F(k)))$: Remark~\ref{rem-intermediate-gens}
 \item $e_k\in H_{2k+1}(BGL_1(F))$: Definition~\ref{defn-aH}
 \item $e_k(s)\in S[k]\psb{s}$: Definition~\ref{defn-be-series-alt}
 \item $e_{ki}\in S[k]$: Definition~\ref{defn-Sk}
 \item $F$ (Finite field of order $q=1\pmod{p^r}$): Definition~\ref{defn-F}
 \item $F[m]$, $F(k)$, $F(\infty)$ $\bF$ (Extension fields of $F$): Definition~\ref{defn-F-bar}
 \item $\Fix(V)$ (Fixed subspace of a representation): Definition~\ref{defn-Fix}
 \item $\CF_k$ (A groupoid of field extensions): Definition~\ref{defn-field-cat}
 \item $\fix$ (Cannibalistic class): Proposition~\ref{prop-fix}
 \item $f_W(t)$ (Chern polynomial): Definition~\ref{defn-chern-poly}
 \item $G_d$, $\bG_d$, $GW_d$, $\bGW_d$ ($GL_d$ of various rings): Definition~\ref{defn-groups}
 \item $G=GL_p(F)$: Used temporarily in Section~\ref{sec-ordinary}
 \item $G=GL_{p^k}(F)$ (Used temporarily in Section~\ref{sec-indec}): Definition~\ref{defn-IJ}
 \item $\CG(a)=\Aut_{\CG}(a)$: Definition~\ref{defn-gpd-misc}
 \item $H=GL_{p^{k-1}}(F)^p$ (Used temporarily in Section~\ref{sec-indec}): Definition~\ref{defn-IJ}
 \item $H_*(X)$, $H^*(X)$ ((Co)homology with coefficients $\F_p$): Definition~\ref{defn-H-coeffs}
 \item $h_k(x)$ (Weierstrass polynomial): Proposition~\ref{prop-p-series}
 \item $\hom\:\Div^+(\HH)^2\to\aff^1$: Definition~\ref{defn-hom}
 \item $I=\Prim_{p^k}(E^0(B\CV_*))$, $J=\Dec_{p^k}(E^0(B\CV_*))$ (Used temporarily in Section~\ref{sec-indec}): Definition~\ref{defn-IJ}
 \item $I'=D'\ot_{E^0}I$, $J'=D'\ot_{E^0}J$ (Used temporarily in Section~\ref{sec-indec}): Definition~\ref{defn-IJ-prime}
 \item $\bI=I/\mxi I$, $\bJ=J/\mxi J$ (Used temporarily in Section~\ref{sec-indec}): Definition~\ref{defn-IJ-bar}
 \item $\Ind_d(E^0(B\CV_*))$ (Module of $\tm$-indecomposables): Definition~\ref{defn-Ind-Prim}
 \item $\Irr(\Tht^*)=\Irr(\Tht^*;F)$, $\Irr_d(\Tht^*)$, $\Irr(\Tht^*;\bF)$ (Sets of irreducible representations): Example~\ref{eg-HKR-V}
 \item $i\:GL_1(\bF)\to\mu_{q'}(\C)$: Proposition~\ref{prop-all-units}
 \item $i\:W\bF\to\C$: Proposition~\ref{prop-witt-embedding}
 \item $K$ (Morava $K$-theory spectrum): Definition~\ref{defn-E}
 \item $K_a$ (Kernel of a fibration): Definition~\ref{defn-fib-K}
 \item $L\CG=L_{K(n)}\Sgip B\CG$: Definition~\ref{defn-K-loc-cat}
 \item $\CL$, $\bCL$ (Groupoids of lines over $F$ and $\bF$): Definition~\ref{defn-cats}
 \item $m_k$ (Degree of a useful field extension): Remark~\ref{rem-nested-subfields}
 \item $N_d$, $\bN_d$, $NW_d$, $\bNW_d$ (Groups of monomial matrices): Definition~\ref{defn-groups}
 \item $N_k$, $\ov{N}_k$, $N^*_k$, $\CN_k$: Definition~\ref{defn-CN}
 \item $PS(\CV)(t)$ (Poincar\'e series): Definition~\ref{defn-PSV}
 \item $\Prim_d(E^0(B\CV_*))$ (Module of coalgebra primitives): Definition~\ref{defn-Ind-Prim}
 \item $\tQ=E^0\psb{x}/(\ip{p}([m'](t)))$, $Q=\tQ^\Gm$: Definition~\ref{defn-Q}
 \item $q_0$, $q$ (Characteristic and order of the field $F$): Definition~\ref{defn-F}
 \item $R=E^0(BGL_{p^k}(F))$ (Used temporarily in Section~\ref{sec-indec}): Definition~\ref{defn-IJ}
 \item $R'=D'\ot_{E^0}R$ (Used temporarily in Section~\ref{sec-indec}): Definition~\ref{defn-IJ-prime}
 \item $\bR=R/\mxi R$ (Used temporarily in Section~\ref{sec-indec}): Definition~\ref{defn-IJ-bar}
 \item $\Rep(\Tht^*)=\Rep(\Tht^*;F)$, $\Rep_d(\Tht^*)$, $\Rep(\Tht^*;\bF)$ (Sets of representations): Example~\ref{eg-HKR-V}
 \item $r=v_p(q-1)$: Definition~\ref{defn-F}
 \item $S[k]$, $S[\infty]$: Definition~\ref{defn-Sk}
 \item $\CS$ (Category of $K(n)$-local spectra): Definition~\ref{defn-K-loc-cat}
 \item $\soc_{\CG}$, $\soc_{\CG,\CH}$ (Absolute and relative socle generators): Definition~\ref{defn-soc}
 \item $s_i$ (Integer characterised by $i\in\CN_{s_i}$): Definition~\ref{defn-CN}
 \item $s_k=c_{p^k}\in K^0(B\CV_{p^k})$: Definition~\ref{defn-sk}
 \item $T_d=GL_1(F)^d$, $\bT_d=GL_1(\bF)^d$, $TW_d$, $\bTW_d$ (Groups of diagonal matrices): Definition~\ref{defn-groups}
 \item $T=GL_1(F)^p$: Used temporarily in Section~\ref{sec-ordinary}
 \item $T=\Hom(\Z/p^\infty,\mu_{p^\infty}(\bF))$:  Remark~\ref{rem-duals}
 \item $\CV$, $\CV(k)$, $\bCV$ (Groupoids of vector spaces): Definition~\ref{defn-cats}
 \item $v_k\in H_{2k-1}(BGL_d(F))$: Definition~\ref{defn-H-gens-T}
 \item $v_j^{(k)}\in H_{2j-1}(BGL_d(F(k)))$: Remark~\ref{rem-intermediate-gens}
 \item $W$, $WF$, $WF(m)$, $W\bF$ (Witt functor and Witt rings): Definition~\ref{defn-witt}
 \item $X^\perp$ (Inner product annihilator): Definition~\ref{defn-ann}
 \item $\CX$, $\CXL$, $\bCXL$ (Groupoids of finite sets or line bundles): Definition~\ref{defn-cats}
 \item $x=x_H\in H^2(BGL_1(\bF))$: Definition~\ref{defn-xH}
 \item $x_E\in E^0(BGL_1(\bF))$, $x_K\in K^0(BGL_1(\bF))$: Definition~\ref{defn-xE}
\end{itemize}

\end{document}